\documentclass[a4paper,11pt]{article}
\usepackage{amsmath}
\usepackage{amssymb}
\usepackage{graphicx}
\usepackage{subfigure}
\usepackage[top=2.2cm,bottom=2.2cm,left=3.18cm,right=3.18cm]{geometry}
\usepackage{bm}
\usepackage{multicol}
\usepackage{float}
\usepackage{color}
\usepackage{mathrsfs}
\usepackage{array}
\usepackage[noend]{algpseudocode}
\usepackage{algorithmicx,algorithm}
\usepackage{caption2}
\usepackage{txfonts}
\usepackage[colorlinks,citecolor=red,urlcolor=blue,bookmarks=false,hypertexnames=true]{hyperref}

\newenvironment{proof}{{\noindent\it Proof}\quad}{\hfill $\square$\par}  
\newtheorem{theorem}{Theorem}
\newtheorem{lemma}{Lemma}
\newtheorem{remark}{Remark}
\newtheorem{example}{Example}
\numberwithin{equation}{section}
\numberwithin{remark}{section}
\numberwithin{theorem}{section}
\numberwithin{lemma}{section}
\numberwithin{figure}{section}
\numberwithin{table}{section}
\numberwithin{example}{section}
\newcommand{\bs}[1]{\boldsymbol{#1}}

 \def\Span{\operatorname{span}}
\def\NN{\mathbb{N}}
\def\P{\mathcal{P}}
\def\ZZ{\mathbb{Z}}
\def\RR{\mathbb{R}}
\def\tr{\mathsf{T}}
\def\TT{\mathcal{T}}

\def\JJ{\mathcal{J}}
\def\hx{\hat{x}}

\def\F{\mathbf{F}}

\def\PP{\mathbf{P}}

\def\bss{\mathbf{s}}
\def\bk{\mathbf{k}}
\def\ii{\rm{i}}
\def\bell{\bm{\ell}}
\def\balpha{\bm{\alpha}}

\begin{document}
\title{Sparse Spectral-Galerkin Method on An Arbitrary Tetrahedron Using Generalized Koornwinder Polynomials}

\author{Lueling Jia\thanks{Beijing Computational Science Research Center, Beijing, 100193, China.
	  email: {\tt lljia@csrc.ac.cn}. }
       \and
         Huiyuan Li\thanks{State Key Laboratory of Computer Science/Laboratory of Parallel Computing, Institute of Software, Chinese Academy of Sciences, Beijing 100190, China. 
	  email: {\tt huiyuan@iscas.ac.cn}.}
	\and 
	Zhimin Zhang\thanks{Beijing Computational Science Research Center, Beijing 100193, China; and Department of Mathematics, Wayne State University, Detroit, MI 48202, USA.
	  email: {\tt zmzhang@csrc.ac.cn; ag7761@wayne.edu}.}
	}
\date{}

\maketitle

\begin{abstract}
In this paper, we  propose a sparse spectral-Galerkin approximation scheme for solving the second-order partial differential equations on an arbitrary tetrahedron. Generalized Koornwinder polynomials are introduced on the reference tetrahedron as basis functions with  their various  recurrence relations and differentiation properties being explored.
The method leads to  well-conditioned and sparse linear systems 
whose entries can either be calculated directly by the orthogonality of the generalized Koornwinder polynomials for differential equations with constant coefficients
or be evaluated efficiently via our recurrence algorithm for problems with variable coefficients.  
Clenshaw algorithms for the evaluation of any polynomial in an expansion of the generalized Koornwinder basis are also designed to boost the efficiency of the  method.
Finally, numerical experiments are carried out  to illustrate the effectiveness of  the proposed Koornwinder spectral method.

\end{abstract}



\section{Introduction}
Spectral element methods with unstructured mesh have been widely used in the study of computational fluid dynamics, elastodynamics, resistivity modeling and many other fields due to their ``spectral accuracy"  \cite{Sherwin1996, Karniadakis2005, resistivity2019, 2017elastody}.
 Their virtue of high accuracy also makes spectral (element) methods powerful tools for solving eigenvalue problems as they are able to provide more reliable eigen-solutions than the low order methods such as  finite element methods and finite difference methods \cite{zhangeig2015, Baogaps2020}.
%
Meanwhile, as  simplices are one kind of the basic geometric elements, their use gives flexibility in the discretization of complex domains.
In view of this, spectral methods on simplex elements, especially on tetrahedra in three dimensions, with sparse structures in discrete matrices,  play a fundamental role in designing accurate and efficient numerical schemes in practical applications.

Based on the Galerkin framework, the accuracy and computational effectivity of the numerical scheme depend on the choice of basis functions.
Hierarchical basis functions defined in the barycentric  coordinate system on the tetrahedron  have been proposed and developed in \cite{peano1979, szabo1991, Carnevali1993, Adjerid2001}, which possess good symmetry but lack useful orthogonality. Thus, it requires complicated numerical integration for obtaining linear systems in high order case.
 In the Cartesian coordinate system, although a fully tensorial spectral method using rational basis functions put forward in \cite{LiWang_rational}  has spectral accuracy in approximations and  can be implemented effectively,
 the use of orthogonal basis polynomials is more natural.
Koornwinder polynomials form a family of fully orthogonal polynomials  with respect to a particular weight function on the simplex  \cite{KOORNWINDER1975435} and its simplest family, the $L^2$-orthogonal Koornwinder-Dubiner polynomials have been studied in \cite{Dubiner1991}.
Motivated by generalized Jacobi polynomials \cite{GSW2001, GSW2006, STW2011},
 some progress on numerical schemes and theoretical analysis have been made for generalized Koornwinder polynomials on triangles \cite{LiShen2009, ShanLi2017, Townsend2019}. Indeed, generalized Koornwinder polynomials simplify the design of shape functions in triangular spectral element approximations  with efficient numerical algorithms and well-conditioned sparse  linear systems.
 However, few results are achieved for the extension of generalized Koornwinder polynomials to tetrahedra, although polynomial  basis  functions for  tetrahedral elements have been proposed by  Sherwin and Karniadakis
based on classical Koornwinder polynomials in 
1990's \cite{KS1995, Karniadakis2005}, and  by  Beuchler et al.  based on integrated Jacobi polynomials 
 in  2000's \cite{intlegendre2006, intlegendre2007, intlegendre2008}.

In this paper, we first introduce the  generalized Koornwinder polynomials on a reference tetrahedron and explore their various recurrence relations and differentiation properties.
We then
propose a sparse spectral-Galerkin method for second-order partial differential equations on an arbitrary tetrahedron
by employing generalized Koornwinder polynomials
 to design modal basis functions in simple presentations.
The sparsity that exists in various recurrence relations of generalized Koornwinder polynomials allows us to assemble the discrete matrices efficiently.
 Indeed, 
a generalized Koornwinder polynomial of certain order or its derivatives are equal to a finite  combination of  Koornwinder-Dubiner polynomials.
 For differential equations with constant coefficients, the integrals of two generalized Koornwinder polynomials or their derivatives, which are entries of the stiffness matrix and the mass matrix, can be exactly evaluated by the expansion coefficients  and the $L^2$-orthogonality of Koornwinder-Dubiner polynomials.
For the case of variable coefficients,
 the three-term recurrence relation for generalized Koornwinder polynomials
yields  a recursive assembling of the mass matrix that only requires $\mathcal{O}(M^6)$ operations (for polynomials of total degree $\le M$), instead of the complexity of $\mathcal{O}(M^9)$ by directly using numerical quadrature.
The three-term recurrence relation also admits
an efficient implementation of the Clenshaw algorithm \cite{clenshaw1955} to evaluate the generalized Koornwinder expansions in $\mathcal{O}(M^3)$ operations.
More importantly,  a numerical study reveals that   the sparse linear system  resulted from our spectral-Galerkin method
has a condition number asymptotically in $\mathcal{O}(M^4)$, which is superior to  $\mathcal{O} (M^7)$
for those using classical Koornwinder polynomials  and  $\mathcal{O}(M^{10})$ for those using integrated Jacobi polynomials.
 Hence, our linear system is well-conditioned and can be efficiently solved.


The paper is organized as follows.
In Section \ref{preliminaries}, we formulate definitions and basic properties of generalized Jacobi polynomials and  generalized  Koornwinder polynomials, including their
various recurrence relations and differentiation properties.
In Section \ref{clenshew}, an efficient implementation of the Clenshaw algorithm  for Koornwinder expansions  based on the three-term recurrence relation of generalized  Koornwinder polynomials has been studied. 
The sparse spectral-Galerkin method for second-order partial differential equations on an arbitrary  tetrahedron using generalized Koornwinder polynomials together with its implementation is presented in Section  \ref{spectral-Galerkin}.
We report some illustrative numerical results to confirm the sparsity  as well as exponential orders of convergence of the method in Section \ref{numerical_exper}.
Finally, a conclusion remark is given in Section \ref{conclusion}.

\section{Preliminaries}\label{preliminaries}
Let $\Omega\subset\RR^3$ be a bounded domain and $w$ be a weight function.
Denote by $(\cdot,\cdot)_{w,\Omega}$ and $\|\cdot\|_{w,\Omega}$ the inner product and the norm of $L^2_{w}(\Omega)$, respectively. 
$H_w^s(\Omega)$ and $H^s_{0,w}(\Omega)$ are the usual Sobolev spaces with respect to the weight function $w$. 
 Denote  by $\ZZ$, $\NN$, $\NN_0$  and $\ZZ^-$ the set of integers, positive integers, non-negative integers and negative integers, respectively. 
Further let $I_m\in \RR^{m\times m}$ be the  identity matrix and
$\bm{e}_{n}$ be the unit column vector only with its $n$-th entry being 1.

For any $M\in\NN_0$, let $\P_M(\Omega)$ be the space of polynomials of total degree no greater than $M$ in $\Omega$ and denote
\begin{equation}\label{dm}
d_M:=\dim \P_M(\Omega)= \binom{M+3}{M}.
\end{equation}
$\Omega$  and $w$ (if $w\equiv 1$) could be dropped from the notation when no confusion would arise.

Let $\hat{\TT}$ be the reference tetrahedron defined as
\begin{equation*}\label{reference}
\hat{\TT} := \{ \bm{\hat{x}}=(\hx_1, \hx_2, \hx_3)^\tr: 0\le \hx_1, \hx_2, \hx_3, \hx_1+\hx_2+\hx_3\le1\},
\end{equation*}
with vertices 
$$\hat{P}_0=(0,0,0)^\tr,\quad\hat{P}_1=(1,0,0)^\tr,\quad\hat{P}_2=(0,1,0)^\tr,\quad\hat{P}_3=(0,0,1)^\tr.$$
Moreover, let $\left( \bm{a}\cdot \bm{b}\right)$ and  $\bm{a} \times \bm{b}$ denote the dot product and the cross product of any $\bm{a}, \bm{b}\in\RR^3,$ respectively.
Denote $(\bm{a}, \bm{b}, \bm{c}) = \bm{a} \cdot (\bm{b}\times \bm{c}) = \bm{b} \cdot (\bm{c}\times \bm{a})
= \bm{c} \cdot (\bm{a}\times \bm{b})$ as the triple product of 
 any $\bm{a}, \bm{b}, \bm{c}\in\RR^3$.
For any  $\bell=(\ell_1, \ell_2, \ell_3)\in \NN_0^3$ and 
$\balpha=(\alpha_0,\alpha_1,\alpha_2,\alpha_3)\in\RR^4$, we introduce the following multi-index notation
\begin{align*}
&|\bell| =\ell_1+\ell_2+\ell_3, && |\balpha| = \alpha_0 + \alpha_1 + \alpha_2 + \alpha_3,
\\
&\bell^i = (\ell_1,\cdots,\ell_i), && |\bell^i| = \ell_1+\cdots+\ell_i && 1\le i\le 2, \, i\in\NN,
\\
&\balpha^j = (\alpha_0,\cdots,\alpha_j),&& |\balpha^j| = \alpha_0+\cdots+\alpha_j, && 0\le j\le 2, \, j\in\NN_0. 
\end{align*}


 


\subsection{Generalized Jacobi polynomials}
Let $I=(-1,1)$.
For any $k\in\NN_0$, the classical Jacobi polynomial $J_k^{\alpha,\beta}(z)$  of degree $k$  with $\alpha,\beta>-1$
 has the following representation in hypergeometric series,  
 \begin{equation}
\begin{aligned}
\label{jacobi_hyper}
J_k^{\alpha,\beta}(z) 
=\sum\limits_{j=0}^k\frac{(\alpha+j+1)_{k-j}(k+\alpha+\beta+1)_j}{j!(k-j)!}\left(\frac{z-1}{2}\right)^j,
\end{aligned}
\end{equation}
where $(a)_n=a(a+1)\cdots(a+n-1)$ is the Pochhammer symbol. 
Classical Jacobi polynomials are mutually orthogonal with respect to the Jacobi weight function 
$\varpi^{\alpha,\beta}(z):=(1-z)^\alpha(1+z)^\beta$,
\begin{equation}\label{jacobi_ortho}
({J}_k^{\alpha,\beta},{J}_j^{\alpha,\beta})_{\varpi^{\alpha,\beta},I}= 2^{\alpha+\beta+1} h_{k}^{\alpha,\beta}\delta_{k,j},\quad
h_k^{\alpha,\beta}=\frac{\Gamma(k+\alpha+1)\Gamma(k+\beta+1)}{(2k+\alpha+\beta+1)\Gamma(k+1)\Gamma(k+\alpha+\beta+1)},\;
k,j \in\mathbb{N}_0,
\end{equation}
where $\delta_{k,j}$ is the Kronecker delta.
From the representation \eqref{jacobi_hyper}, the index parameters $\alpha$ and/or $\beta$ of Jacobi polynomials could be extended to any real numbers.
In the case of $\alpha$ and/or $\beta$ being negative integer parameters,
they are exactly generalized Jacobi polynomials attracting much attention in literature
for their applications in scientific computations \cite{GSW2001, GSW2006, STW2011}.
However,  a degree reduction occurs if and only if $-k-\alpha-\beta\in \{1,2,\cdots,k\}$.
In this paper, we are interested in the generalized Jacobi polynomials when $\alpha=-1$ and/or $\beta=-1.$
At first, we directly obtain from \eqref{jacobi_hyper} that
\begin{align}
J_0^{\alpha,-1}(z) = 1,\quad 
J_k^{\alpha,-1}(z) = \frac{k+\alpha}{k} \frac{z+1}{2} J_{k-1}^{\alpha,1}(z),\quad k\ge1, \alpha>-1,\\
J_0^{-1,\beta}(z) = 1,\quad 
J_k^{-1,\beta}(z) = \frac{k+\beta}{k} \frac{z-1}{2} J_{k-1}^{1,\beta}(z),\quad k\ge1, \beta>-1.
\end{align}
Meanwhile, we modify the definition of $J_1^{-1,-1}$ and then obtain the following
complete system:
\begin{equation}
J_0^{-1,-1}(z)=1,\quad J_1^{-1,-1}(z) = z,\quad
J_k^{-1,-1}(z) = \frac{z-1}{2} \frac{z+1}{2} J_{k-2}^{1,1}(z),\quad k\ge 2.
\end{equation}

Some important properties on generalized Jacobi polynomials are derived  from \cite[(3.110)-(3.111)]{STW2011} and \cite[(6.4.20)-(6.4.22)]{AAR1999} with piecewise coefficients. 
We summarize these conclusions in the following lemmas.

\begin{lemma}\label{lemma_three}
For any $k\in\NN_0$ and $\alpha,\beta\ge-1,$
the three-term recurrence relation for $J_k^{\alpha,\beta}(z)$ holds,
\begin{equation}\label{three_1d}
\begin{aligned}
&z { J_k^{\alpha,\beta}}(z) = 
a_{1,k}^{\alpha,\beta} { J_{k+1}^{\alpha,\beta}}(z) + a_{2,k}^{\alpha,\beta}  { J_{k}^{\alpha,\beta}}(z)+ a_{3,k}^{\alpha,\beta}  { J_{k-1}^{\alpha,\beta}}(z),
\end{aligned}
\end{equation}
where
\begin{equation*}
\left(a_{1,k}^{\alpha,\beta}, a_{2,k}^{\alpha,\beta}, a_{3,k}^{\alpha,\beta}\right)=\begin{cases}
\left(1,0,0 \right),  \qquad\qquad\qquad\qquad\qquad\qquad\qquad\qquad\quad   k=0, \alpha=\beta=-1,\\[0.3em]
\left(\frac{2}{\alpha+\beta+2}, \frac{\beta-\alpha}{\alpha+\beta+2}, 0 \right), 
\qquad\qquad\qquad\qquad\qquad\qquad\quad\,\,
k=0, \alpha+\beta\neq-2,\\[0.3em]
\left( 4,0,1\right),  
\qquad\qquad\qquad\qquad\qquad\qquad\qquad\qquad\quad
k=1, \alpha=\beta=-1,\\[0.3em]
\left( \frac{1}{2},0,0\right),
\qquad\qquad\qquad\qquad\qquad\qquad\qquad\qquad\quad
k=2, \alpha=\beta=-1,\\[0.3em]
\left( \frac{2(k+1)(k+\alpha+\beta+1)}{(2k+\alpha+\beta+1)(2k+\alpha+\beta+2)},
\frac{\beta^2-\alpha^2}{(2k+\alpha+\beta)(2k+\alpha+\beta+2)},
\frac{2(k+\alpha)(k+\beta)}{(2k+\alpha+\beta)(2k+\alpha+\beta+1)}
\right),  {\rm{otherwise}}.
\end{cases} 
\end{equation*}

\end{lemma}

\begin{lemma}\label{lemma_increasing}
For any $k\in\NN_0$ and $\alpha,\beta\ge-1,$ the generalized Jacobi polynomials $J_k^{\alpha,\beta}(z)$ satisfy
\begin{align}
&J_k^{\alpha,\beta}(z) = 
 b_{1,k}^{\alpha,\beta}J_k^{\alpha+1,\beta}(z) +  b_{2,k}^{\alpha,\beta}J_{k-1}^{\alpha+1,\beta}(z),\label{LemEQ2(1)}\\
 &J_k^{\alpha,\beta}(z) = 
 b_{1,k}^{\alpha,\beta}J_k^{\alpha,\beta+1}(z) -  b_{2,k}^{\beta,\alpha}J_{k-1}^{\alpha,\beta+1}(z),\label{LemEQ2(2)}\\
&J_k^{\alpha,\beta}(z) = c_{1,k}^{\alpha,\beta}J_k^{\alpha+2,\beta}(z) + c_{2,k}^{\alpha,\beta}J_{k-1}^{\alpha+2,\beta}(z) + c_{3,k}^{\alpha,\beta}J_{k-2}^{\alpha+2,\beta}(z),\label{LemEQ2(3)}
\end{align}
where
\begin{align*}
&\left( b_{1,k}^{\alpha, \beta}, b_{2,k}^{\alpha, \beta}\right) = \begin{cases}
\left(1,0\right), & k=0, \alpha,\beta\ge -1,\\
\left(2,-1\right), & k=1, \alpha=\beta=-1,\\
\left(\frac{k+\alpha+\beta+1}{2k+\alpha+\beta+1},  -\frac{k+\beta}{2k+\alpha+\beta+1}\right),& {\rm{otherwise}},
\end{cases}
\\
&\left(c_{1,k}^{\alpha,\beta},c_{2,k}^{\alpha,\beta},c_{3,k}^{\alpha,\beta}\right)= 
\left( b_{1,k}^{\alpha,\beta}b_{1,k}^{\alpha+1,\beta},
b_{1,k}^{\alpha,\beta}b_{2,k}^{\alpha+1,\beta} + b_{2,k}^{\alpha,\beta} b_{1,k-1}^{\alpha+1,\beta},
  b_{2,k}^{\alpha,\beta} b_{2,k-1}^{\alpha+1,\beta}\right).
\end{align*}
\end{lemma}

\begin{lemma}\label{lemma_decrease}
For any $k\in\NN_0$ and $\alpha,\beta\ge-1,$ the generalized Jacobi polynomials $J_k^{\alpha,\beta}(z)$  satisfy
\begin{align}
&\frac{1-z}{2} J_k^{\alpha+1,\beta}(z)= e_{1,k}^{\alpha,\beta} J_k^{\alpha,\beta}(z) + e_{2,k}^{\alpha,\beta} J_{k+1}^{\alpha,\beta}(z),\label{LemEQ3(1)}\\[0.1em]
&\frac{1+z}{2} J_k^{\alpha,\beta+1}(z)= e_{1,k}^{\beta,\alpha} J_k^{\alpha,\beta}(z) - e_{2,k}^{\alpha,\beta} J_{k+1}^{\alpha,\beta}(z),\label{LemEQ3(2)}\\[0.1em]
&\left(\frac{1-z}{2}\right)^2J_{k}^{\alpha+2,\beta}(z) = 
g_{1,k}^{\alpha,\beta} J_{k+2}^{\alpha,\beta}(z)+ 
g_{2,k}^{\alpha,\beta} J_{k+1}^{\alpha,\beta}(z)+ 
g_{3,k}^{\alpha,\beta} J_{k}^{\alpha,\beta}(z),\label{LemEQ3(3)}
\end{align}
where
\begin{align*}
&\left( e_{1,k}^{\alpha, \beta}, e_{2,k}^{\alpha, \beta}\right) = \begin{cases}
\left(\frac{1}{2}, -\frac{1}{2} \right), & k=0, \alpha=\beta= -1,\\
\left(0,-1\right), & k=1, \alpha=\beta=-1,\\
\left(\frac{k+\alpha+1}{2k+\alpha+\beta+2},  -\frac{k+1}{2k+\alpha+\beta+2}\right),& {\rm{otherwise}},
\end{cases}
\\
&\left(g_{1,k}^{\alpha,\beta},g_{2,k}^{\alpha,\beta},g_{3,k}^{\alpha,\beta}\right) =
\left( e_{2,k}^{\alpha+1,\beta} e_{2,k+1}^{\alpha,\beta},\,
 e_{1,k}^{\alpha+1,\beta} e_{2,k}^{\alpha,\beta} +e_{2,k}^{\alpha+1,\beta}e_{1,k+1}^{\alpha,\beta},\,
e_{1,k}^{\alpha+1,\beta} e_{1,k}^{\alpha,\beta}\right).
\end{align*}
\end{lemma}

\begin{lemma}\label{lemma_differential}
For any $k\in\NN_0$ and $\alpha,\beta\ge-1,$ the generalized Jacobi polynomials $J_k^{\alpha,\beta}(z)$  satisfy
\begin{equation}\label{LemEQ1}
\partial_z J_k^{\alpha,\beta}(z) = d_k^{\alpha,\beta} J_{k-1}^{\alpha+1,\beta+1}(z),
\end{equation}
where 
\begin{equation*}
d_k^{\alpha,\beta} = \begin{cases}
1,\quad & k=1,\alpha=\beta=-1,\\
\frac{k+\alpha+\beta+1}{2},\quad &{\rm{otherwise}}.
\end{cases}
\end{equation*}
\end{lemma}
Hereafter, we use the convention that $J^{\alpha,\beta}_{k}= b_{1,k}^{\alpha,\beta}= b_{2,k}^{\alpha,\beta} \equiv  0$ for $k<0$.

\subsection{Generalized Koornwinder polynomials}

For $\balpha=(\alpha_0,\alpha_1,\alpha_2,\alpha_3) \in [-1,+\infty)^4$, 
the generalized Koornwinder polynomials ${\JJ}_{\bell}^{\balpha}(\hat{\bm x})$, $\bell=(\ell_1,\ell_2,\ell_3) \in \NN_0^3$
on the reference tetrahedron $\hat{\TT}$
can be defined through the generalized Jacobi polynomials and the collapsed coordinate transform from the reference cube  to $\hat{\TT}$   \cite{LiWang_rational} as
\begin{equation}\label{koorn_3d}
\begin{array}{rl}
&{\JJ}_{\bell}^{{\balpha}}(\hat{ \bm x}):={\JJ}_{\ell_1,\ell_2,\ell_3}^{\alpha_0,\alpha_1,\alpha_2,\alpha_3}(\hat{\bm x})
 =(1-\hx_2-\hx_3)^{\ell_1} {J_{\ell_1}^{\alpha_0,\alpha_1}}\left(\dfrac{2\hx_1}{1-\hx_2-\hx_3}-1\right)(1-\hx_3)^{\ell_2}\\[1em]
&\qquad\times {J_{\ell_2}^{2\ell_1+\alpha_0+\alpha_1+1,\alpha_2}}\left(\dfrac{2\hx_2}{1-\hx_3}-1\right) { J_{\ell_3}^{2\ell_1+2\ell_2+\alpha_0+\alpha_1+\alpha_2+2,\alpha_3}}(2\hx_3-1).
\end{array}
\end{equation}
Denote by $\chi(x)= \max(\lfloor -x\rfloor, 0)$ where $\lfloor s\rfloor$ is the integer part of the real number $s$.
The generalized Koornwinder polynomials ${\JJ}_{\bell}^{\balpha}(\hat{\bm x})$
 are fully orthogonal  with respect to the Jacobi weight function $\omega^{\balpha}(\bm{\hx}):=(1-\hx_1-\hx_2-\hx_3)^{\alpha_0} \hx_1^{\alpha_1} \hx_2^{\alpha_2}\hx_3^{\alpha_3},$
\begin{equation}\label{ortho_3d}
\begin{array}{rl}
\left( {\JJ}_{{\bell}}^{{\balpha}},{\JJ}_{\bm{k}}^{{\balpha}}\right)_{\omega^{\balpha},\hat{\TT}}=&\gamma_{\bell}^{\balpha}\delta_{\bell,\bm{k}},
\quad \ell_1, k_1 \ge \chi(\alpha_0) + \chi(\alpha_1),\, \ell_2, k_2\ge \chi(\alpha_2),\, \ell_3, k_3\ge \chi(\alpha_3),\\[1.em]
\gamma_{\bell}^{\balpha}=&h_{\ell_1}^{\alpha_0,\alpha_1} 
h_{\ell_2}^{2\ell_1+\alpha_0+\alpha_1+1,\alpha_2}
h_{\ell_3}^{2\ell_1+2\ell_2+\alpha_0+\alpha_1+\alpha_2+2,\alpha_3},
\end{array}
\end{equation}
where $h_{k}^{\alpha,\beta}$ is defined as in \eqref{jacobi_ortho}.

Various recurrence relations and differentiation properties of generalized Koornwinder polynomials  are consequently  achieved  according to those  of generalized Jacobi polynomials in Lemma \ref{lemma_increasing}-\ref{lemma_differential}.
For the sake of brevity, 
we conclude these identities of generalized Koornwinder polynomials in Appendix \ref{section_increase}-\ref{section_diff}.

Define the column vector 
\begin{equation*}\label{PP_m}
\PP^m=
\left(
\begin{array}{c}
\mathbf{P}^{m}_0\\
\PP^{m}_1\\
\vdots\\
\PP^{m}_m
\end{array}
\right),\quad
\PP^{m}_k = \left(
\begin{array}{c}
{\JJ}^{\balpha}_{k,0,m-k}(\bm \hx)\\
{\JJ}^{\balpha}_{k,1,m-k-1}(\bm \hx)\\
\vdots\\
{\JJ}^{\balpha}_{k,m-k,0}(\bm \hx)
\end{array}
\right),\quad 0\le k\le m,
\end{equation*}
for all generalized Koornwinder polynomials of  degree  $m$.
It is well known that 
\begin{equation}\label{def_rm}
\PP^m\in \RR^{r_m \times 1},\quad
r_m: = \binom{m+2}{m}.
\end{equation}
Thus, the following column vector
\begin{equation}\label{bmPM}
\bm{P}_M = 
\left(
\begin{array}{c}
\PP^0\\
\PP^1\\
\vdots\\
\PP^M
\end{array}
\right),
\end{equation}
 contains all generalized Koornwinder  polynomials of degree no greater than $M$.

The three-term recurrence relation for $\JJ_{\bell}^{\balpha}$ is concluded in the following theorem.


\begin{theorem}\label{3Recurrence}
For any $\bell\in \NN_0^3$ and $\balpha\in [-1,+\infty)^4,$ it holds that
\begin{align}
\hx_1 {\JJ}_{{\bell}}^{{\balpha}}(\hat{\bs  x})&=
 \sum\limits_{p=-1}^1\sum\limits_{q=-1}^1 \sum\limits_{r=-1}^1\mathscr{C}_{p,q,r}(\bell,\balpha) {\JJ}^{\balpha}_{{\bell}-(p,\, q-p,\, r-q)}(\hat{\bs  x}),\label{recurrence_x1}\\[0.05em]
 \hx_2 {\JJ}_{\bell}^{\balpha} (\hat{\bs  x})&=
 \sum\limits_{q=-1}^1  \sum\limits_{r=-1}^1\mathscr{C}_{q,r}(\bell,\balpha){\JJ}^{\balpha}_{{\bell}-(0,\, q,\,r-q)}(\hat{\bs  x}),\label{recurrence_x2}\\[0.05em]
\hx_3 {\JJ}_{\bell}^{\balpha} (\hat{\bs  x})&= 
\sum\limits_{r=-1}^1 \mathscr{C}_{r}(\bell,\balpha){\JJ}^{\balpha}_{{\bell}-(0,\, 0,\, r)}(\hat{\bs  x}),\label{recurrence_x3}
\end{align}
where  $\mathscr{C}_{p,q,r}(\bell,\balpha)$, $\mathscr{C}_{q,r}(\bell,\balpha)$ and $\mathscr{C}_{r}(\bell,\balpha)$ are expansion coefficients presented in Appendix \ref{threecoeff}.
Equivalently, for any $m=|\bell|\in \NN_0$ and $\bm{\hx}\in  \hat{\mathcal{T}}$, there exist unique matrices $A_{m}\in \RR^{3 r_m\times r_{m+1}}$, $B_{m}(\bm{\hx})\in \RR^{3 r_m\times r_m}$ and $C_{m}
\in \RR^{ 3r_m\times r_{m-1}}$  with 
\begin{align*}
&A_m = {\begin{pmatrix}
   E^1_0      &F^1_0          \\[0.2em]
   G^1_1     & E^1_1 &  F^1_1  \\[0.2em]
      & \ddots   & \ddots  & \ddots    \\[0.2em]                            
&      & \ddots   & \ddots  & \ddots    \\[0.2em]                            
  &&&   G^1_m      & E^1_m  &   F^1_m \\[0.2em]
      E^2_0           \\[0.2em]
         & E^2_1    \\[0.2em]
      &   & \ddots      \\[0.2em]                            
      &&   & \ddots      \\[0.2em]                            
&  &&        & E^2_m   \\[0.2em]
       E^3_0          \\[0.2em]  
         & E^3_1    \\[0.2em]  
      &   & \ddots     \\[0.2em]                            
    &  &   & \ddots     \\[0.2em]                             
&  &&        & E^3_m 
  \end{pmatrix}},
\qquad
C_m ={  \begin{pmatrix}
   Y^1_0      &Z^1_0           \\[0.2em] 
       X^1_1  & Y^1_1 & \ddots  \\[0.2em]                       
      & X^1_2  & \ddots  & Z^1_{m-2} \\[0.2em]                          
      && \ddots   & Y^1_{m-1}  \\[0.2em]                          
  &&& X_m^1
\\   
     Y^2_0          \\[0.2em]    
         & Y^2_1    \\[0.2em]    
      &   & \ddots     \\[0.2em]                             
  & &        & Y^2_{m-1}  \\[0.2em]    
\\   
\\
     Y^3_0          \\[0.2em]    
         & Y^3_1    \\[0.2em]    
      &   & \ddots      \\[0.2em]                            
  &&        & Y^3_{m-1} \\[0.2em]    
\\[0.2em]    
\end{pmatrix}},
\end{align*}
\begin{align*}
 &B_m(\hat{\bm x}) ={\begin{pmatrix}
   V^1_0-\hat x_1 I      &W^1_0         \\[0.2em]  
   U^1_1     & V^1_1-\hat x_1 I &  W^1_1 \\[0.2em]  
      & U^1_2 & \ddots  &  \ddots     \\[0.2em]                          
      && \ddots   & V^1_{m-1}-\hat x_1 I  & W^1_{m-1}     \\[0.2em]                          
  &&&   U^1_m      & V^1_m -\hat x_1 I 
 \\[0.2em]  
       V^2_0-\hat x_2 I         \\[0.2em]  
         & V^2_1-\hat x_2 I    \\[0.2em]  
      &   & \ddots     \\[0.2em]                            
  &&        & V^2_{m-1}-\hat x_2 I  \\[0.2em]  
  &&        && V^2_m -\hat x_2 I 
\\[0.2em]  
     V^3_0 -\hat x_3 I          \\[0.2em] 
         & V^3_1-\hat x_3 I    \\[0.2em] 
      &   & \ddots      \\[0.2em]                          
  &&        & V^3_{m-1}-\hat x_3 I  \\[0.2em] 
  &&        && V^3_m -\hat x_3 I 
\end{pmatrix}},
\end{align*}
such that
\begin{equation}\label{3recurrence}
 A_{m} \PP^{m+1} + B_{m}(\hat{\bm x}) \PP^m + C_{m} \PP^{m-1} = \bm{0},
\end{equation}
where
$E^i_k\in \RR^{(m+1-k)\times (m+2-k)}$, $V^i_k\in \RR^{(m+1-k)\times (m+1-k)}$ and $Y^i_k\in \RR^{(m+1-k)\times (m-k)}$  are tridiagonal for $i=1,2$ and diagonal for $i=3$;
$F^1_k\in \RR^{(m+1-k)\times (m+1-k)} $, $W^1_k\in \RR^{(m+1-k)\times (m-k)} $ and $Z^1_k\in \RR^{(m+1-k)\times (m-k-1)} $ are lower tridiagonal (i.e., the main diagonal plus  two immediate subdiagonals);
 and $G^1_k\in \RR^{(m+1-k)\times (m+3-k)} $, $U^1_k\in \RR^{(m+1-k)\times (m+2-k)} $,  and $X^1_k\in \RR^{(m+1-k)\times (m+1-k)} $ are upper tridiagonal (i.e., the main diagonal plus  two immediate supdiagonals). 



\end{theorem}
We also postpone the derivation of 
coefficients $\mathscr{C}_{p,q,r}(\bell,\balpha)$, $\mathscr{C}_{q,r}(\bell,\balpha)$ and $\mathscr{C}_{r}(\bell,\balpha)$ to Appendix \ref{threecoeff}.
In this paper, we are more interested in generalized Koornwinder polynomials of the case 
\begin{equation}\label{1_koorn_3d}
\begin{array}{rl}
&{\JJ}_{\bell}(\bm{\hat{x}}):={\JJ}_{\ell_1,\ell_2,\ell_3}^{-1,-1,-1,-1}(\bm{\hat{x}}),
\end{array}
\end{equation}
which would be used to design modal basis functions for tetrahedral spectral elements.

\section{Clenshaw algorithm for Koornwinder expansions}\label{clenshew}
In general, the Clenshaw algorithm is designed to evaluate the sum of a finite series of functions which satisfy a linear recurrence relation.
In this section, we set focus on the Clenshaw algorithm 
 to evaluate  the following Koornwinder expansion on the reference tetrahedron,
\begin{equation}\label{f_expansion}
f(\bs{\hx}) = \sum\limits_{|\bell|=0}^M 
\widehat{f}_{\bell} {\JJ}^{\balpha}_{\bell}(\bs{\hx}) = {\bm{P}_M} ^\tr \bm{F}_M,\quad  \bs{\hx}\in \hat{\TT}, M\in\NN,
\end{equation}
 where $\bm{P}_M$ is defined as in \eqref{bmPM}, and

\begin{equation*}
\begin{aligned}
&\bm{F}_M= 
\begin{pmatrix}
\F^0\\
\F^1\\
\vdots\\
\F^M
\end{pmatrix},\quad {\rm{with}}\;
\F^m = 
\begin{pmatrix}
\F^{m}_0\\
\F^{m}_1\\
\vdots\\
\F^{m}_m
\end{pmatrix},\quad 
\F^{m}_k = 
\begin{pmatrix}
\hat{f}_{k,0,m-k}\\
\hat{f}_{k,1,m-k-1}\\
\vdots\\
\hat{f}_{k,m-k,0}
\end{pmatrix},\quad
\begin{array}{l}
0\le k\le m,\\
0\le m\le M.
\end{array}
\end{aligned}
\end{equation*}


Indeed, the three-term recurrence relation \eqref{3recurrence} yields
\begin{equation}\label{G_M}
\mathbb{G}_M\bm{P}_M= 
\bm{e}_{1},
\end{equation}
where $ \mathbb{G}_M $ is a block lower tridiagonal matrix
\begin{align*}
 \mathbb{G}_M=
  \begin{pmatrix}
1          \\
 B_0(\bm{\hx})   & A_0    \\ 
 C_1   & B_1(\bm{\hx}) & A_1\\
         & \ddots &\ddots &\ddots\\
   & &C_{M-1}   & B_{M-1}(\bm{\hx}) & A_{M-1}
  \end{pmatrix}.
  \end{align*}
It has been concluded in \cite[Theorem 3.2.4]{DX2001} that the matrix $A_m$ has full column rank and there exists a generalized inverse $D_m\in\RR^{r_{m+1}\times 3r_m}$ such that 
$$D_m  A_m = I .$$

We claim that the sparsity of $A_m$  admits a sparse $D_m$ as follows,
\begin{align}\label{Dm}
&D_m = \left(\begin{array}{lllllllllllll}
\bm 0& & & & D_{0}^2 & & & &D_{0}^3 \\
&   \ddots  & & & & \ddots & & & &\ddots  \\
& & \bm 0 & & & & D_{m-1}^2 & & & & D_{m-1}^3  \\
& & & \bm 0 & & & & D_{m}^2 & & &  & D_m^3 \\[0.2em]
& & & \bm v_m^1 & & & \bm v_{m-1}^2 &  \bm v_{m}^2  & & &   \bm v_{m-1}^3  &  \bm v_{m}^3 
\end{array}\right),
\end{align}
where $D_k^2, D_k^3\in \RR^{(m+2-k)\times (m+1-k)}$, $\bm v_{m-1}^2, \bm v_{m-1}^3\in \RR^{1\times 2}$ and $\bm v_{m}^1$, $\bm v_{m}^2$, $\bm v_{m}^3 \in \RR.$
Indeed,
\begin{align*}
 D_mA_m= {\footnotesize\begin{pmatrix}
   D_0^2E_0^2+D_0^3E_0^3      &       \\[0.2em]
   &D_1^2E_1^2+D_1^3E_1^3 &    \\[0.2em]
      &  & \ddots     \\[0.2em]                                                      
  &&  & D_{m-1}^2E_{m-1}^2+D_{m-1}^3E_{m-1}^3       \\[0.2em]   
  &&  && D_m^2E_m^2+D_m^3E_m^3       \\[0.2em]   
  &&  &\bm{v}_m^1 G_m^1+ \bm v_{m-1}^2 E_{m-1}^2+{\bm{v}}^3_{m-1} E_{m-1}^3& \bm v_m^1 E_m^1+ \bm v_m^2 E_m^2+\bm v^3_m E_m^3 & \bm v_m^1F_m^1       \\[0.2em]   
  \end{pmatrix}},
\end{align*}
leads to
\begin{align}
 &D_k^2 E_k^2 + D_k^3 E_k^3 = I,\quad 0\le k\le m ,\label{DEident}
 \\[0.1em]
 &\bm v_m^1F_m^1    = 1,\label{lastrow3}\\[0.1em]
&\bm{v}_m^1 G_m^1+ \bm v_{m-1}^2 E_{m-1}^2+{\bm{v}}^3_{m-1} E_{m-1}^3 = \bm 0, \label{lastrow1}\\[0.1em]
&\bm v_m^1 E_m^1+ \bm v_m^2 E_m^2+\bm v^3_m E_m^3 = \bm 0.\label{lastrow2}
\end{align}
Note that $F_m^1\in \RR$,   $G_m^1\in \RR^{1\times 3}$, $E_m^i\in \RR^{1\times 2}$ for $1\le i\le 3$ with $E_m^3(1,2)=0 $,
 while  $E_{m-1}^3\in  \RR^{2\times 3}$ is diagonal and $E_{m-1}^2\in  \RR^{2\times 3}$ is tridiagonal.
  Combining \eqref{lastrow3} with  \eqref{lastrow2} yields 
\begin{equation*}
 (\bm v^3_m,\bm v_m^2,\bm v_m^1) \begin{pmatrix} E_m^3 &0  \\[0.2em]  E_m^2 &0\\[0.2em]  E_m^1 & F_m^1 \end{pmatrix} = (0,0,1).
\end{equation*}
Thus, we  solve that
\begin{align}
\label{vm1}
\begin{split}
 (\bm v^3_m,\bm v_m^2,\bm v_m^1)&= \bm{e}_3^\tr \,\begin{pmatrix} E_m^3  &0\\[0.2em]  E_m^2&0 \\[0.2em]  E_m^1 & F_m^1 \end{pmatrix}^{-1} 
 =\bm{e}_3^\tr\,\begin{pmatrix} \frac{1}{E_m^3(1,1)}  &0 & 0\\[0.4em]  
 -\frac{E_m^2(1,1)}{E_m^3(1,1)E_m^2(1,2)} &\frac{1}{E_m^2(1,2)} & 0 \\[0.4em] 
\frac{E_m^1(1,2) E_m^2(1,1)  - E_m^1(1,1) E_m^2(1,2)}{F_m^1 E_m^3(1,1) E_m^2(1,2)}  & -\frac{E_m^1(1,2)}{E_m^2(1,2)F_m^1} & \frac{1}{F_m^1} \end{pmatrix}
 \\
&=\begin{pmatrix}  \frac{E_m^1(1,2) E_m^2(1,1)  - E_m^1(1,1) E_m^2(1,2)}{F_m^1 E_m^3(1,1) E_m^2(1,2)}  ,  
- \frac{E_m^1(1,2)}{F_m^1  E_m^2(1,2)} ,\frac{1}{F_m^1}  \end{pmatrix}.
\end{split}
\end{align}
%
%
%
%
Substituting \eqref{vm1} into \eqref{lastrow1} and letting $\bm v_{m-1}^2(1)=0$, we further  obtain
\begin{align*}
(\bm v_{m-1}^3,\bm v_{m-1}^2(2))  \begin{pmatrix} E_{m-1}^3\\[0.2em]  E_{m-1}^2(2,:) \end{pmatrix} =- \frac{1}{F_m^1} \bm G_m^1.
\end{align*}
Owing to fact 
\begin{align*}
\begin{pmatrix} E_{m-1}^3\\[0.2em]
  E_{m-1}^2(2,:)
   \end{pmatrix}^{-1} 
= \begin{pmatrix} \frac{1}{E_{m-1}^3(1,1)}  & 0 & 0 \\[0.3em]  
0&  \frac{1}{E_{m-1}^3(2,2)} & 0\\[0.3em] 
-\frac{ E_{m-1}^2(2,1)}{E_{m-1}^3(1,1)E_{m-1}^2(2,3)} & -\frac{ E_{m-1}^2(2,2)}{E_{m-1}^3(2,2)E_{m-1}^2(2,3)}  
&\frac{ 1}{E_{m-1}^2(2,3)}  \end{pmatrix},
\end{align*}
we find  that
\begin{align}\label{vm2}
\begin{split}
&\bm{v}_{m-1}^2=
\begin{pmatrix}
{0} ,
 -\frac{G_m^1(1,3)}{ F_m^1 E_{m-1}^2(2,3)}
 \end{pmatrix},
\\
&\bm{v}_{m-1}^3  = 
\begin{pmatrix}
 \frac{G_m^1(1,3) E_{m-1}^2(2,1) - G_m^1(1,1) E_{m-1}^2(2,3) }{ F_m^1 E_{m-1}^3(1,1) E_{m-1}^2(2,3)},
\frac{  G_m^1(1,3) E_{m-1}^2(2,2) - G_m^1(1,2)E_{m-1}^2(2,3) }{ F_m^1 E_{m-1}^3(2,2) E_{m-1}^2(2,3)} \end{pmatrix}.
\end{split}
\end{align}

We now determine $D_k^2$ and $D_k^3$ from \eqref{DEident}.
Assume $D^2_k(:,1:m-k)=0$. Then \eqref{DEident} becomes
\begin{align*}
(D_k^3, D_k^2(:,m+1-k))   \begin{pmatrix}
E_k^3  \\[0.2em]
E_k^2(m+1-k,:)
 \end{pmatrix} = I.
\end{align*}
Since $E_k^3\in \RR^{(m+1-k)\times(m+2-k)} $ is diagonal and  $E_k^2\in \RR^{(m+1-k)\times(m+2-k)} $ is tridiagonal,
we derive 
\begin{align*}
 \begin{pmatrix}
E_k^3  \\[0.2em]
E_k^2(m+1-k,:)
 \end{pmatrix}^{-1} =  
 \begin{pmatrix}
\frac{1}{E_k^3(1,1)} \\
&\ddots
\\
&&\frac{1}{E_k^3(m+1-k,m+1-k)}
\\[0.5em]
&-\beta_k\frac{E_k^2(m+1-k,m-k) }{E_k^3(m-k,m-k) } &
-\beta_k\frac{E_k^2(m+1-k,m+1-k) }{E_k^3(m+1-k,m+1-k) } &\beta_k
 \end{pmatrix},
\end{align*}
by denoting $\beta_k = \frac{1}{E_k^2(m+1-k,m+2-k)}$ and $E_k^2(1,0) = 0.$
Thus,
\begin{align}\label{Dk_23}
D_k^3 = \begin{pmatrix}
\frac{1}{E_k^3(1,1)} \\
&\ddots
\\
&&\frac{1}{E_k^3(m-k,m-k)}\\[0.5em]
&&&\frac{1}{E_k^3(m+1-k,m+1-k)}
\\[0.5em]
&&-\beta_k\frac{E_k^2(m+1-k,m-k) }{E_k^3(m-k,m-k) } &-\beta_k\frac{E_k^2(m+1-k,m+1-k) }{E_k^3(m+1-k,m+1-k) } 
 \end{pmatrix}, \quad
  D_k^2 = \begin{pmatrix}
0 \\[0.5em]
&\ddots
\\[0.5em]
&&0
\\[0.5em]
&&&0\\[0.5em]
&&& \beta_k
 \end{pmatrix}.
\end{align}

From \eqref{G_M} and the definition of $D_m$, one readily obtains that
\begin{equation} \label{G_tilde}
\tilde{\mathbb{G}}_M\bm{P}_M = 
\bm{e}_{1},  
\end{equation}
where 
\begin{align}\label{Gmatrix}
\tilde{\mathbb{G}}_M := 
\begin{pmatrix}
1&&&\\
&D_0&&\\
&& \ddots&\\
&&& D_{M-1}
\end{pmatrix} \mathbb{G}_M=
  \begin{pmatrix}
1          \\
 D_0B_0(\bm{\hx})   & I    \\ 
 D_1C_1   & D_1 B_1(\bm{\hx}) & I\\
         & \ddots &\ddots &\ddots\\
   & &D_{M-1}C_{M-1}   & D_{M-1}B_{M-1}(\bm{\hx}) & I
  \end{pmatrix}.
\end{align}
Combining \eqref{f_expansion} and \eqref{G_tilde}, one has
\begin{equation*}
\begin{aligned}
f(\bs{\hx}) &= {\bm{P}_M}^\tr  \bm{F}_M = \bm{e}_{1}^\tr {\tilde{\mathbb{G}}_M}^{-\tr} \bm{F}_M.
\end{aligned}
\end{equation*}
Denote 
\begin{align*}
{\tilde{\mathbb{G}}_M}^{-\tr} \bm{F}_M=\bm{b}_M,\quad{\text{with}}\,\,
\bm b_M=\begin{pmatrix}
\bm b^0\\
\bm b^1\\
\vdots\\
\bm b^M
\end{pmatrix},\quad
\bm b^m\in\RR^{r_m\times 1},\, 0\le m\le M.
\end{align*}
Then $f(\bs{\hx}) =\bm b^0$ is exactly the first entry of $\bm{b}_M$, 
which can be solved recursively by
\begin{equation}
\label{bmRec}
\begin{cases}
\bm b^M = \bm F^M,\\
\bm b^{M-1} = \bm F^{M-1} - B_{M-1}^\tr(\bm{\hx}) D_{M-1}^\tr \bm b^{M},\\
\bm b^m = \bm F^m  - B_m^\tr(\bm{\hx}) D_m^\tr \bm b^{m+1} - C_{m+1}^\tr D_{m+1}^\tr \bm b^{m+2},\quad m= M-2, M-3,\dots,0.
\end{cases}
\end{equation}

Thus, we summarize the Chenshaw algorithm as follows.

\begin{algorithm}[H]
\caption{The Clenshaw Algorithm} 
\hspace*{0.02in} {\bf Input:} $M$, $\bm{F}_M$, $A_i, B_i(\bm{\hx}) \,(0\le i \le M-1)$, $C_i\, (1\le i\le M-1)$  \\
\hspace*{0.02in} {\bf Output:} the value of $f(\bs{\hx})$
\begin{algorithmic}[1]
\State{Compute the matrices: $D_0, D_1, \cdots D_{M-1}$ from \eqref{Dm}, \eqref{vm1}, \eqref{vm2} and \eqref{Dk_23}.}
\State{Solve the linear equation: 
\begin{equation}\label{linear_equ}
{\tilde{\mathbb{G}}_M}^{\tr} \bm{b}_M = \bm{F}_M, 
\end{equation}
through \eqref{bmRec}, where $\tilde{\mathbb{G}}_M$ is defined as in \eqref{Gmatrix}.}
\State{$f(\bs{\hx})=\bm{b}^0.$}
\end{algorithmic}
\end{algorithm}

Since it contains at most thirteen non-zero entries in each column of $B_m$ and $C_m$, and at most two non-zero entries in each column of $D_m$,  only  $ \frac{53 M^3}{6}+ \mathcal{O}(M^2)$ operations are required to solve \eqref{linear_equ}. In return, the Clenshaw algorithm shares the same order of complexity.

%

\section{Sparse spectral-Galerkin method on an arbitrary tetrahedron}\label{spectral-Galerkin}
In this section, we shall  design sparse spectral-Galerkin approximation scheme on an arbitrary tetrahedron $\TT$
with vertices
$$P_j = \bm{x}^{(j)} = (x_{1}^{(j)}, x_{2}^{(j)}, x_{3}^{(j)})^\tr,\quad 0\le j\le 3,$$
which is affine equivalent to the reference tetrahedron $\hat{\TT}$ via
\begin{equation}\label{mapping}
\Psi: \hat{\TT}\rightarrow \TT.
\end{equation}
\subsection{Variational formulation and numerical scheme}

Consider the second-order model equation on the tetrahedron $\TT$:
\begin{equation}\label{model_pde}
\begin{cases}
-\Delta u (\bm{x})+\gamma(\bm{x}) u(\bm{x})=f(\bm{x}), \quad &\bm{x}\in\TT,\\
 u(\bm{x})=g(\bm{x}),\quad &\bm{x}\in\partial\TT,
\end{cases}
\end{equation} 
where $\gamma\ge 0.$ The variational formulation of \eqref{model_pde} reads: 
to find $u\in H^1(\TT)$ such that $u=g$ on $\partial \TT$
and
\begin{equation}\label{variation_pde}
a_\gamma(u,v) := (\nabla u, \nabla v)_{\TT} + ( \gamma u,v)_{\TT} = (f,v)_{\TT},\quad \forall v\in H_0^1(\TT).
\end{equation}
$\gamma$ is dropped from the notation $a_\gamma(\cdot,\cdot)$ when $\gamma=0$.
It is straightforward by the Lax-Milgram lemma \cite{Evans1998} that \eqref{variation_pde} admits a unique solution.

For any $M\in \NN_0$, define the approximation space as 
\begin{equation*}
X_{M} := \P_M(\TT) \cap H^1(\TT),\quad 
X_{M,0}:=\P_M(\TT) \cap H_0^1(\TT).
\end{equation*}
Then the numerical scheme for \eqref{variation_pde} reads: to find $u_{M}\in X_{M}$ such that
\begin{equation}
\begin{cases}
a_\gamma(u_{M}, v_M) = (f, v_M)_{\TT},\quad &\forall v_M\in X_{M,0},\\
\left( u_M, \phi_M \right)_{\partial\TT} 
= \left( g, \phi_M\right)_{\partial\TT},
\quad &\forall \phi_M\in X_M\setminus X_{M,0}.
\end{cases}
\label{scheme_pde}
\end{equation}
It is worthy to note that the second equation in \eqref{scheme_pde} defines a unique $u_b\in X_M\setminus X_{M,0}$,
and the Lax-Milgram  lemma implies a unique solution $u_0\in X_{M,0}$ to  
$a_\gamma(u_{0}, v_M) = (f, v_M)_{\TT} - a_\gamma(u_{b}, v_M)$.
Thus $u_M=u_0+u_b$ is uniquely solvable.  


For the Laplacian eigenvalue problem: 
\begin{equation}\label{model_eigen}
\begin{cases}
-\Delta u(\bm{x}) =\mu u(\bm{x}), \quad &\bm{x}\in\TT,\\
 u(\bm{x})=0,\quad & \bm{x}\in\partial\TT,
\end{cases}
\end{equation} 
the variational formulation is defined by 
\begin{equation}\label{variation_eigen}
a(u,v) = \mu (u,v)_{\TT},\quad \forall v\in H_0^1(\TT),
\end{equation}
and the corresponding numerical scheme reads: to find $u_M\in X_{M,0}$ such that
\begin{equation}\label{scheme_eigen}
a(u_M, v_M) = \mu_M(u_M, v_M)_{\TT},\quad \forall v_M\in X_{M,0}.
\end{equation}

\subsection{Implementations}

\subsubsection{Shape functions}\label{modalbasis}
The space $X_M$ provides much of convenience in treating non-homogeneous boundary conditions and in enforcing continuity across the interface for the tetrahedral spectral element method.
Let
\begin{equation*}
X_M = \{  \varphi_{\bell}(\bs{x})=\hat{\varphi}_{\bell}(\bs{\hat{x}})\circ\Psi^{-1}: 0\le \ell_1,\ell_2,\ell_3,|\bell|\le M \},
\end{equation*}
where $\Psi$ is defined as in \eqref{mapping}
and $\hat{\varphi}_{\bell}$ 
are  proper basis functions defined on the reference tetrahedron.  
We further let $\hat{F}_j$ be the face opposite to the vertex $\hat{P}_j$
 and 
$$\hat{E}_{jk} = \hat{P}_j\hat{P}_k,\quad 0\le j<k \le 3,$$
denote the edge of $\hat{\TT}$ within the endpoints $\hat{P}_j$ and $\hat{P}_k$.

Modal basis functions are split into interior and boundary modes (including face, edge and vertex modes).
The interior modes are identically zero on the tetrahedron boundary, 
and the face modes only have magnitude along one face and are zero at all other faces, while the edge modes only have magnitude along one edge and the vertex modes only have magnitude at one vertex.

$\bullet$ Interior modes:
\begin{equation*}
\begin{array}{rl}
&\quad\hat{\varphi}_{\ell_1,\ell_2,\ell_3}(\bs{\hx})={\JJ}_{\ell_1,\ell_2,\ell_3}(\bs{\hx}),\quad
( \ell_1\geq2,\ell_2\geq1,\ell_3\geq1).
\end{array}
\end{equation*}

$\bullet$ Face modes:
\begin{equation*}
\begin{array}{rl}
\hat{F}_0:& \hat{\varphi}_{1,\ell_2-1,\ell_3}(\bs{\hx})={\JJ}_{0,\ell_2,\ell_3}(\bs{\hx})- \dfrac{\ell_2-1}{\ell_2}{\JJ}_{1,\ell_2-1,\ell_3}(\bs{\hx}),\quad (\ell_2\geq2,\ell_3\geq1),\\[0.6em]
\hat{F}_1:& \hat{\varphi}_{0,\ell_2,\ell_3}(\bs{\hx})={\JJ}_{0,\ell_2,\ell_3}(\bs{\hx})+ \dfrac{\ell_2-1}{\ell_2}{\JJ}_{1,\ell_2-1,\ell_3}(\bs{\hx}),\quad ( \ell_2\geq2,\ell_3\geq1), \\[0.6em]
\hat{F}_2:& \hat{\varphi}_{\ell_1,0,\ell_3}(\bs{\hx})={\JJ}_{\ell_1,0,\ell_3}(\bs{\hx}),\quad (\ell_1\geq2,\ell_3\geq1),\\[0.6em]
\hat{F}_3:& \hat{\varphi}_{\ell_1,\ell_2,0}(\bs{\hx})={\JJ}_{\ell_1,\ell_2,0}(\bs{\hx}),\quad(\ell_1\geq2,\ell_2\geq1).
\end{array}
\end{equation*}

$\bullet$ Edge modes:
\begin{equation*}
\begin{array}{rl}
\hat{E}_{01}:&\hat{\varphi}_{\ell_1,0,0}(\bs{\hx})={\JJ}_{\ell_1,0,0}(\bs{\hx}),\quad (\ell_1\geq2),\\[0.6em]
\hat{E}_{02}:&\hat{\varphi}_{0,\ell_2,0}(\bs{\hx})={\JJ}_{0,\ell_2,0}(\bs{\hx})+ \dfrac{\ell_2-1}{\ell_2} {\JJ}_{1,\ell_2-1,0}(\bs{\hx}),\quad (\ell_2\geq2),\\[0.6em]
\hat{E}_{03}:& \hat{\varphi}_{0,0,\ell_3}(\bs{\hx})= \dfrac{1}{2}{\JJ}_{0,0,\ell_3}(\bs{\hx})+ \dfrac{\ell_3-1}{2\ell_3} {\JJ}_{0,1,\ell_3-1}(\bs{\hx})+ \dfrac{\ell_3-1}{\ell_3} {\JJ}_{1,0,\ell_3-1}(\bs{\hx}),\quad (\ell_3\geq2),\\[0.6em]
\hat{E}_{13}:&\hat{\varphi}_{1,0,\ell_3-1}(\bs{\hx})= \dfrac{1}{2}{\JJ}_{0,0,\ell_3}(\bs{\hx})+ \dfrac{\ell_3-1}{2\ell_3} {\JJ}_{0,1,\ell_3-1}(\bs{\hx})- \dfrac{\ell_3-1}{\ell_3} {\JJ}_{1,0,\ell_3-1}(\bs{\hx}),\quad (\ell_3\geq2),\\[0.6em]
\hat{E}_{12}:&\hat{\varphi}_{1,\ell_2-1,0}(\bs{\hx})={\JJ}_{0,\ell_2,0}(\bs{\hx})- \dfrac{\ell_2-1}{\ell_2} {\JJ}_{1,\ell_2-1,0}(\bs{\hx}),\quad (\ell_2\geq2),\\[0.6em]
\hat{E}_{23}:&\hat{\varphi}_{0,1,\ell_3-1}(\bs{\hx})={\JJ}_{0,0,\ell_3}(\bs{\hx})- \dfrac{\ell_3-1}{\ell_3} {\JJ}_{0,1,\ell_3-1}(\bs{\hx}),\quad (\ell_3\geq2).
\end{array}
\end{equation*}

$\bullet$ Vertex modes:
\begin{equation*}
\begin{array}{rl}
\hat{P}_0:& \hat{\varphi}_{0,0,0}(\bs{\hx})= \dfrac{1}{8}{\JJ}_{0,0,0}(\bs{\hx})- \dfrac{1}{2}{\JJ}_{1,0,0}(\bs{\hx})- \dfrac{1}{4}{\JJ}_{0,1,0}(\bs{\hx})-\dfrac{1}{8}{\JJ}_{0,0,1}(\bs{\hx}),\\[0.6em]
\hat{P}_1:&\hat{\varphi}_{1,0,0}(\bs{\hx})= \dfrac{1}{8}{\JJ}_{0,0,0}(\bs{\hx}) + \dfrac{1}{2}{\JJ}_{1 ,0,0}(\bs{\hx}) -\dfrac{1}{4} {\JJ}_{0,1,0}(\bs{\hx})- \dfrac{1}{8} {\JJ}_{0,0,1}(\bs{\hx}),\\[0.6em]
\hat{P}_2:&\hat{\varphi}_{0,1,0}(\bs{\hx})= \dfrac{1}{4}{\JJ}_{0,0,0} (\bs{\hx}) +  \dfrac{1}{2}{\JJ}_{0,1,0}(\bs{\hx})- \dfrac{1}{4} {\JJ}_{0,0,1}(\bs{\hx}),\\[0.6em]
\hat{P}_3:&\hat{\varphi}_{0,0,1}(\bs{\hx})=\dfrac{1}{2} {\JJ}_{0,0,0}(\bs{\hx})+ \dfrac{1}{2}{\JJ}_{0,0,1}(\bs{\hx}).
\end{array}
\end{equation*}

\begin{remark}
Similar shape functions have been studied in literature, including the modal basis functions proposed by Sherwin and Karniadakis based on  mixed-weight Jacobi polynomials \cite{KS1995, Karniadakis2005} and those designed by Beuchler et al. 
 employing integrated Jacobi polynomials \cite{intlegendre2006, intlegendre2007, intlegendre2008}.
Both of them are expressed as a generalized tensor product of polynomials in one dimensions.
In comparison, our modal basis functions have a  simple presentation in generalized Koornwinder polynomials.
 Specifically, these three kinds of modal basis functions coincide, up to generic constants, with each other in interior modes  and main differences exist in boundary modes.
\end{remark}

\subsubsection{Equivalent algebraic system}

We shall examine the linear system associated with the numerical scheme \eqref{scheme_pde} and \eqref{scheme_eigen} when $g=0$.
It is obvious that $\varphi_{\bell}$ of interior modes provide a series of basis functions for $X_{M,0}$  that 
\begin{equation*}
X_{M,0} = \Span\{ \varphi_{\bell} (\bs{x}): \ell_1\ge 2, \ell_2, \ell_3\ge 1, |\bell|\le M\}.
\end{equation*} 
The basis polynomials are arranged in ${\tilde{\Phi}}_M$ such that
\begin{equation*}
{\tilde{\Phi}}_M=\begin{pmatrix}
\bm{\tilde{\varphi}}_2\\
\bm{\tilde{\varphi}}_3\\
\vdots\\
\bm{\tilde{\varphi}}_{M-2}
\end{pmatrix},\quad {\text{with}}\,\,
\bm{\tilde{\varphi}}_{\ell_1} = \begin{pmatrix}
\bm{\tilde{\varphi}}_{\ell_1,1}\\
\bm{\tilde{\varphi}}_{\ell_1,2}\\
\vdots\\
\bm{\tilde{\varphi}}_{\ell_1,M-\ell_1-1}
\end{pmatrix},\quad
\bm{\tilde{\varphi}}_{\ell_1,\ell_2} = \begin{pmatrix}
{\varphi}_{\ell_1,\ell_2,1}\\
{\varphi}_{\ell_1,\ell_2,2}\\
\vdots\\
{\varphi}_{\ell_1,\ell_2,M-\ell_1-\ell_2}
\end{pmatrix},\quad 
\begin{array}{l}
1\le \ell_2\le M-\ell_1-1,\\
2\le \ell_1\le M-2.
\end{array}
\end{equation*}

Let
\begin{equation*}
u_M(\bs{x}) = \sum\limits_{\ell_1=2}^{M-2} \sum\limits_{\ell_2=1}^{M-\ell_1-1} \sum\limits_{\ell_3=1}^{M-\ell_1-\ell_2} \widehat{u}_{\bell}  \varphi_{\bell} (\bs{x}).
\end{equation*}
The linear system induced by \eqref{scheme_pde} becomes
\begin{equation}\label{linearsystem}
\left( {\mathcal{S}}+{\mathcal{M}}_\gamma \right) \hat{\bm{u}} = \bm{f},
\end{equation}
where
\begin{equation*}
\begin{aligned}
{\mathcal{S}} &= \int_{\TT}  \big[ \nabla {\tilde{\Phi}}_M(\bs x)\big]  \big[\nabla {\tilde{\Phi}}_M(\bs x)\big]^{\tr}  d{\bs x} 
,\quad 
{\mathcal{M}}_\gamma   = \int_{\TT} \gamma(\bs x) {\tilde{\Phi}}_M(\bs x)   {\tilde{\Phi}}_M(\bs x) ^{\tr}  d{\bs x} 
\\[0.2em]
\bm{f}&=  \int_{\TT}  f (\bs x)      {\tilde{\Phi}}_M(\bs x)  d{\bs x},\\
\end{aligned}
\end{equation*}
\begin{equation*}
\begin{aligned}
\hat{\bm{u}} &=
\begin{pmatrix}
\hat{\bm{u}}_2\\
\hat{\bm{u}}_3\\
\vdots\\
\hat{\bm{u}}_{M-2}
\end{pmatrix},\quad {\rm{with}}\;
\hat{\bm{u}}_{\ell_1} = 
\begin{pmatrix}
\hat{\bm{u}}_{\ell_1,1}\\
\hat{\bm{u}}_{\ell_1,2}\\
\vdots\\
\hat{\bm{u}}_{\ell_1,M-\ell_1-1}
\end{pmatrix},\quad 
\hat{\bm{u}}_{\ell_1,\ell_2} = 
\begin{pmatrix}
\widehat{\bm{u}}_{\ell_1,\ell_2,1}\\
\widehat{\bm{u}}_{\ell_1,\ell_2,2}\\
\vdots\\
\widehat{\bm{u}}_{\ell_1,\ell_2,M-\ell_1-\ell_2}
\end{pmatrix},\quad
\begin{array}{l}
1\le \ell_2\le M-\ell_1-1,\\
2\le \ell_1\le M-2.
\end{array}
\end{aligned}
\end{equation*}
The non-zero entries of ${\mathcal{S}}$ and ${\mathcal{M}}_\gamma$ (if $\gamma$ is a constant) can be exactly evaluated owing to the orthogonality.
Furthermore,
the numerical scheme \eqref{scheme_eigen} for eigenvalue problem is equivalent to the following system:
\begin{equation}
{\mathcal{S}} \hat{\bm{u}} = \mu_M {\mathcal{M}} \hat{\bm{u}}.
\end{equation}
Here we drop the notation $\gamma$ from ${\mathcal{M}}_\gamma$ when $\gamma=1$.

We depict 
the non-zero patterns of the stiffness matrix ${\mathcal{S}}$ and the mass matrix ${\mathcal{M}}$ in Figure \ref{exam1_strubd}.
It is observed that $\mathcal{S}$ is a block penta-diagonal matrix and $\mathcal{M}$ is a block tri-diagonal matrix, 
with all blocks being hepta-digonal, which confirm the sparsity of the discrete matrices.

\begin{figure}[H]
\centering
\subfigure{
\begin{minipage}[t]{0.5\linewidth}
\centering
\includegraphics[width=2.4in]{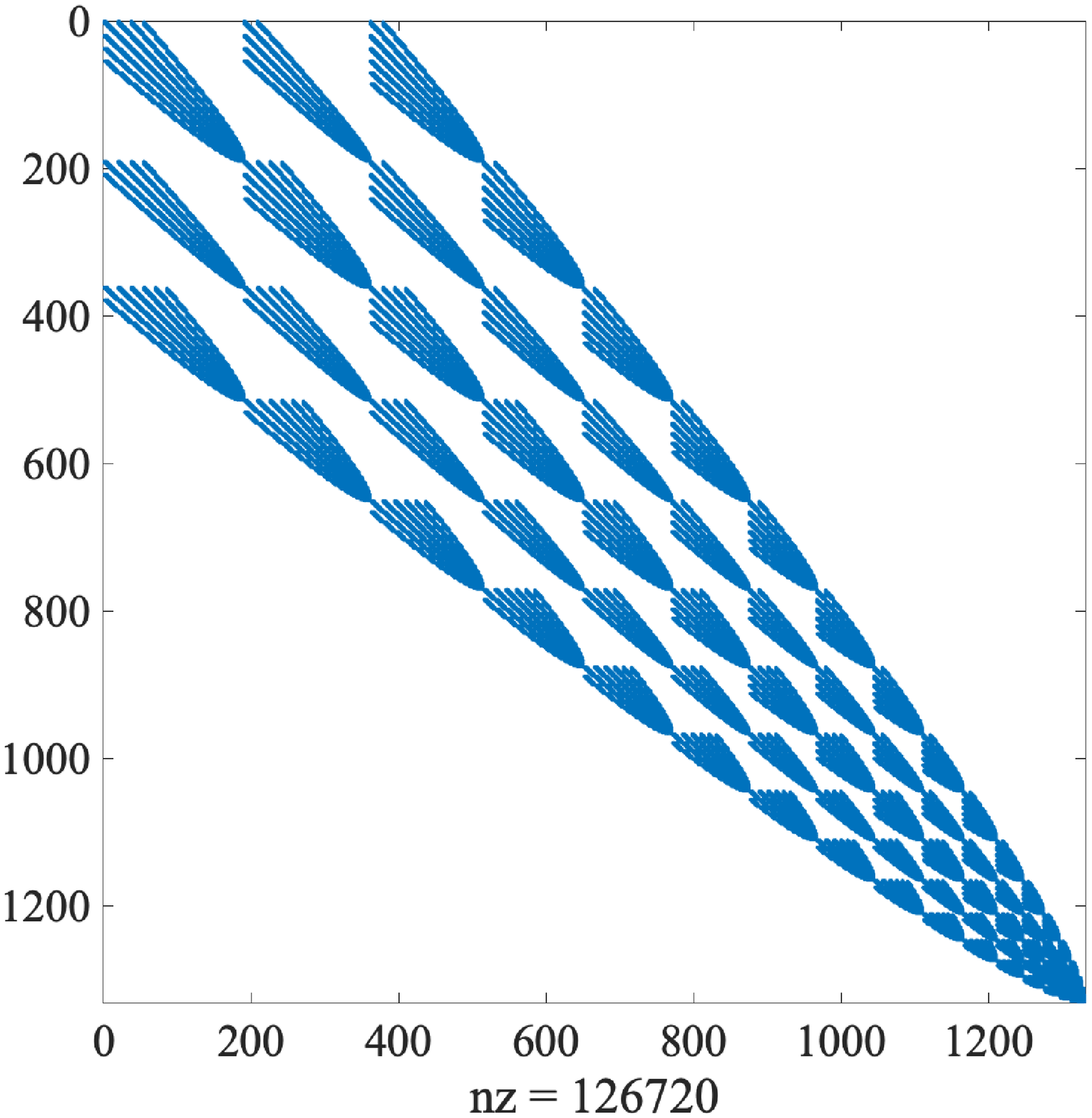}
\end{minipage}%
}%
\subfigure{
\begin{minipage}[t]{0.5\linewidth}
\centering
\includegraphics[width=2.4in]{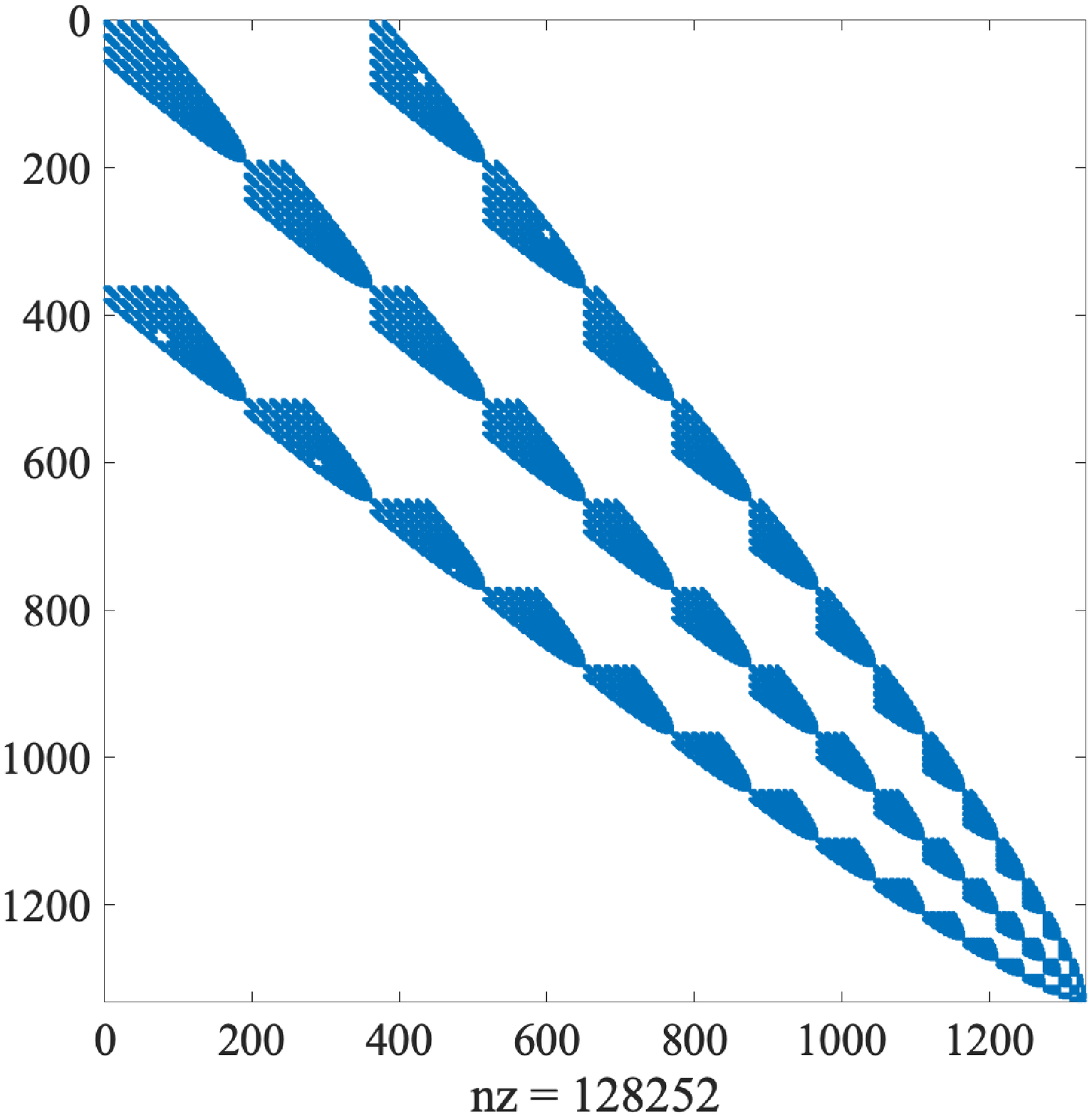}
\end{minipage}%
}%
\centering\vspace{-0.5em}
\caption{The sparse pattern of  $\mathcal{S}$ (left) and $\mathcal{M}$ (right) on the reference tetrahedron when $M=22$.}\label{exam1_strubd}
\end{figure}

\begin{remark}\label{basis_compare}
Condition numbers of the stiffness matrix and the mass matrix associated with different interior bases are quite different.
Without preconditioning, condition numbers of matrices generated by basis functions based on integrated Jacobi polynomials \cite{intlegendre2007} and by the ones proposed by Sherwin and Karniadakis \cite{KS1995} grow as asymptotically as $\mathcal{O}(M^{10})$ and $\mathcal{O}(M^7)$, respectively.
While the condition number of the stiffness matrix induced by our spectral-Galerkin method only grows in $\mathcal{O}(M^4),$ which shares the same order with that of the diagonally preconditioned matrix as Figure \ref{cond_basis} indicated.

\begin{figure}[H]
\centering
\subfigure{
\begin{minipage}[t]{0.3\linewidth}
\centering
\includegraphics[width=1.8in,height=1.6in]{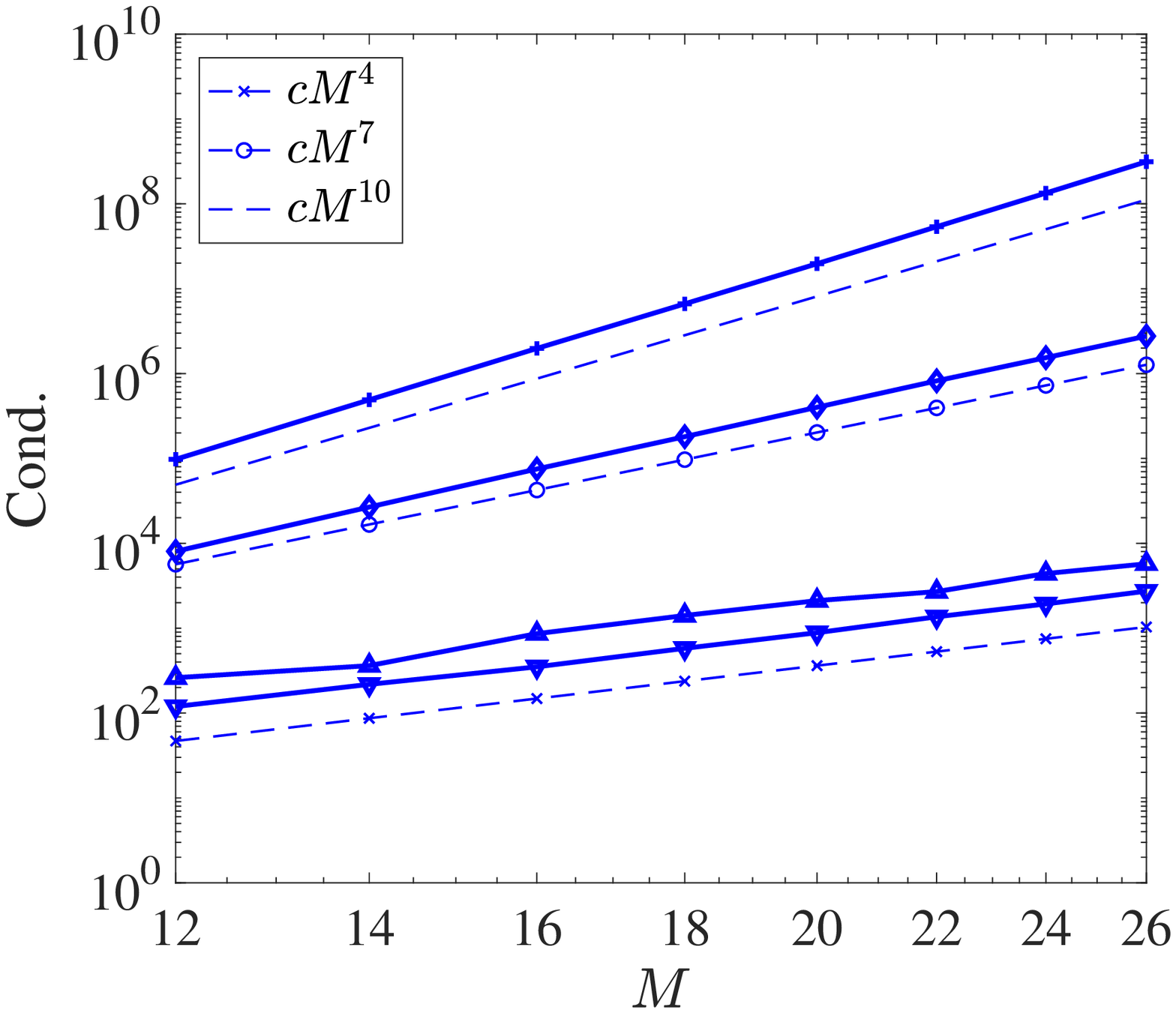}
\end{minipage}%
}%
\subfigure{
\begin{minipage}[t]{0.3\linewidth}
\centering
\includegraphics[width=1.8in,height=1.6in]{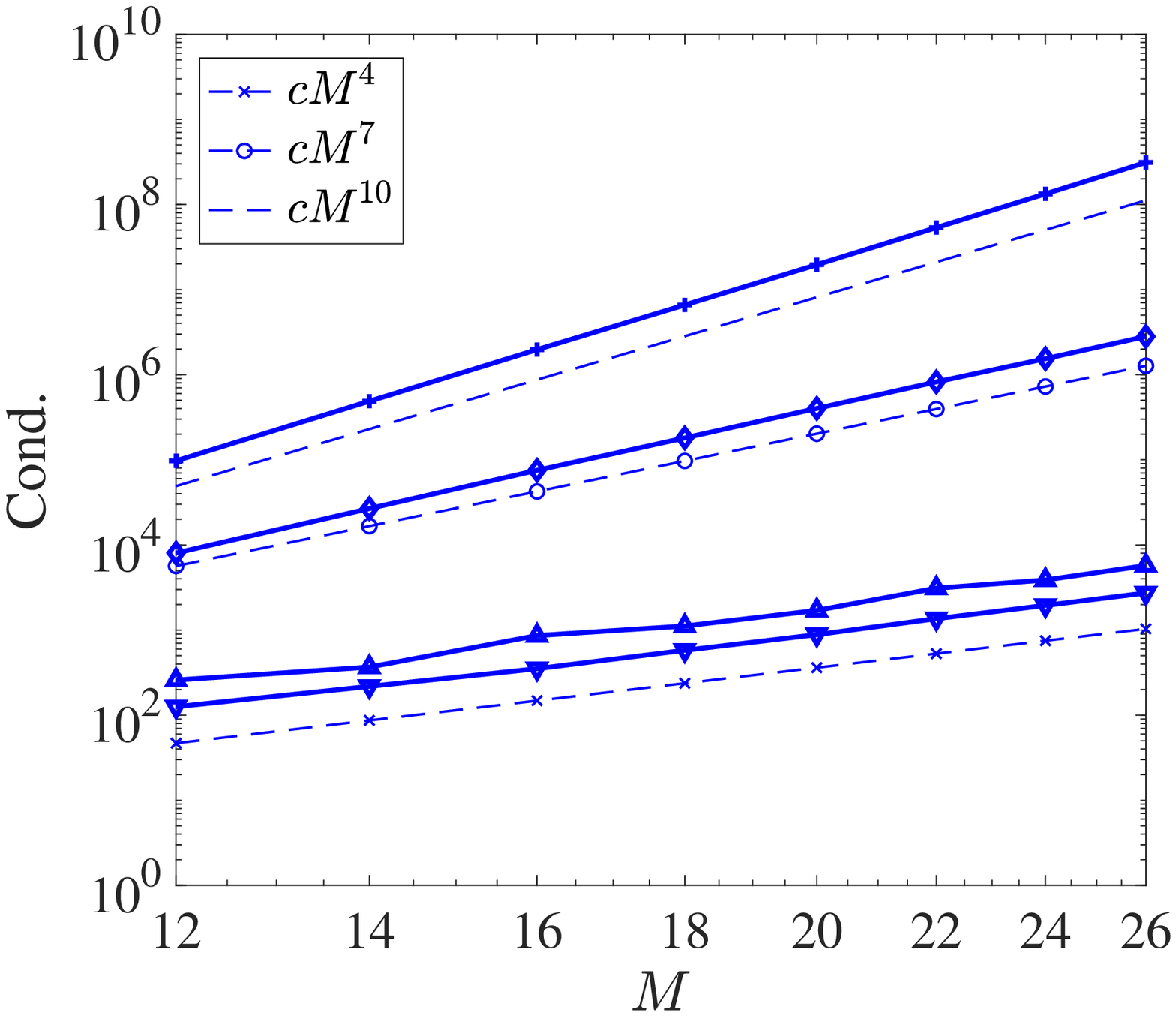}
\end{minipage}%
}%
\subfigure{
\begin{minipage}[t]{0.3\linewidth}
\centering
\includegraphics[width=1.8in,height=1.6in]{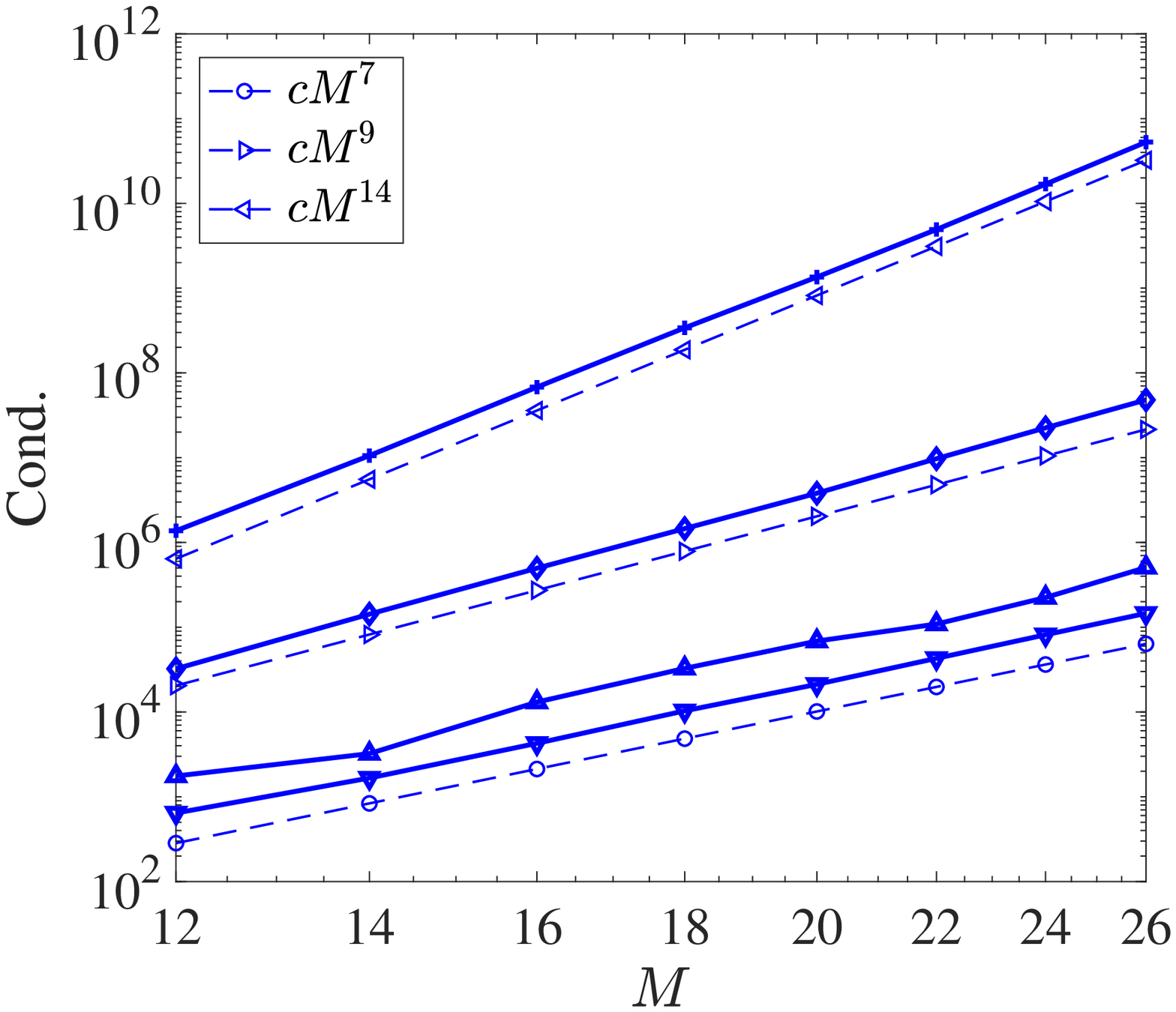}
\end{minipage}%
}%
\centering
\caption{Condition numbers of ${\mathcal{S}}+{\mathcal{M}}$ (left), ${\mathcal{S}}$ (middle) and ${\mathcal{M}}$ (right) associated with different interior basis functions against $M$: 
  $\JJ_{\bell}$ (marked by $``\bigtriangleup"$);
  diagonally preconditioned  (marked by $``\bigtriangledown"$);
basis functions associated with integrated Jacobi polynomials  (marked by $``+"$);
basis functions proposed by Sherwin and Karniadakis (marked by $``\diamond"$). 
}\label{cond_basis}
\end{figure}
\end{remark}

\subsubsection{Stiffness matrix assembling}
To assemble the stiffness matrix, we need to evaluate the integral $(\nabla \varphi_{\bell}, \nabla \varphi_{\bm{k}})_{\TT}$. 
Indeed, the linear mapping $\Psi$ defined in \eqref{mapping} has the following explicit form:
\begin{equation}\label{mappingx}
\bm{x} = \Psi(\bm{\hx}):  = 
\bm{x}^{(0)} (1-\hx_1-\hx_2-\hx_3)  + \bm{x}^{(1)} \hx_1 + \bm{x}^{(2)} \hx_2 +\bm{x}^{(3)} \hx_3.
\end{equation}
It is straightforward by the chain rule in calculus that
\begin{equation*}
\begin{aligned}
\Delta &= \sum\limits_{i=1}^3 \partial_{x_i}^2 = \sum\limits_{i=1}^3 \left( \sum\limits_{j=1}^3 \frac{\partial_{\hx_j}}{\partial_{x_i}} \partial_{\hx_j}\right)^2 \\
&= \sum\limits_{j=1}^3\left( \nabla \hx_j \cdot \nabla \hx_j\right) \partial_{\hx_j}^2 + 2\sum\limits_{1\le j<k\le 3} \left( \nabla \hx_j\cdot \nabla\hx_k\right) \partial_{\hx_j} \partial_{\hx_k}\\
&= \sum\limits_{j=1}^3\left( \nabla \hx_j \cdot \nabla \hx_j\right) \partial_{\hx_j}^2 + \sum\limits_{1\le j<k\le 3}
\left(\nabla \hx_j\cdot \nabla \hx_k \right) \left( \partial^2_{\hx_j} + \partial^2_{\hx_k}\right)
-\sum\limits_{1\le j<k\le 3} \left(\nabla \hx_j\cdot \nabla \hx_k \right) \left( \partial_{\hx_k}-\partial_{\hx_j} \right)^2\\
&=\sum\limits_{j=1}^3 \left( \nabla \hx_j \cdot \nabla(\hx_1+\hx_2+\hx_3)\right) \partial_{\hx_j}^2
 -\sum\limits_{1\le j<k\le 3} \left(\nabla \hx_j\cdot \nabla \hx_k \right) \left( \partial_{\hx_k}-\partial_{\hx_j} \right)^2.
\end{aligned}
\end{equation*}

Combining with geometric interpretation of the cross product and the triple product,
it then follows from \eqref{mappingx} that
\begin{equation*}
\begin{cases}
\nabla \hx_1 = \frac{(\bm{x}^{(2)}-\bm{x}^{(0)}) \times ( \bm{x}^{(3)} - \bm{x}^{(0)})}{ (\bm{x}^{(1)}-\bm{x}^{(0)}, \bm{x}^{(2)}-\bm{x}^{(0)}, \bm{x}^{(3)}-\bm{x}^{(0)}) }=
-\frac{|F_1|}{3|\TT|} \bm{n}_1\\[0.5em]
\nabla \hx_2 = \frac{(\bm{x}^{(3)}-\bm{x}^{(0)}) \times ( \bm{x}^{(1)} - \bm{x}^{(0)})}{ (\bm{x}^{(1)}-\bm{x}^{(0)}, \bm{x}^{(2)}-\bm{x}^{(0)}, \bm{x}^{(3)}-\bm{x}^{(0)}) }=-\frac{|F_2|}{3|\TT|} \bm{n}_2
\\[0.5em]
\nabla \hx_3 = \frac{(\bm{x}^{(1)}-\bm{x}^{(0)}) \times ( \bm{x}^{(2)} - \bm{x}^{(0)})}{ (\bm{x}^{(1)}-\bm{x}^{(0)}, \bm{x}^{(2)}-\bm{x}^{(0)}, \bm{x}^{(3)}-\bm{x}^{(0)}) }=
-\frac{|F_3|}{3|\TT|} \bm{n}_3,\\[0.5em]
\nabla \left( \hx_1+\hx_2+\hx_3\right)=
\frac{(\bm{x}^{(2)}-\bm{x}^{(1)})\times (\bm{x}^{(3)}-\bm{x}^{(1)})}{(\bm{x}^{(1)}-\bm{x}^{(0)}, \bm{x}^{(2)}-\bm{x}^{(0)}, \bm{x}^{(3)}-\bm{x}^{(0)})} = \frac{|F_0|}{3|\TT|} \bm{n}_0,
\end{cases}
\end{equation*}
where  $F_j$ denotes the face opposite to the vertex $P_j$ and $\bm{n}_j$ is the outward normal vector of the face $F_j$ on the tetrahedron $\TT$ for $0\le j\le 3$;
$|\TT|$ and $|F_j|$ stand for the volume of $\TT$ and the area of $F_j$, respectively.
Further let $\langle F_j, F_k\rangle$ be the dihedral angle of the face $F_j$ and $F_k$. Then
\begin{equation*}
\begin{aligned}
&\nabla \hx_j \cdot \nabla(\hx_1+\hx_2+\hx_3) = \frac{|F_0| |F_j|}{9|\TT|^2} \cos \langle F_0, F_j \rangle,\quad 1\le j\le3,\\
&\nabla \hx_j \cdot \nabla \hx_k = -\frac{|F_j| |F_k|}{9|\TT|^2} \cos \langle F_j, F_k \rangle,\quad 1\le j < k\le 3,
\end{aligned}
\end{equation*}
since $\bm{n}_j\cdot \bm{n}_k = -\cos \langle F_j, F_k \rangle$ when $j \neq k.$
Thus, the elements in the stiffness matrix are evaluated by
\begin{equation*}
\begin{aligned}
\left( \nabla {\varphi}_{\bell}, \nabla {\varphi}_{\bm{k}}\right)_\TT=
& -\frac{2}{3|\TT|} \,\bigg[ \sum\limits_{j=1}^3 
|F_0| |F_j|\cos \langle F_0, F_j \rangle   \left( \partial_{\hx_j} \hat{\varphi}_{\bell}, \partial_{\hx_j} \hat{\varphi}_{\bm{k}} \right)_{\hat{\TT}}\\
&\qquad\qquad+ \sum\limits_{1\le j<k\le3}
|F_j| |F_k|\cos \langle F_j, F_k \rangle 
 \left( (\partial_{\hx_k}-\partial_{\hx_j}) \hat{\varphi}_{\bell},
 (\partial_{\hx_k}-\partial_{\hx_j})  \hat{\varphi}_{\bm{k}}\right)_{\hat{\TT}}\bigg].
 \end{aligned}
\end{equation*}
According to Lemma \ref{lemma_diff} and Lemma \ref{lemma_increase}, each derivative on the reference tetrahedron $\hat{\TT}$ is exactly a finite series of Koornwinder-Dubiner polynomials,  which allows us to evaluate the accurate matrix entries by the orthogonality.

\subsubsection{Mass matrix assembling}
When $\gamma$ is a constant, the entries of the mass matrix could be evaluated by
\begin{equation}\label{mass_entry}
\left( {\varphi}_{\bell}, {\varphi}_{\bm{k}}\right)_\TT = 6\gamma |\TT| \left( \hat{\varphi}_{\bell},  \hat{\varphi}_{\bm{k}}\right)_{\hat{\TT}}.
\end{equation}
Again, each $\hat{\varphi}_{\bell}$ is a finite expansion of Koornwinder-Dubiner polynomials based on Lemma \ref{lemma_increase} so that the integration in \eqref{mass_entry} could be evaluated exactly via the orthogonality.

However, when $\gamma=\gamma(x)$ is a variable coefficient, the cost in order to obtain ${\mathcal{M}}_\gamma$ is $\mathcal{O}(M^9)$ by using the qualified numerical quadrature.
In this subsection, we shall assemble the  matrix ${\mathcal{M}}_\gamma$ 
associated with a variable coefficient recursively by making use of the three-term recurrence relation \eqref{3recurrence} to reduce the order of complexity to $\mathcal{O}(M^6)$.

We first rearrange the basis polynomials $\{{\varphi}_{\bell}\}$ with respect to the total degree in $\Phi_M$ where
\begin{equation*}
\Phi_M=\left(
\begin{array}{c}
\bm{\varphi}_4 \\
\bm{\varphi}_5\\
\vdots\\
\bm{\varphi}_M
\end{array}\right),\quad{\text{with}}\,\,
\bm{\varphi}_m = \left(
\begin{array}{c}
\bm{\varphi}_{m,2}\\
\bm{\varphi}_{m,3}\\
\vdots\\
\bm{\varphi}_{m,{m-2}},
\end{array}\right),\quad 
\bm{\varphi}_{m,{k}} = \left(
\begin{array}{c}
\varphi_{k,1,m-k-1}\\
\varphi_{k,2,m-k-2}\\
\vdots\\
\varphi_{k,m-k-1,1}
\end{array}
\right),\quad
\begin{array}{l}
2\le k\le m-2,\\
4\le m\le M.
\end{array}
\end{equation*}
Then,  the matrix in a block form 
\begin{equation}\label{block_S}
\int_{\TT} \gamma(\bs x)
\Phi_M(\bs x)   \Phi_M(\bs x) ^{\tr}  d{\bs x} 
=\left(
\begin{array}{ccc}
H_{4,4} & \cdots  &H_{4,M}\\
\vdots & \ddots & \vdots\\
H_{M,4} &\cdots  &H_{M,M}
\end{array}\right),
\end{equation} 
with 
\begin{align*}
H_{m,k} = \int_\TT \gamma(\bs{x}) 
{\bm{\varphi}}_m(\bm{x})  
{{\bm{\varphi}}}_k(\bm{x})^\tr d\bm{x} ,\quad 4\le m,k \le M,
\end{align*}
could be regarded as a rearrangement of the rows and columns in the matrix ${\mathcal{M}}_\gamma$.

For convenience, all coefficient matrices in the three-term recurrence relation \eqref{3recurrence} and the generalized inverse $D_m$ are equally partitioned into three blocks,
\begin{equation*}
A_m=\left(
\begin{array}{c}
A_m^1\\
A_m^2\\
A_m^3
\end{array}\right),\quad
B_m(\bm{x})+\left(
\begin{array}{c}
x_1 I\\
x_2 I\\
x_3 I
\end{array}
\right)=
 \left(
\begin{array}{c}
B_m^1\\
B_m^2\\
B_m^3
\end{array}\right),\quad
C_m=\left(
\begin{array}{c}
C_m^1\\
C_m^2\\
C_m^3
\end{array}\right),\quad 
D_m=\left(D_m^1,D_m^2,D_m^3\right).
\end{equation*}
For any integer $1\le i\le 3$, it is straightforward 
to obtain that
\begin{equation*}
\begin{aligned}
\int_\TT \gamma(\bm{x}) x_i   
{\bs \varphi}_m(\bm{x})
  {\bm{\varphi}}_{k}(\bm{x})^\tr d \bm{x}&= 
\int_\TT \gamma(\bs{x}) 
\left( C_{m}^i {{\bm{\varphi}}}_{m-1}(\bm{x}) + 
B_m^i {{\bm{\varphi}}}_m(\bm{x})+ 
A_m^i {{\bm{\varphi}}}_{m+1}(\bm{x}) \right)
{{\bm{\varphi}}}_k(\bm{x})^\tr d\bm{x}\\
&= C_{m}^i  H_{m-1,k} + B_m^i H_{m,k} + A_m^i H_{m+1,k}\\
&=\int_\TT \gamma(\bs{x}) {{\bm{\varphi}}}_{m}(\bm{x})
\left(C_{k}^i {{\bm{\varphi}}}_{k-1}(\bm{x})+ B_k^i {{\bm{\varphi}}}_k(\bm{x}) + A_k^i {{\bm{\varphi}}}_{k+1}(\bm{x}) \right)^\tr d\bm{x}\\
&=H_{m,k-1} {C_{k}^{i\,\tr}} + H_{m,k} B_k^{i\,\tr} + H_{m,k+1} A_k^{i\,\tr}.
\end{aligned}
\end{equation*}
As a result, it holds that
\begin{equation}\label{eq_recu}
A_m^i H_{m+1,k} = H_{m,k-1} {C_{k}^{i\,\tr}} + H_{m,k} B_k^{i\,\tr} + H_{m,k+1} A_k^{i\,\tr} -B_m^i H_{m,k} - C_{m}^i  H_{m-1,k}.
\end{equation}
Equivalently,
one has
\begin{equation*}\label{eq_recu_mat}
\begin{aligned}
\left(
\begin{array}{c}
A_m^1\\[0.2em]
A_m^2\\[0.2em]
A_m^3
\end{array}\right) H_{m+1,k}& = 
\left(
\begin{array}{c}
 H_{m,k-1}C_{k}^{1\,\tr}\\[0.2em]
 H_{m,k-1}C_{k}^{2\,\tr}\\[0.2em]
 H_{m,k-1}C_{k}^{3\,\tr}
\end{array}\right)
+ 
\left(
\begin{array}{c}
 H_{m,k}B_{k}^{1\,\tr}\\[0.2em]
 H_{m,k}B_{k}^{2\,\tr}\\[0.2em]
 H_{m,k}B_{k}^{3\,\tr}
\end{array}\right)+
\left(
\begin{array}{c}
 H_{m,k+1}A_{k}^{1\,\tr}\\[0.2em]
 H_{m,k+1}A_{k}^{2\,\tr}\\[0.2em]
 H_{m,k+1}A_{k}^{3\,\tr}
\end{array}\right)\\[0.3em]
&\qquad\quad
-
\left(
\begin{array}{c}
B_{m}^{1}\\[0.2em]
B_{m}^{2}\\[0.2em]
B_{m}^{3}
\end{array}\right) H_{m,k}
-
\left(
\begin{array}{c}
C_{m}^{1}\\[0.2em]
C_{m}^{2}\\[0.2em]
C_{m}^{3}
\end{array}\right) H_{m-1,k}.
\end{aligned}
\end{equation*}
Further recalling that $D_mA_m=I,$ we arrive at
\begin{equation}\label{recu_mat}
\begin{aligned}
H_{m+1,k}& = 
\sum_{i=1}^3 
D_m^iH_{m,k-1}C_{k}^{i\,\tr}
+ 
\sum_{i=1}^3 
D_m^iH_{m,k}B_{k}^{i\,\tr}
+ 
\sum_{i=1}^3 
D_m^iH_{m,k+1}A_{k}^{i\,\tr}
\\
&\,\, -
\sum_{i=1}^3 
D_m^i B_{m}^{i}  H_{m,k}-
\sum_{i=1}^3 
D_m^i C_{m}^{i}  H_{m-1,k}
\end{aligned}
\end{equation}
It indicates that the block $H_{m+1,k}$ is derived by other small matrices known in previous steps.
To obtain each block matrix in \eqref{block_S}, one first needs to compute small blocks
$$H_{4,k} = \int_\TT\gamma(\bs{x}) {{\bm{\varphi}}}_{4}(\bm{x}) {{\bm{\varphi}}}_{k}(\bm{x})^\tr d\bm{x},\quad 4\le k\le 2M-4,$$
where ${{\bm{\varphi}}}_{4}$ only contains the basis function $\varphi_{2,1,1}$.
As $H_{m,k} = H_{k,m}^\tr$ is symmetric,
one then follows \eqref{recu_mat} to derive the blocks
$$H_{m+1,k},\quad m+1\le k \le 2M-m-1,\,\, 4\le m\le M-1.$$
With these blocks arranged as \eqref{block_S} defines,
  ${{\mathcal{M}}}_\gamma$ is consequently derived after a rearrangement.


\section{Numerical experiments}\label{numerical_exper}
To illustrate the validation of our spectral-Galerkin approximation scheme, we carry out some numerical experiments in this section.

\subsection{Numerical examples for source problems}

We shall present some numerical results for source problems in this subsection.

\begin{example}\label{exam1}
Consider the second-order model equation subject to the  homogeneous Dirichlet boundary condition:
\begin{equation}\label{exam1_equ}
\begin{cases}
-\Delta u(\bm{x}) +u(\bm{x})=f(\bm{x}), \quad &  \bm{x}\in\hat{\TT},\\
 u(\bm{x})=0,\quad & \bm{x}\in\partial\hat{\TT},
\end{cases}
\end{equation} 
with the exact solution
\begin{equation}\label{u_exact}
u(\bs{x}) = \sin\frac{\pi x_1}{2} \sin\frac{\pi x_2}{2} \sin\frac{\pi x_3}{2} \sin\frac{\pi (1-x_1-x_2-x_3)}{2}.
\end{equation}
\end{example}

\begin{figure}[H]
\centering
\subfigure{
\begin{minipage}[t]{0.5\linewidth}
\centering
\includegraphics[width=2.3in]{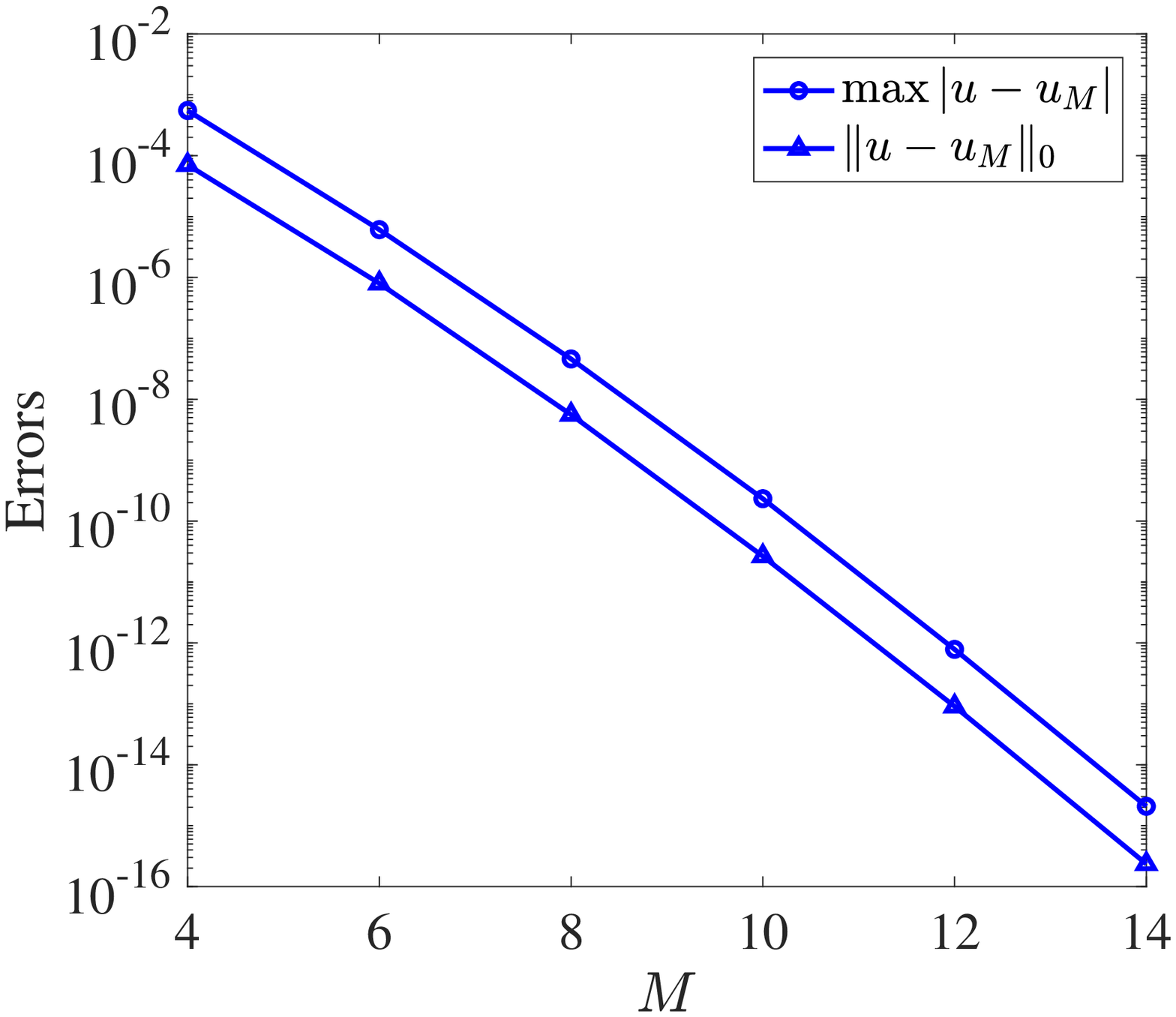}
\end{minipage}%
}%
\subfigure{
\begin{minipage}[t]{0.5\linewidth}
\centering
\includegraphics[width=2.3in]{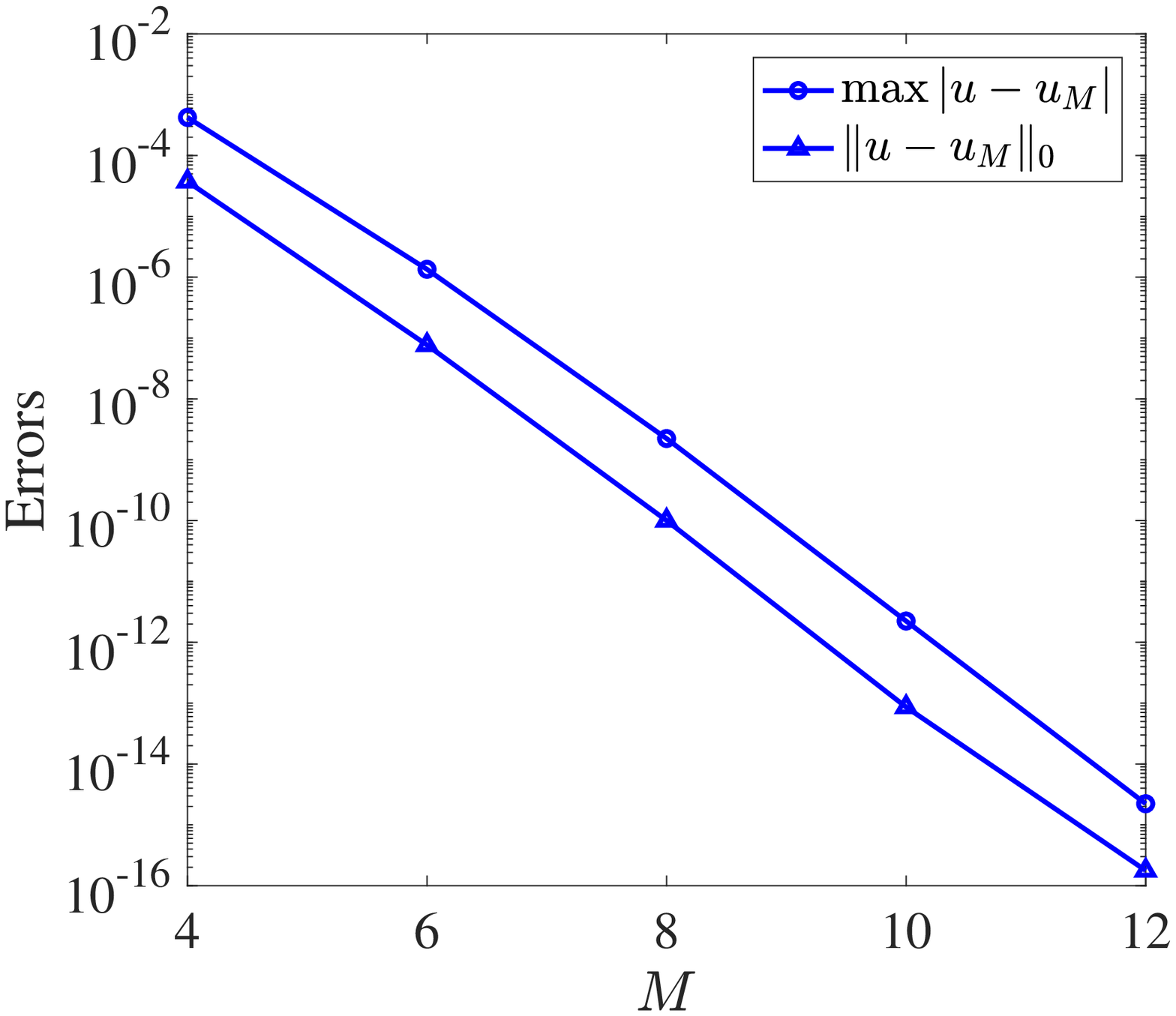}
\end{minipage}%
}%
\centering\vspace{-0.5em}
\caption{Maximum pointwise errors and $L^2$-errors  against $M$ in Example \ref{exam1} (left) and Example \ref{exam2} (right). }\label{exam12_error}
\end{figure}

Owing to the homogeneous Dirichlet boundary, only interior modes of polynomial basis functions are involved in our numerical scheme.
It is observed in the left of Figure \ref{exam12_error} that both maximum pointwise errors and $L^2$-errors of $u-u_{M}$ decay exponentially, which verifies the effectiveness and spectral accuracy of our spectral-Galerkin method.

\begin{example}\label{exam2}
Consider the Poisson equation subject to the non-homogeneous Dirichlet boundary condition:
\begin{equation}\label{exam2_equ}
\begin{cases}
-\Delta u(\bm{x}) = f(\bm{x}),\quad & \bm{x}\in\hat{\TT},\\
u(\bm{x})=g(\bm{x}),\quad &\bm{x}\in\partial\hat{\TT},
\end{cases}
\end{equation}
with the exact solution
$$u(\bs{x}) = (x_1+1)(x_2+1)(x_3+1)e^{1-x_1-x_2-x_3}.$$
\end{example}

It follows from the right of Figure \ref{exam12_error} that the numerical scheme \eqref{scheme_pde} achieves exponential orders of convergence for the second-order model problem with the non-homogeneous Dirichlet boundary condition, which confirms the spectral accuracy on the approximation of the solution along boundaries.

\begin{example}\label{exam3}
Consider the second-order model equation subject to the homogeneous Dirichlet boundary condition:
\begin{equation}\label{exam3_equ}
\begin{cases}
-\Delta u(\bm{x}) +\gamma(\bs{x})u(\bm{x})=f(\bm{x}), \quad &  \bm{x}\in\hat{\TT},\\
 u(\bm{x})=0,\quad & \bm{x}\in\partial\hat{\TT},
\end{cases}
\end{equation} 
with the exact solution defined as in \eqref{u_exact} and a variable coefficient
$$\gamma(\bs{x}) = e^{x_1+x_2+x_3+1}.$$

\end{example}

Exponential orders of convergence of errors  $u-u_M$ in the maximum pointwise errors and $L^2$-errors  are also observed from the semi-log graph in  Figure \ref{exam3_figure}. This reflect the effectiveness of our method for solving equations with  non-homogeneous 
boundary conditions.

\begin{figure}[H]
\centering
\includegraphics[height = 1.9in, width=2.3in]{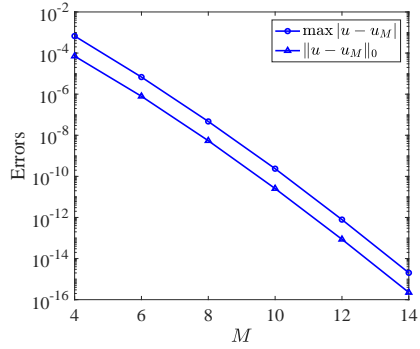}
\centering\vspace{-0.5em}
\caption{Maximum pointwise errors and $L^2$-errors against $M$  in Example \ref{exam3}.}\label{exam3_figure}
\end{figure}

\begin{example}\label{exam4}
Consider the heat equation subject to the homogeneous  Dirichlet boundary condition:
\begin{equation}\label{heat}
\begin{cases}
\partial_t u(\bm{x},t)-\Delta u(\bm{x},t) = f(\bm{x},t),\quad &(\bm{x},t)\in \hat{\TT}\times (0,T],\\
u(\bm{x},t) = 0, \quad &(\bm{x},t)\in\partial\hat{\TT}\times(0,T],\\
u(\bm{x},0) = u_0(\bm{x}),\quad &\bm{x}\in\hat{\TT},
\end{cases}
\end{equation}
with $T=1$ and the exact solution
$$u(\bs{x},t) = \sin \pi x_1 \sin \pi x_2 \sin \pi x_3 \sin \pi (1-x_1-x_2-x_3)e^{-t}.$$
\end{example}

We use the Crank-Nicolson method \cite{CQHZ_spectral_single} to design the fully discretization scheme.
Let 
$$0=t_0 < t_1 < \cdots < t_N = T,\quad N\in\NN,$$
be the discrete partition in time.
We further let  $\Delta t$ be the time-step and $t_n = n\Delta t$ be the $n$-th time-level.
The values of the approximation solution $u_M$ and the right-hand side function $f$ at time-step $n$ are denoted by $u_M^n$ and $f^n,$ respectively.
Combining with \eqref{scheme_pde},
the fully discretization scheme of \eqref{heat} reads: for all $0\le n\le N,$ to find $u_M^n\in X_{M,0}$ such that
\begin{equation}\label{fully_scheme}
\left( \frac{u^{n+1}_M - u^{n}_M}{\Delta t}, v_M\right)_{\hat{\TT}} +  a\left( \frac{u_M^{n+1}+u_M^{n}}{2}, v_M \right)  = 
 \left( \frac{f^{n+1} + f^{n}}{2}, v_M \right)_{\hat{\TT}},\quad \forall v_M\in X_{M,0}.
\end{equation}

\begin{figure}[H]
\centering
\subfigure{
\begin{minipage}[t]{0.3\linewidth}
\centering
\includegraphics[width=1.8in,height=1.6in]{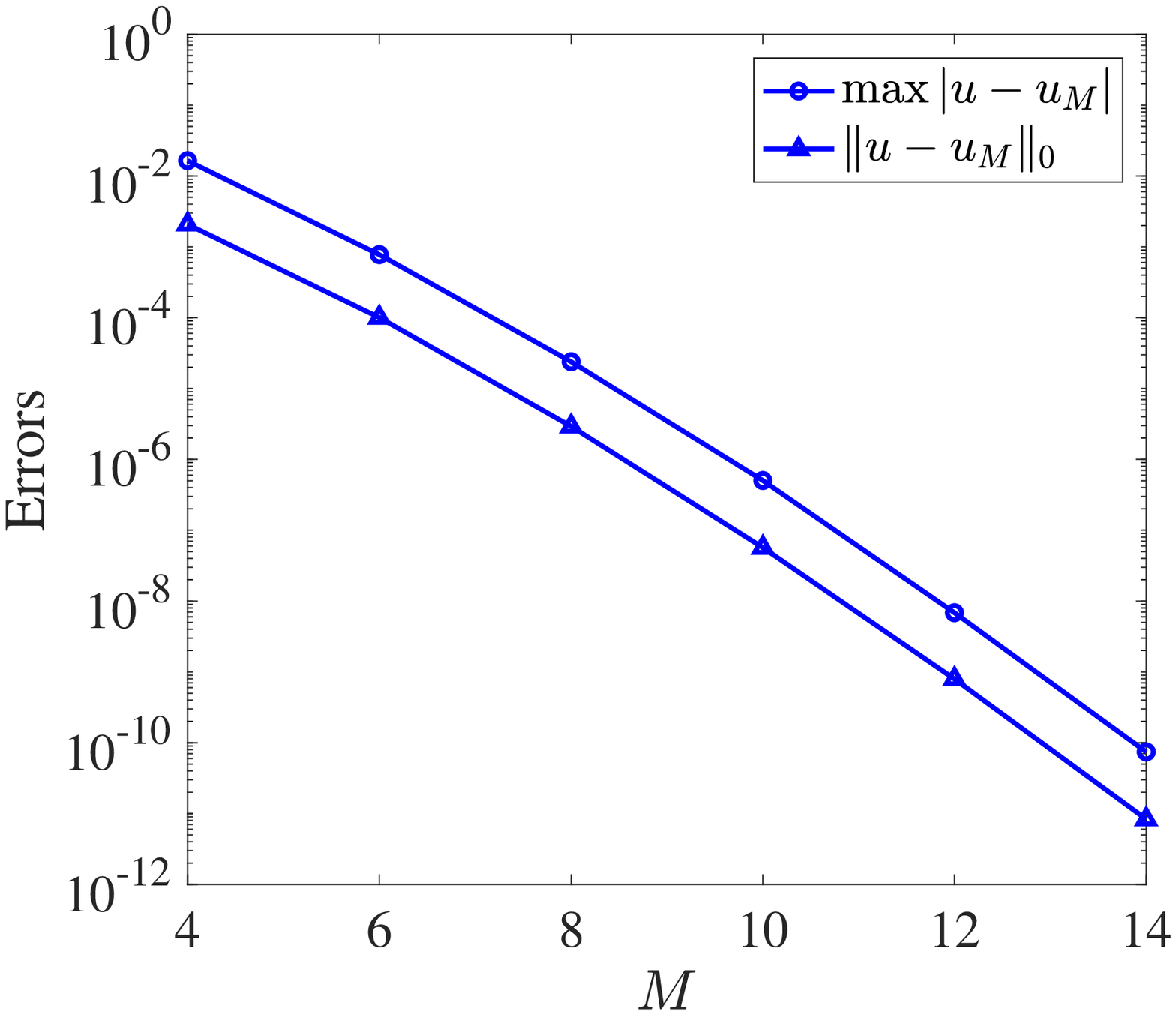}
\end{minipage}%
}%
\subfigure{
\begin{minipage}[t]{0.3\linewidth}
\centering
\includegraphics[width=1.8in,height=1.6in]{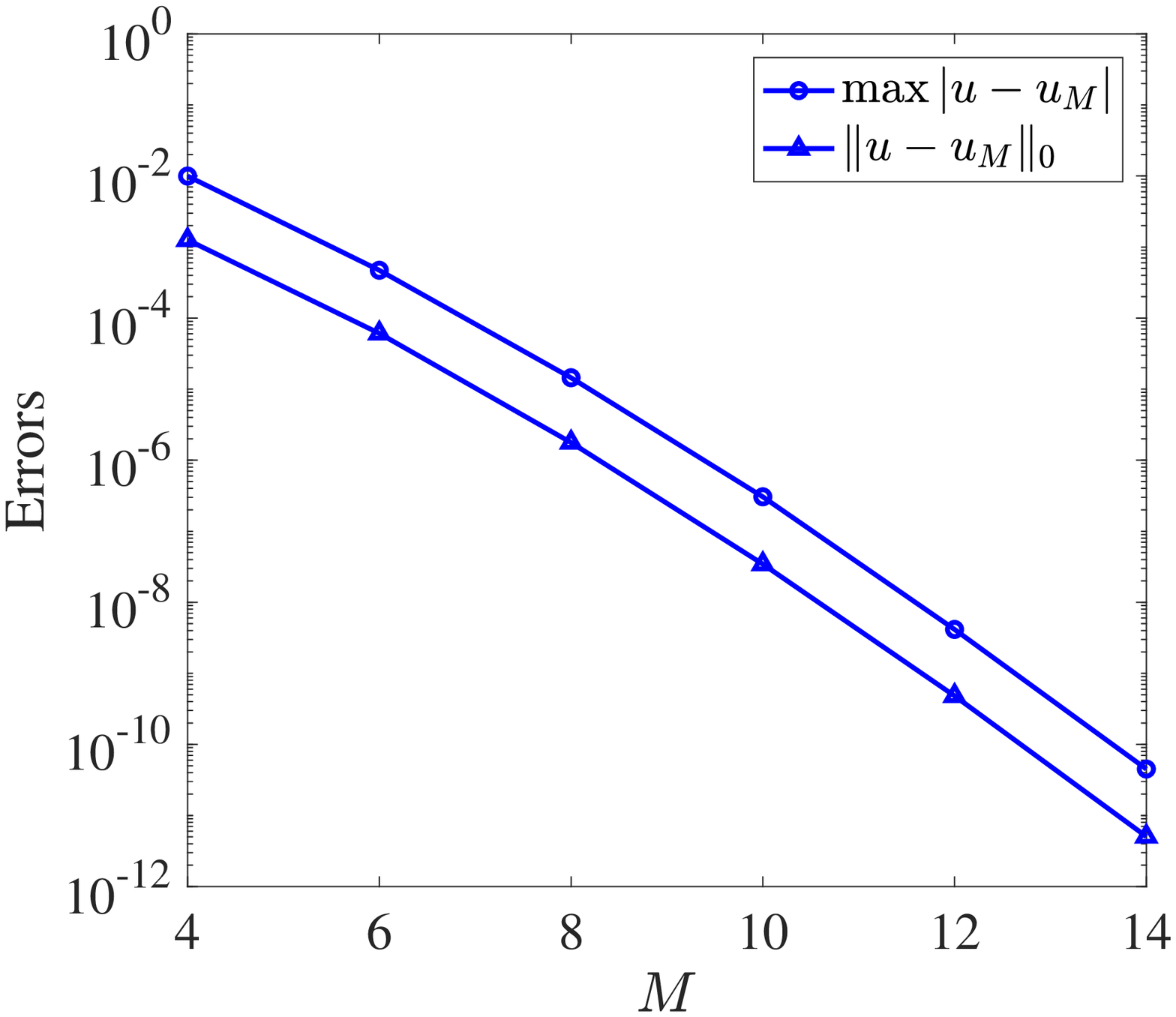}
\end{minipage}%
}%
\subfigure{
\begin{minipage}[t]{0.3\linewidth}
\centering
\includegraphics[width=1.8in,height=1.6in]{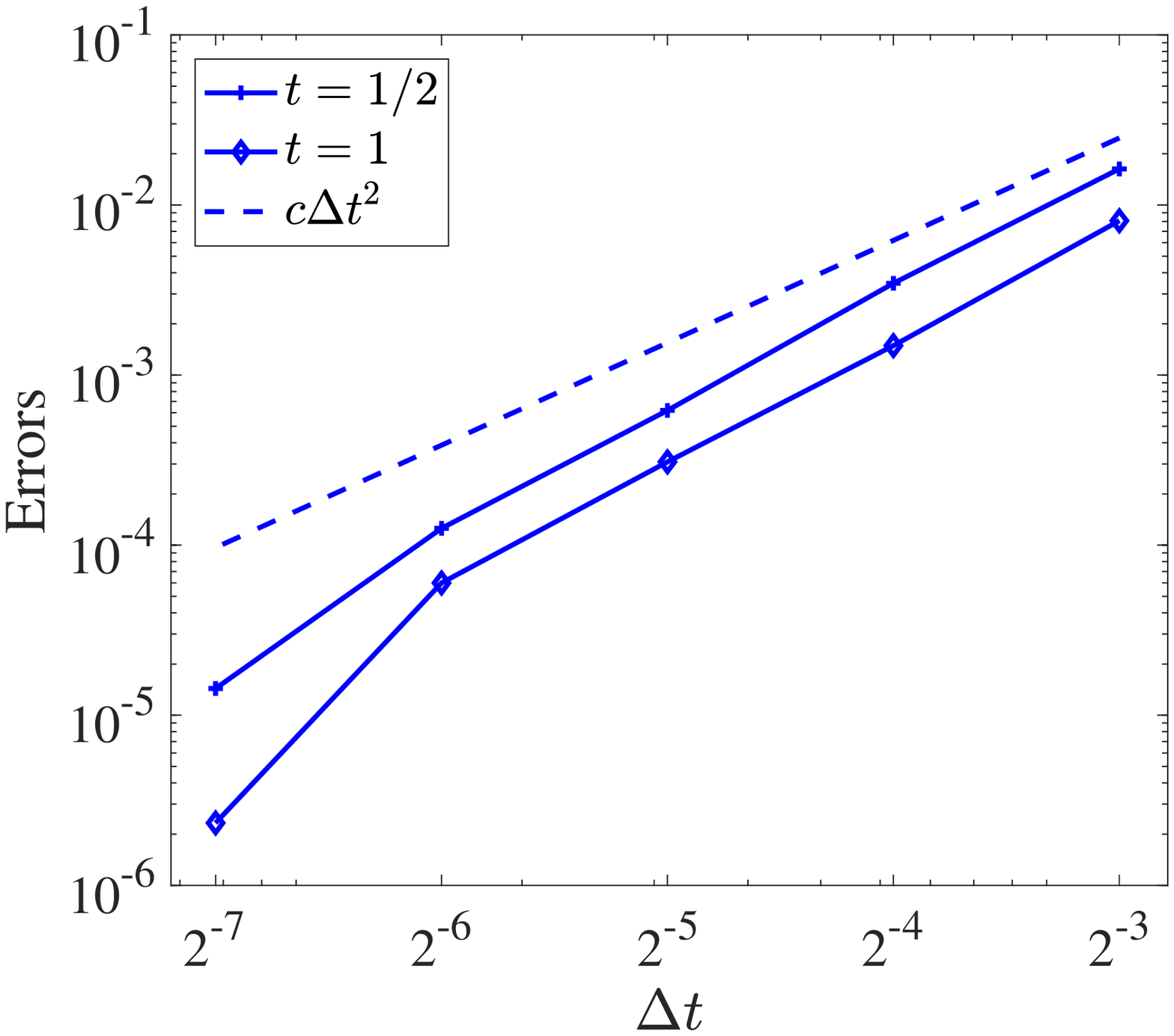}
\end{minipage}%
}%
\centering\vspace{-0.5em}
\caption{Maximum pointwise errors  and $L^2$-errors  against $M$ when  $ t = 1/2$ (left) and $t=1$ (middle); $L^2$-errors against $\Delta t$  (right) in Example \ref{exam4}.
}\label{exam4_error}
\end{figure}

Setting $\Delta t= 2^{-14}$,  
we  first demonstrate errors of $u-u_M$ against $M$ in different time  in Figure \ref{exam4_error}. 
It is reported that the errors decay exponentially both when $t = 1/2$ and $t=1$.
Letting $M=14,$  we also observe from the right of Figure \ref{exam4_error}  that the Crank-Nicolson scheme has second-order convergence in time.


\subsection{Numerical examples for eigenvalue problems}

We report the numerical results for the Laplacian eigenvalue problem \eqref{model_eigen} in this subsection.
Two special tetrahedra would be considered in the following discussions:
the fundamental tetrahedron 
$\TT_F$
with vertices  
$$P^F_0=\left(0,0,0\right)^\tr, \quad P^F_1=\left(0,0,1\right)^\tr, \quad P^F_2=\left(\frac{1}{2},\frac{1}{2},\frac{1}{2}\right)^\tr, \quad P^F_3= \left(-\frac{1}{2},\frac{1}{2},\frac{1}{2}\right)^\tr,$$
and the regular tetrahedron $\TT_R$ with vertices 
$$P^R_0=\left(0,0,\frac{\sqrt{6}}{3}\right)^\tr, \quad P^R_1=\left(\frac{\sqrt{3}}{3},0,0\right)^\tr, 
\quad
 P^R_2=\left(-\frac{\sqrt{3}}{6},\frac{1}{2},0\right)^\tr, \quad
  P^R_3= \left(-\frac{\sqrt{3}}{6},-\frac{1}{2},0\right)^\tr.$$
All eigenvalues of the homogeneous Dirichlet Laplacian can be arranged as 
 $$0<\mu_1 <\mu_2 \le \mu_3 \le \cdots \le \mu_k \le \cdots,\quad k\in \NN.$$

To begin with, we  test absolute errors when approximating the five smallest eigenvalues by numerical scheme \eqref{scheme_eigen} on two tetrahedra. 
For $\TT_F$, the exact eigenvalues are obtained in Appendix \ref{exact_TF}; while for $\TT_R$, the reference eigenvalues are derived with relatively large $M$
by our spectral-Galerkin method.
The semi-log and log-log graphs in Figure \ref{example5_error} reveal that the scheme achieves  exponential orders of convergence  on $\TT_F$ and algebraic orders of convergence on $\TT_R$, respectively.
 It means that the corresponding eigenfunctions on the regular tetrahedron would have singularities, 
which is quite different from the behaviors of eigenfunctions on the regular triangle. 
Indeed, the Laplacian eigenfunctions associated with the first few eigenvalues on the regular triangle
  are analytic  \cite{prager1998, McCartin2003} and the polynomial spectral method achieves exponential orders of convergence when approximating these eigen-solutions \cite{ShanLi2015}. 

\begin{figure}[H]
\centering
\subfigure{
\begin{minipage}[t]{0.5\linewidth}
\centering
\includegraphics[width=2.3in,height=2in]{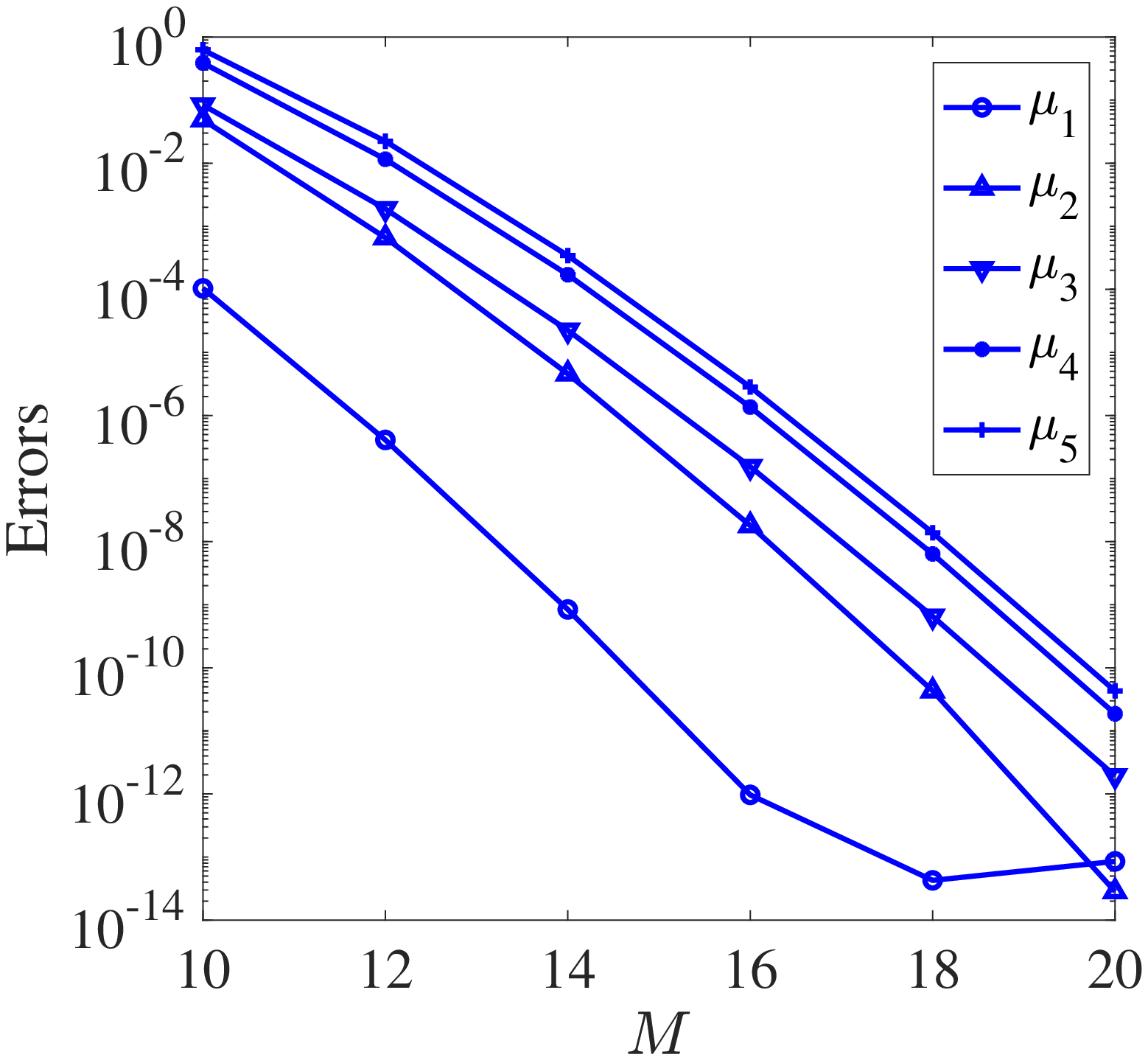}
\end{minipage}%
}%
\subfigure{
\begin{minipage}[t]{0.5\linewidth}
\centering
\includegraphics[width=2.3in,height=2in]{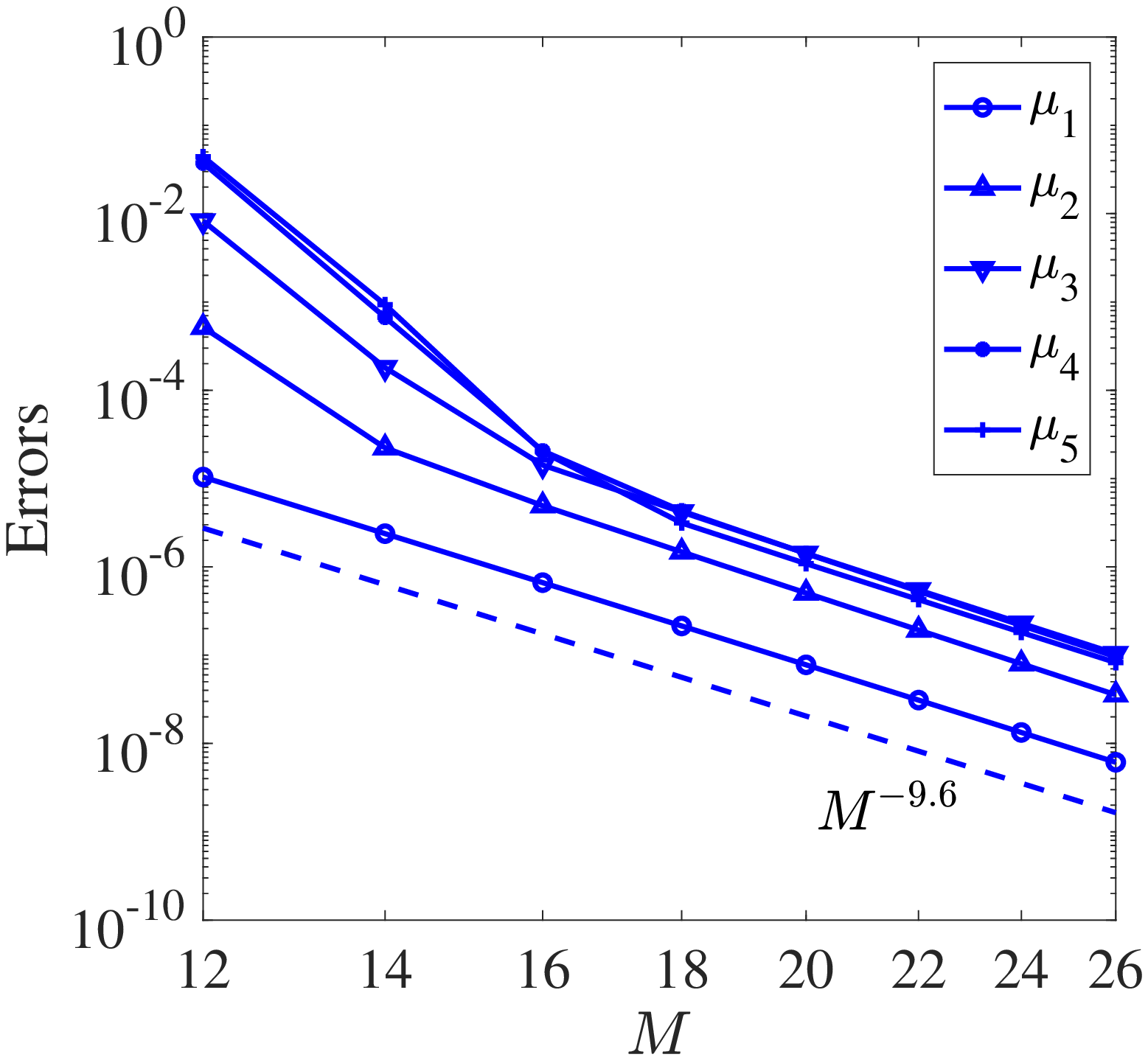}
\end{minipage}%
}%
\centering\vspace{-0.5em}
\caption{Absolute errors of the five smallest numerical Laplacian eigenvalues against $M$ on $\TT_F$ (left) and $\TT_R$ (right).}\label{example5_error}
\end{figure}

We then move on to study the approximations on large eigenvalues.
The Weyl's Conjecture in three dimensions \cite{Weyl1913, WeylV2016} reads that,
\begin{equation}\label{weyl_law}
\mu_k =  \frac{\pi (36\pi)^{\frac{1}{3}}}{ |\TT|^{\frac{2}{3}}}k^{\frac{2}{3}} +\frac{\pi}{2} \left( \frac{3\pi^2}{4}\right)^{\frac{1}{3}} \frac{|\partial \TT|}{ | \TT|^{\frac{4}{3}}} k^{\frac{1}{3}} + o(k^{\frac{1}{3}}) ,\quad k\rightarrow +\infty,
\end{equation}
where $|\TT|$ and $|\partial \TT|$ represent the volume and surface area of $\TT$, respectively.
Thus, the exact eigenvalue $\mu_k $, $k=\mathcal{O}(M^3)$,  grows in $\mathcal{O}(M^2)$ as $M$ tends to $\infty$. However,
 we observe in Figure \ref{example5_largest} that the largest numerical eigenvalue $ \mu_{M,\frac{(M-1)(M-2)(M-3)}{6}}$ evaluated by our spectral-Galerkin method with different $M$  grows almost as asymptotically as $\mathcal{O}(M^4)$.
 It indicates that the polynomial spectral method would bring out  a portion of spurious solutions in deriving large numerical eigenvalues. 
\begin{figure}[H]
\centering
\subfigure{
\begin{minipage}[t]{0.5\linewidth}
\centering
\includegraphics[width=2.2in,height=1.9in]{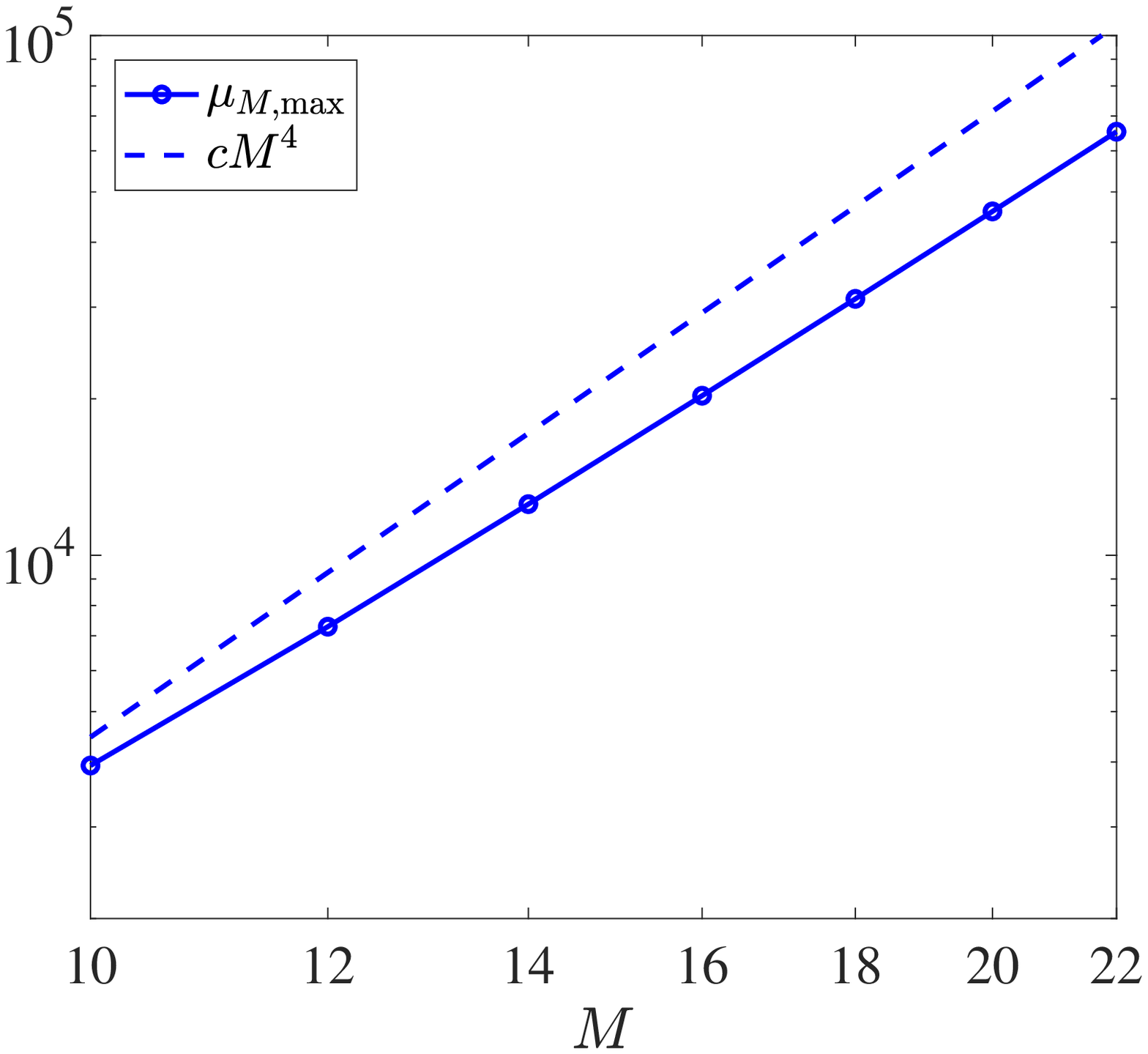}
\end{minipage}%
}%
\subfigure{
\begin{minipage}[t]{0.5\linewidth}
\centering
\includegraphics[width=2.2in,height=1.9in]{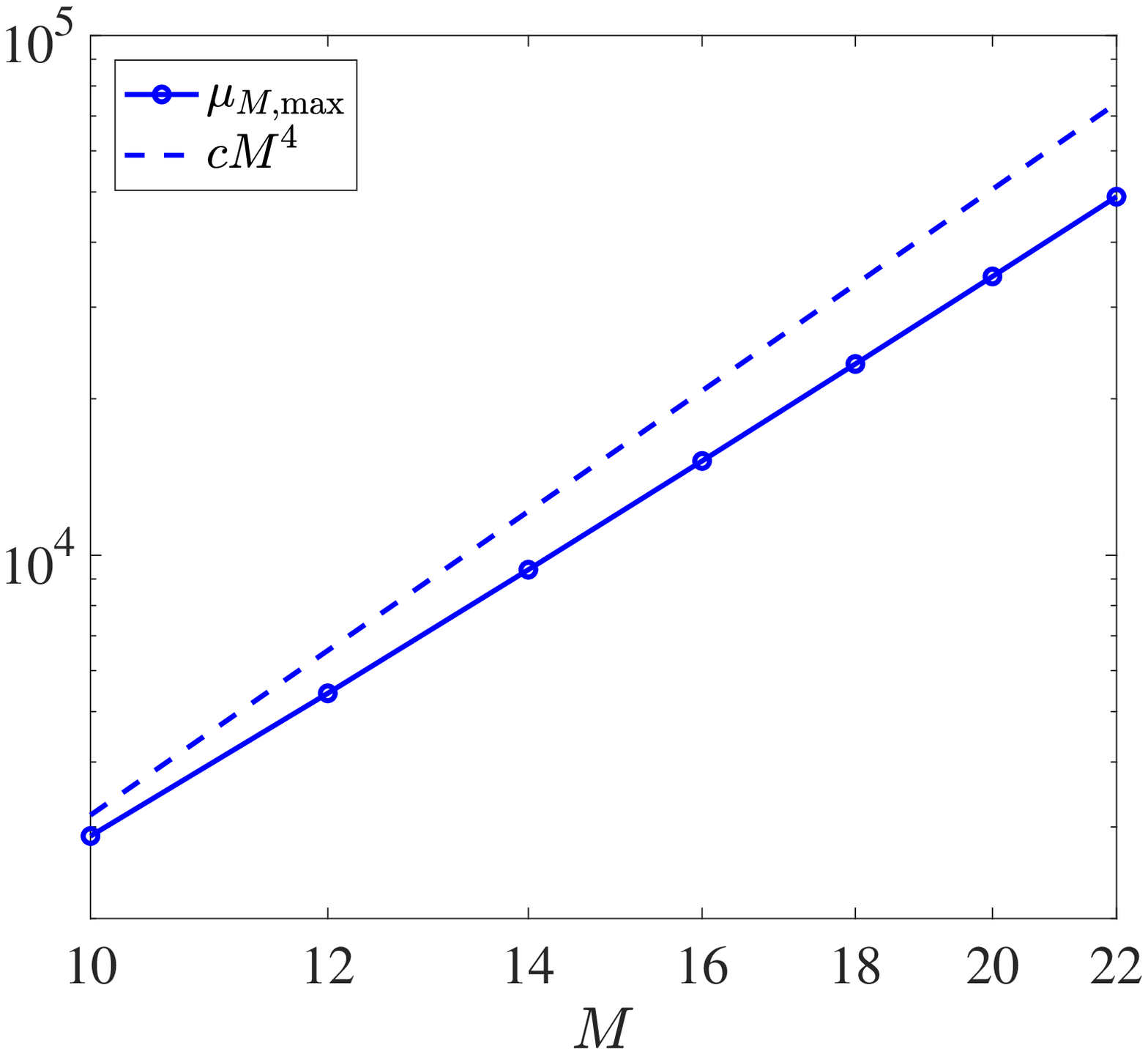}
\end{minipage}%
}%
\centering\vspace{-0.5em}
\caption{The largest numerical Laplacian eigenvalue against $M$  on  $\TT_F$ (left) and $\TT_R$ (right). }\label{example5_largest}
\end{figure}

As a result, one should check how many reliable eigen-solutions that our method is able to provide before examining asymptotic properties of large eigenvalues.
We understand  reliable to mean at least $\mathcal{O}(M^{-1})$ accuracy with polynomial degree $M$.
With exact eigenvalues known in Appendix \ref{exact_TF}, we solve the generalized eigenvalue problem \eqref{scheme_eigen} on $\TT_F$ with different values of $M$ and illustrate their relative errors in the left and the middle of Figure \ref{relative_errors}.

\begin{figure}[H]
\centering
\subfigure{
\begin{minipage}[t]{0.3\linewidth}
\centering
\includegraphics[width=1.8in,height=1.6in]{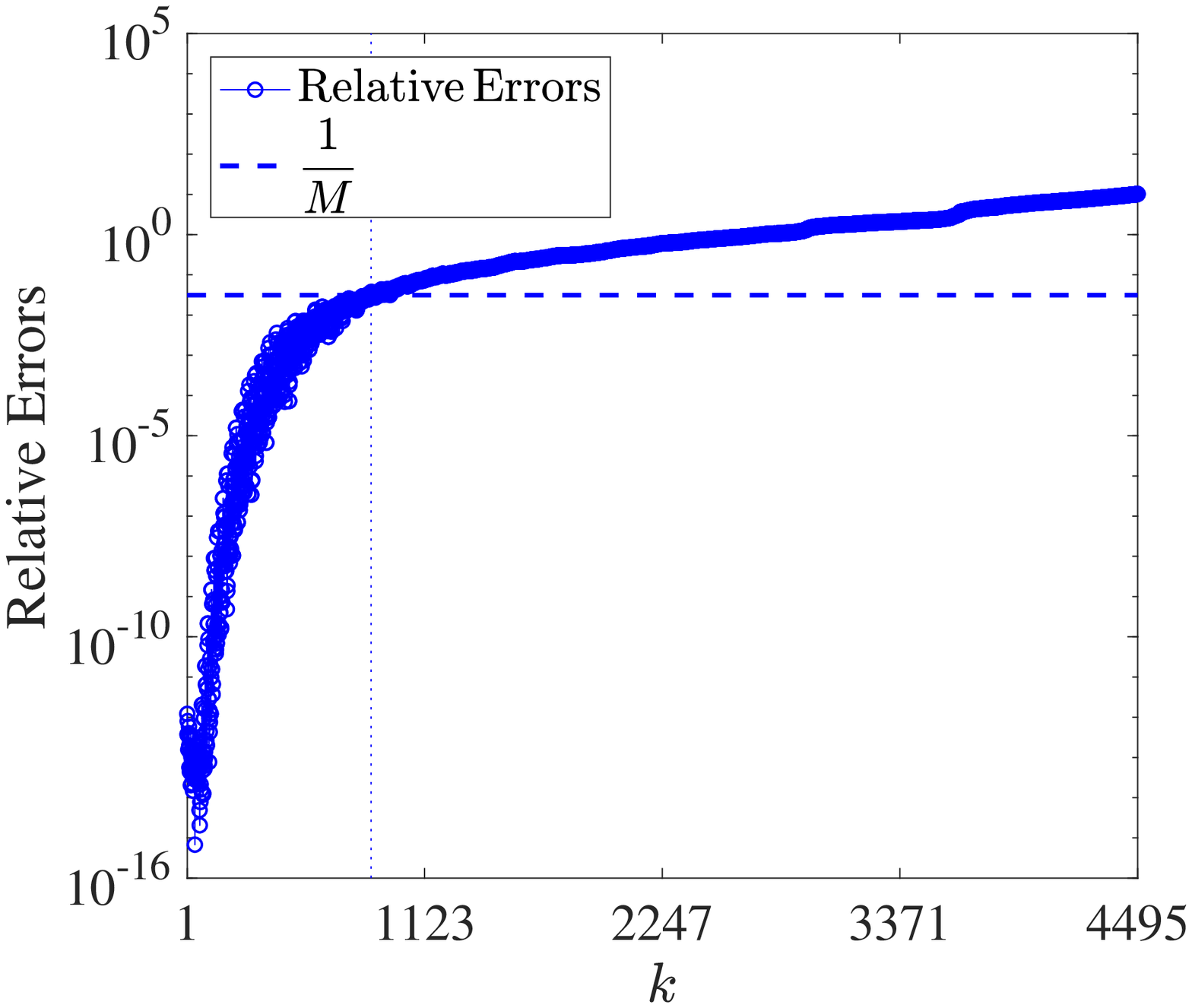}
\end{minipage}%
}%
\subfigure{
\begin{minipage}[t]{0.3\linewidth}
\centering
\includegraphics[width=1.8in,height=1.6in]{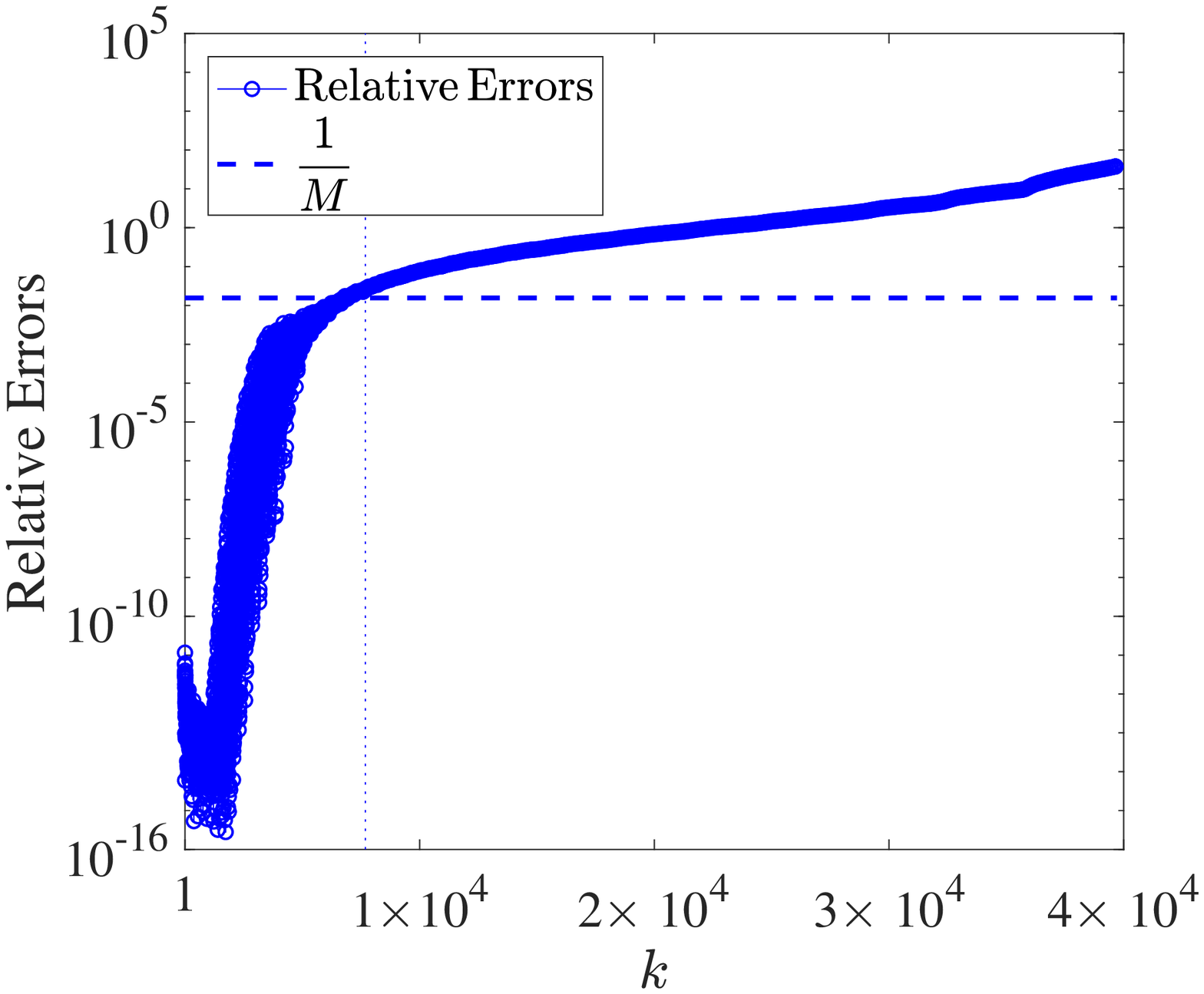}
\end{minipage}%
}%
\subfigure{
\begin{minipage}[t]{0.3\linewidth}
\centering
\includegraphics[width=1.8in,height=1.6in]{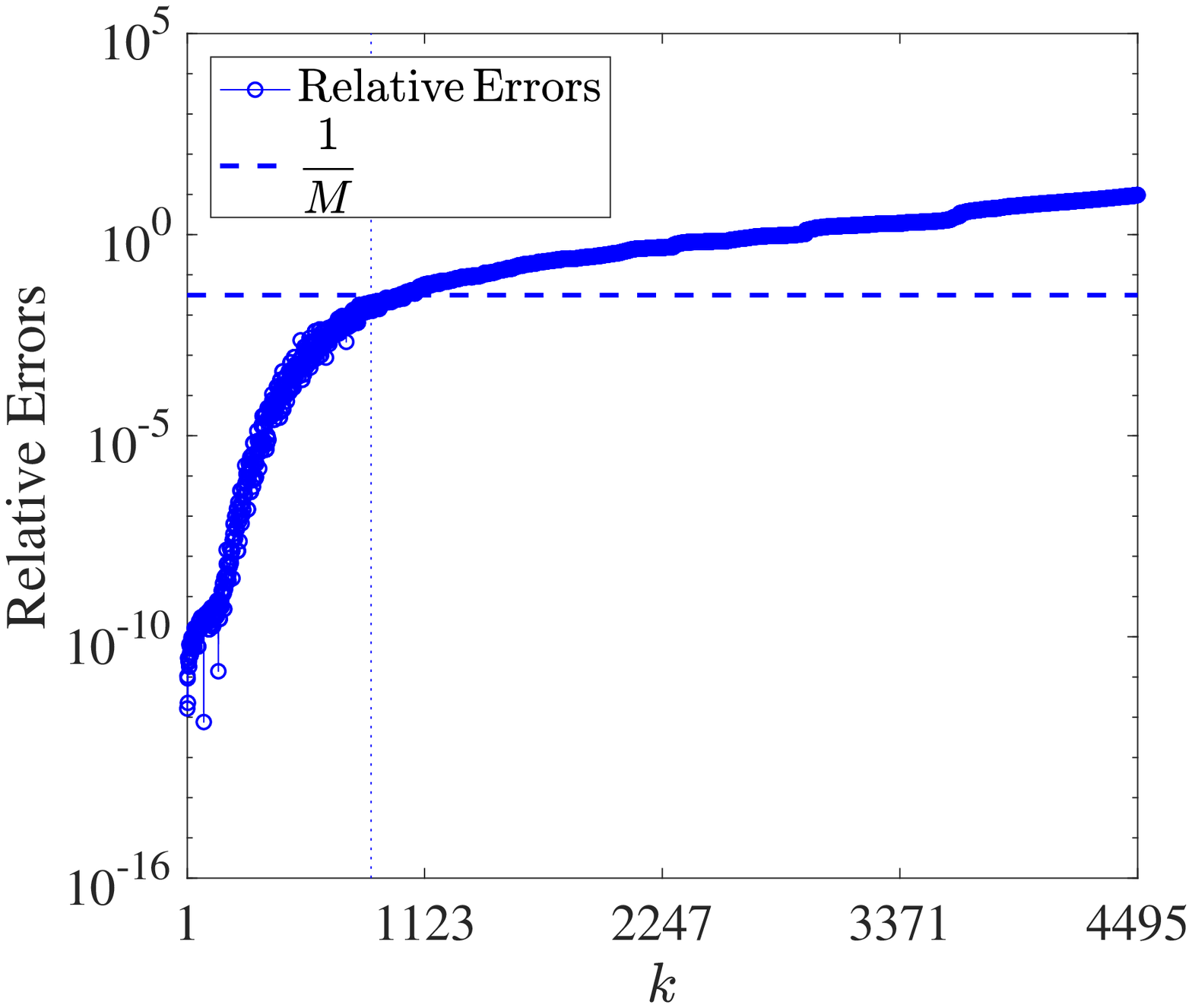}
\end{minipage}%
}%
\centering\vspace{-0.5em}
\caption{Relatives errors of all numerical  eigenvalues on $\TT_F$ when $M = 32$ (left) and $M=64$ (middle); on $\TT_R$ when $M = 32$ (right). The vertical dashed line denotes the portion of $\frac{3}{4}(\frac{2}{\pi})^3.$}\label{relative_errors} 
\end{figure}

 It follows  that there are
about $\frac{3}{4}(\frac{2}{\pi})^3 \approx 19.35\%$ numerical eigenvalues for which relative errors converge at rate $\mathcal{O}(M^{-1})$ for
our spectral-Galerkin method.
Referred by eigenvalues derived with relatively large $M$, we also draw convergence behaviors  of numerical eigenvalues on $\TT_R$ when $M=32$ in the right of Figure \ref{relative_errors}.
Although none of our numerical eigenvalues  can  reach the machine precision in this case,
almost same portion of reliable eigenvalues are observed.

Now, let us  demonstrate  in Figure \ref{example5_weyl} the asymptotic behaviors of the first reliable  3000 numerical eigenvalues of  \eqref{scheme_eigen} 
computed by our spectral method with  $M=64$.
We observe that numerical eigenvalues suit well with the Weyl's conjecture \eqref{weyl_law}.
It, in return, confirms once again the accuracy of these numerical eigenvalues.

\begin{figure}[htb]
\centering
\subfigure{
\begin{minipage}[t]{0.5\linewidth}
\centering
\includegraphics[width=2.2in,height=1.9in]{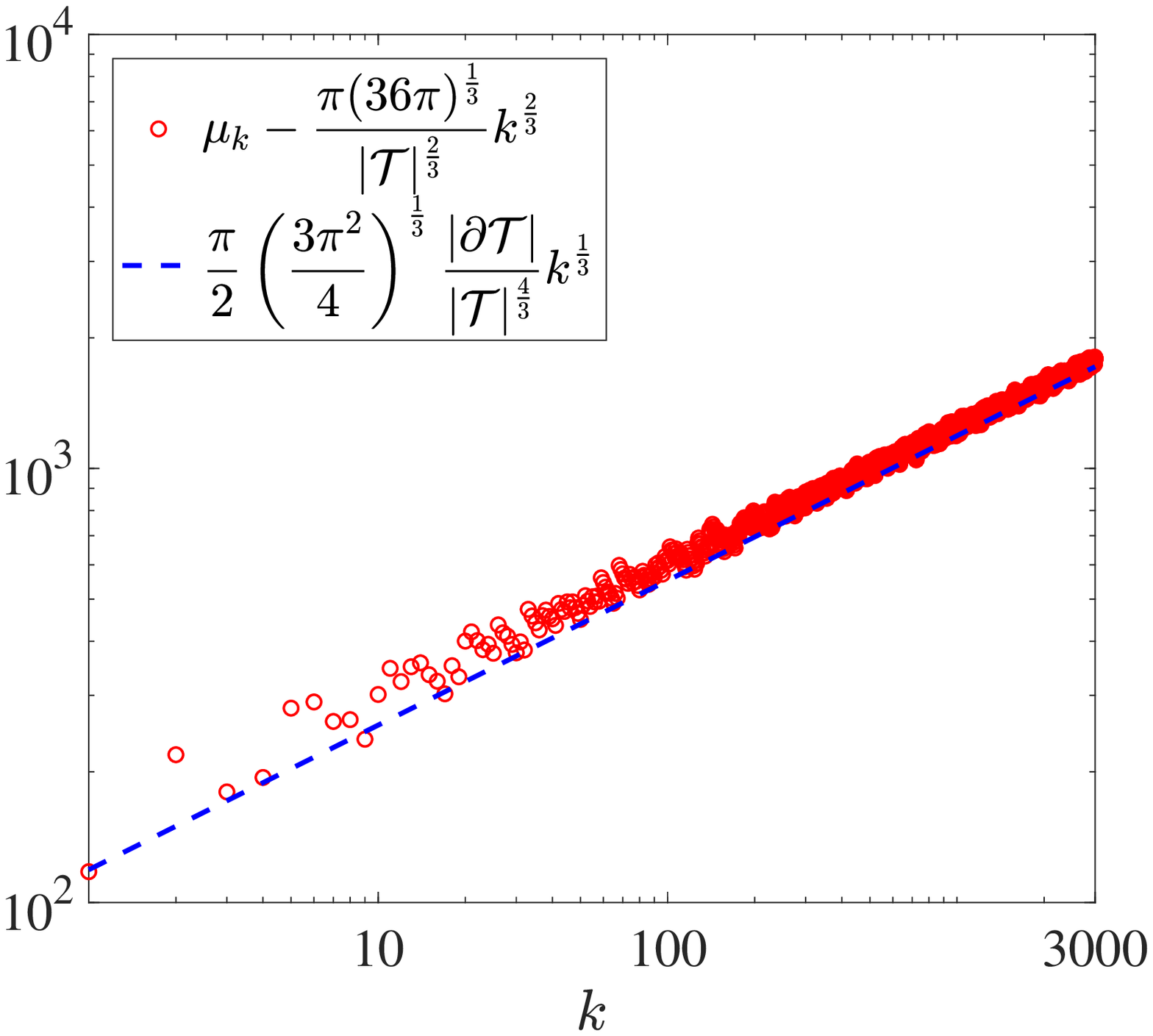}
\end{minipage}%
}%
\subfigure{
\begin{minipage}[t]{0.5\linewidth}
\centering
\includegraphics[width=2.2in,height=1.9in]{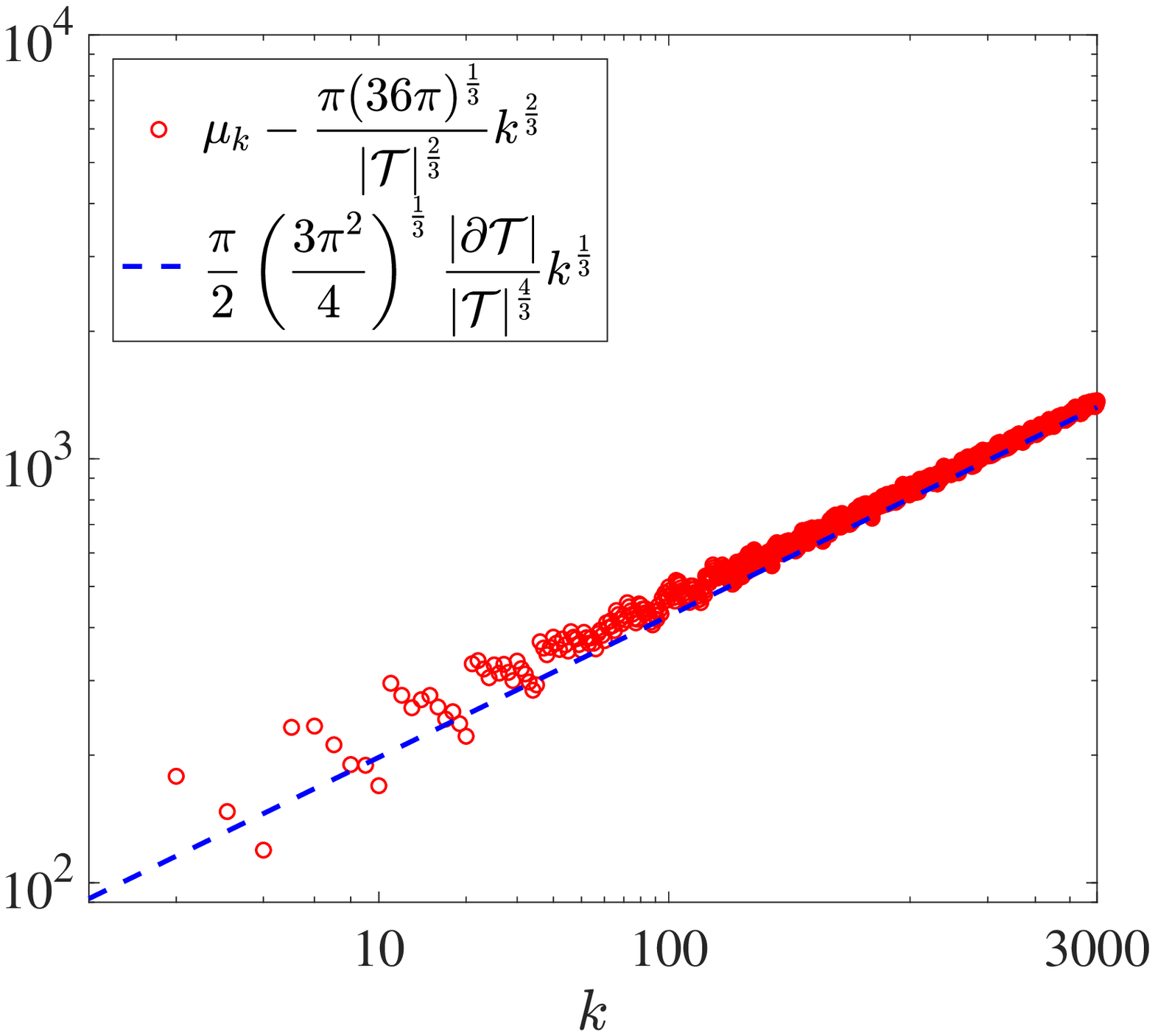}
\end{minipage}%
}%
\centering\vspace{-0.5em}
\caption{Asymptotic behaviors of eigenvalues $\mu_k$ against $k$ on $\TT_F$ (left) and $\TT_R$ (right).}\label{example5_weyl}
\end{figure}

Next,  we turn to explore different gaps of these reliable numerical eigenvalues.
We introduce the following definitions \cite{JMRR1999, Baogaps2020}:

\begin{itemize}


\item the average gaps: \, $\delta_{\rm ave}(k) := \dfrac{1}{k} \sum\limits_{j=1}^k \left( \mu_{j+1}-\mu_j\right) = \dfrac{\mu_{k+1}-\mu_1 }{k},\, k\in \NN;$
\vspace{-0.5em}
\item the normalized gaps: \, $\delta_{\rm norm}(k):= y_{k+1} -y_k,\,$
 $y_k=\left( \mu_k \cdot \frac{ |\TT|^{\frac{2}{3}}}{\pi (36\pi)^{\frac{1}{3}}}\right)^{3/2}$,\, $k\in\NN$.
\end{itemize}
Another interesting term is the level spacing distribution $P(s)$ representing the limiting distribution of the normalized gaps, which is defined by \cite{JMRR1999, Baogaps2020}
\begin{equation*}
\dfrac{ \sharp \{ j\,|1\le j \le k\, |\, \delta_{\rm norm}(j)<x\}}{k} \stackrel{k\rightarrow +\infty}{\longrightarrow} \int_0^x P(s)\, ds,\quad 0\le x <+\infty,
\end{equation*}
where $\sharp S$ denotes the cardinality of the set $S$.

For both $\TT_F$ and $\TT_R$, similar observations are derived from  Figure \ref{3000_gaps} and Figure \ref{regular_gaps}: 
 $\delta_{\rm ave}(k) \sim k^{-\frac{1}{3}}$, which is also a direct consequence of \eqref{weyl_law} and the definition of $\delta_{\rm ave}(k)$;
 statistically, the gaps distribution satisfies  $P(s) = \delta(s)$, where $\delta(s)$ is the Dirac delta function.
 
\begin{figure}[H]
\centering
\subfigure{
\begin{minipage}[t]{0.5\linewidth}
\centering
\includegraphics[width=2.2in,height=1.9in]{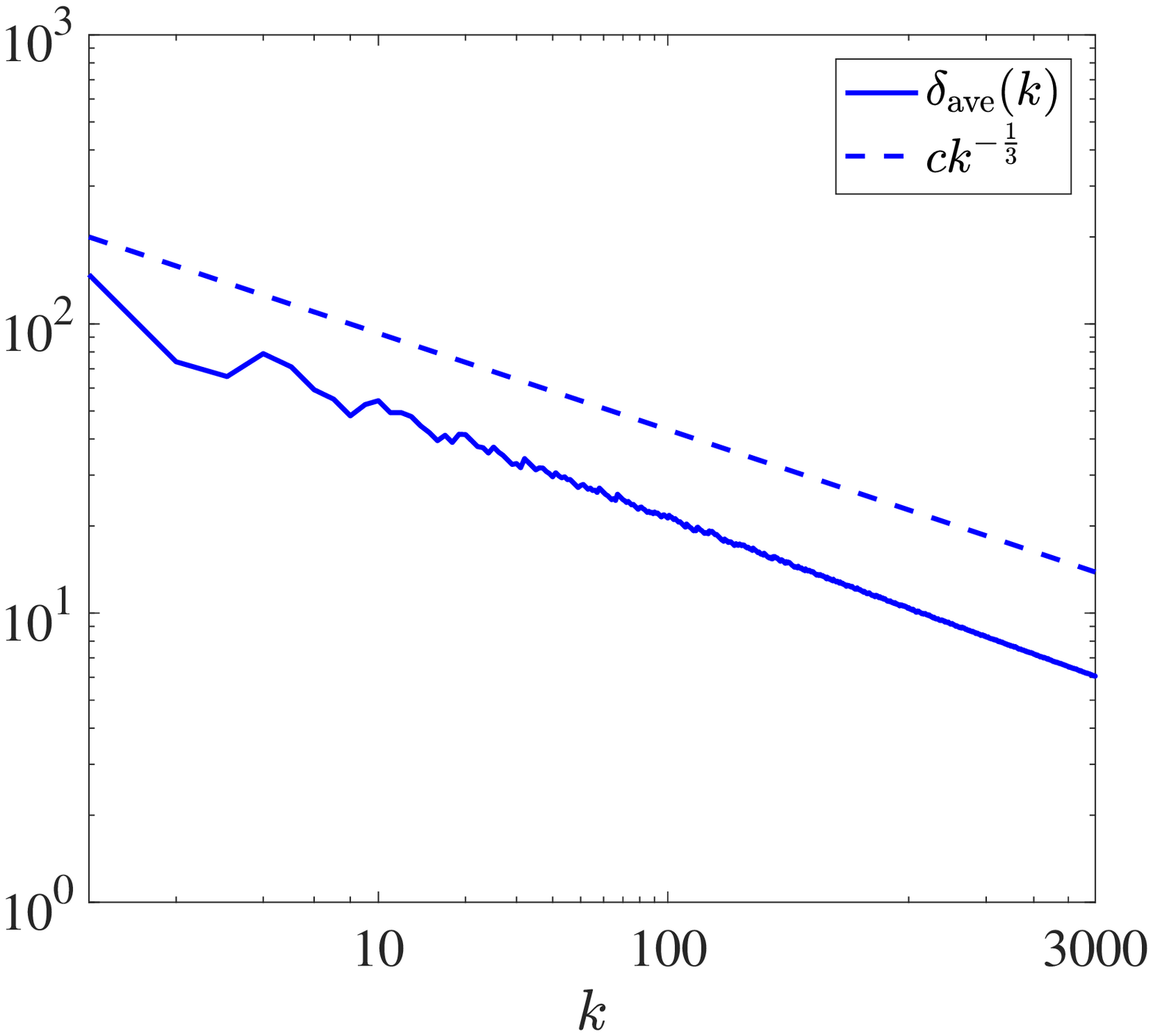}
\end{minipage}%
}%
\subfigure{
\begin{minipage}[t]{0.5\linewidth}
\centering
\includegraphics[width=2.2in,height=1.9in]{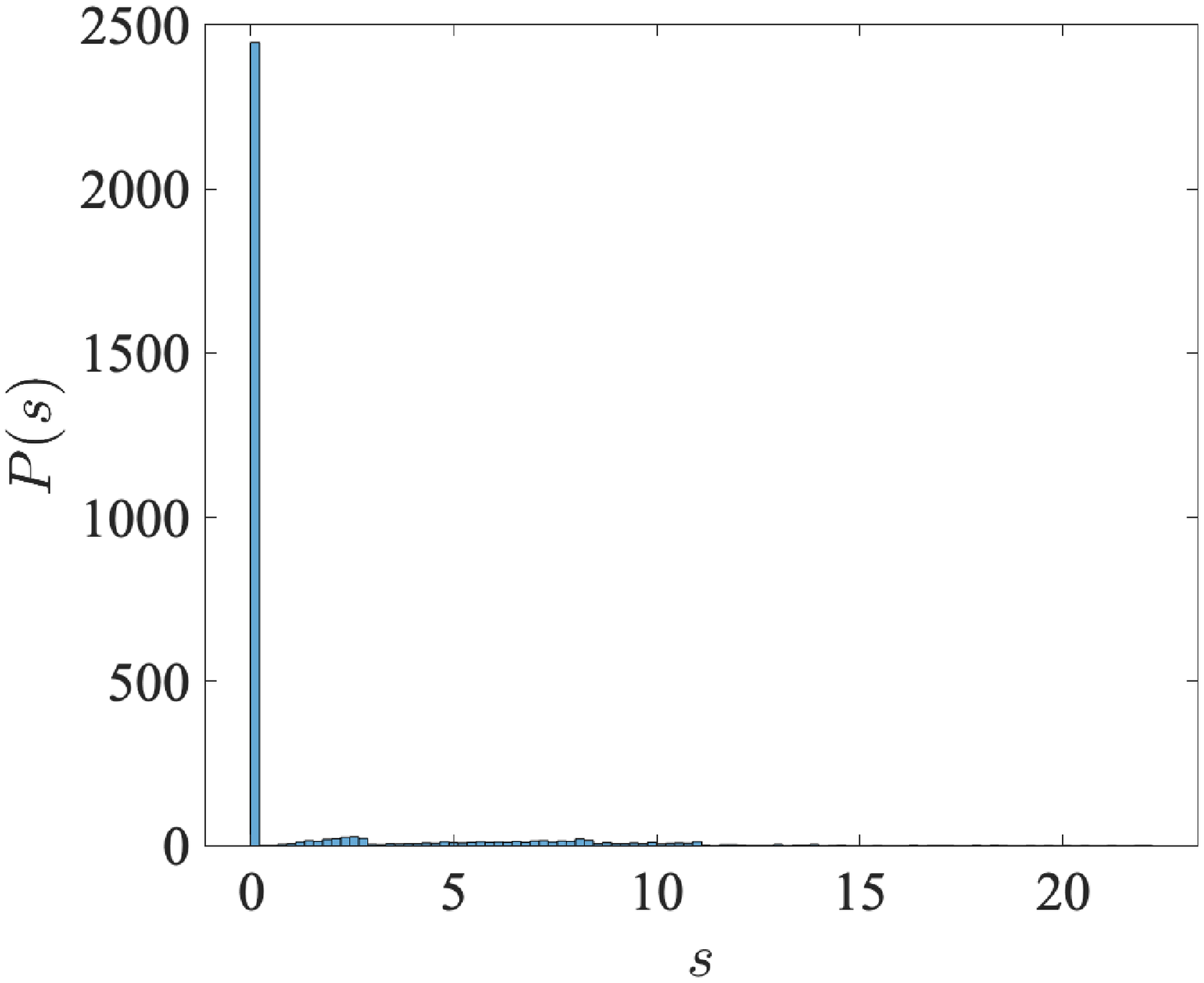}
\end{minipage}%
}%
\centering\vspace{-0.5em}
\caption{The average gaps (left) and the level spacing distribution (right) on $\TT_F$.}\label{3000_gaps}
\end{figure}

\begin{figure}[H]
\centering
\subfigure{
\begin{minipage}[t]{0.5\linewidth}
\centering
\includegraphics[width=2.2in,height=1.9in]{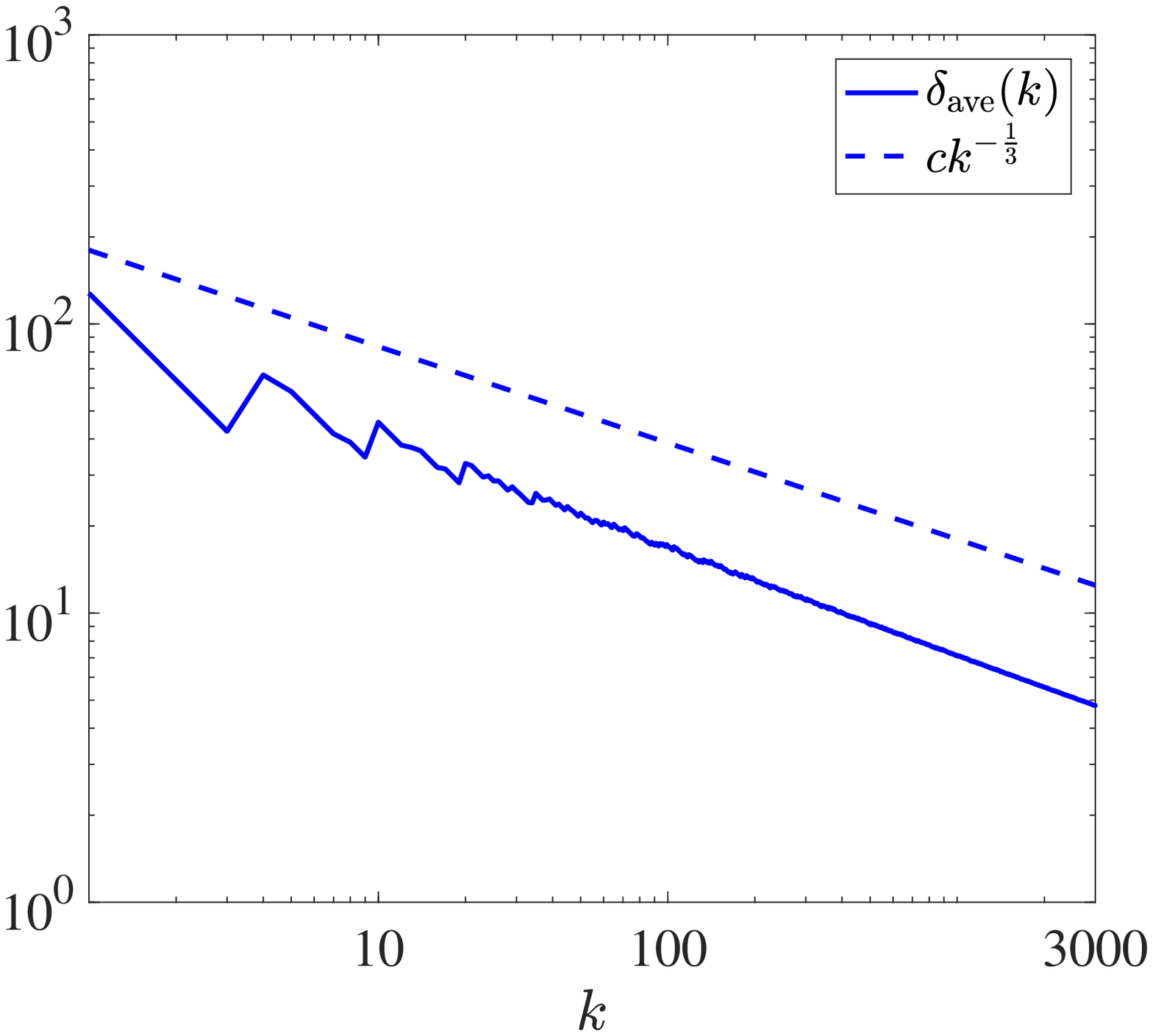}
\end{minipage}%
}%
\subfigure{
\begin{minipage}[t]{0.5\linewidth}
\centering
\includegraphics[width=2.2in,height=1.9in]{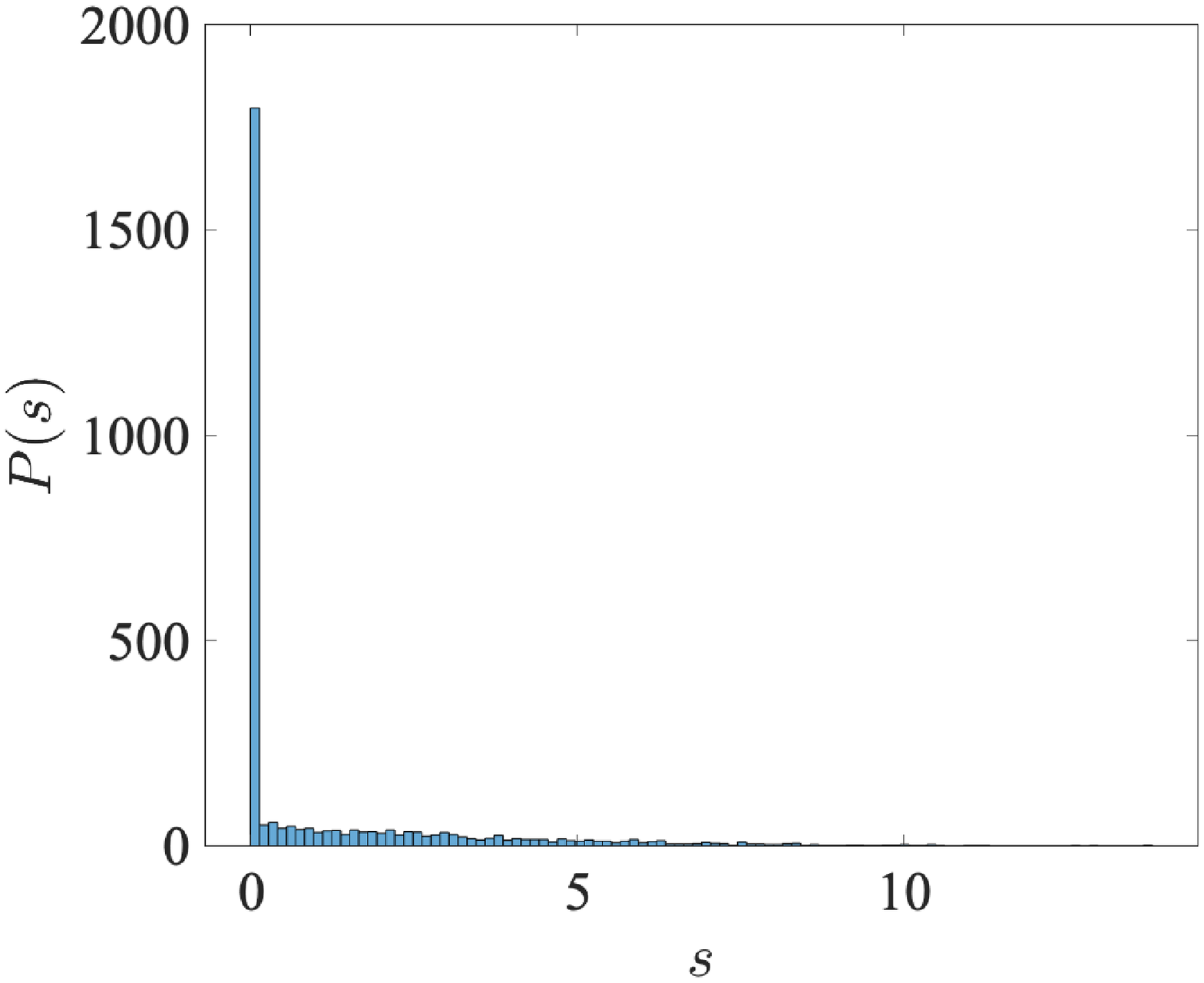}
\end{minipage}%
}%
\centering\vspace{-0.5em}
\caption{The average gaps (left) and the level spacing distribution (right) on $\TT_R$.}\label{regular_gaps}
\end{figure}

\section{Conclusion}\label{conclusion}

 We introduced in this paper a sparse spectral-Galerkin method for second-order partial differential equations on an arbitrary tetrahedron using generalized Koornwinder polynomials.
By exploring various recurrence relations of generalized Koornwinder polynomials, we derive well-conditioned and sparse linear systems which can be efficiently solved.
Numerical results for different kinds of source problems and the Laplacian eigenvalue problem confirm the sparsity, effectiveness 
and spectral accuracy of our method.

With the modal basis functions defined in this paper being applied directly for $C^0$-conforming elements, this work can be instantly  extended to spectral-element methods on tetrahedral meshes for complex geometries.
Theoretical approximation results will also be investigated in a future work.

\begin{appendix}

\section{Recurrence relations for increasing parameters}\label{section_increase}
We derive some useful recurrence relations for generalized Koornwinder polynomials in Appendix \ref{section_increase}-\ref{section_diff}. 
Firstly, we rewrite the Koornwinder polynomials in the collapsed  coordinate  to simplify the incoming  proofs, 
\begin{equation}\label{koornwinder_xi}
\JJ_{\bell}^{\balpha}(\bm{\hx}) = J_{\ell_1}^{\alpha_0,\alpha_1}(\xi)
\left( \frac{1-\eta}{2}\right)^{\ell_1} J_{\ell_2}^{2\ell_1+\alpha_0+\alpha_1+1,\alpha_2}(\eta) 
\left( \frac{1-\zeta}{2}\right)^{\ell_1+\ell_2}
J_{\ell_3}^{2\ell_1+2\ell_2+\alpha_0+\alpha_1+\alpha_2+2,\alpha_3}(\zeta),
\end{equation}
where 
\begin{equation}\label{duffy_3d}
\xi=\dfrac{2\hx_1}{1-\hx_2-\hx_3}-1,\quad
\eta=\dfrac{2\hx_2}{1-\hx_3}-1,\quad
\zeta=2\hx_3-1.
\end{equation}
We also let
\begin{equation*}
\begin{aligned}
&\dot{\bm{e}}_0=(1,0,0,0),\quad \dot{\bm{e}}_1=(0,1,0,0),\quad \dot{\bm{e}}_2=(0,0,1,0),\quad \dot{\bm{e}}_3=(0,0,0,1).
\end{aligned}
\end{equation*}
All coefficient functions in appendixes are defined as in Lemma \ref{lemma_three}-\ref{lemma_differential}.

\begin{lemma}\label{lemma_increase}
For any $\balpha\in[-1,+\infty)^4$ and $\bell\in \NN_0^3$, the following recurrence relations hold:

\begin{align}
&{\JJ}_{\bell}^{\balpha}(\bm{\hx}) = \sum\limits_{p=0}^1  \sum\limits_{q=0}^1 \sum\limits_{r=0}^1 \mathcal{A}_{p,q,r}^1 (\bell,\balpha) {\JJ}_{\bell-\left( p,\, q-p,\, r-q\right)}^{{\balpha}+\dot{\bm{e}}_0} (\bm{\hx}),\label{increase1}\\[-0.3em]
& {\JJ}_{\bell}^{\balpha} (\bm{\hx})= \sum\limits_{p=0}^1  \sum\limits_{q=0}^1 \sum\limits_{r=0}^1 \mathcal{A}_{p,q,r}^2 (\bell,\balpha) {\JJ}_{\bell-\left( p,\, q-p,\, r-q\right)}^{{\balpha}+\dot{\bm{e}}_1} (\bm{\hx}),\label{increase2}\\[-0.3em]
& {\JJ}_{\bell}^{\balpha} (\bm{\hx})= \sum\limits_{q=0}^1 \sum\limits_{r=0}^1 \mathcal{A}_{q,r}^3 (\bell,\balpha) 
{\JJ}_{\bell-\left( 0,\,q,\,r-q\right)}^{{\balpha}+\dot{\bm{e}}_2}(\bm{\hx}),\label{increase3}\\[-0.3em]
& {\JJ}_{\bell}^{\balpha}(\bm{\hx}) = \sum\limits_{r=0}^1 \mathcal{A}_r^4(\bell,\balpha) {\JJ}_{\bell-\left( 0,0,r\right)}^{{\balpha}+\dot{\bm{e}}_3}(\bm{\hx}),\label{increase4}
\end{align}
where the corresponding  coefficients are presented in  Table \ref{increase_table1}.
\begin{table}[H]
\caption{The values of $\mathcal{A}^1_{p,q,r},$ $\mathcal{A}^2_{p,q,r}$, $\mathcal{A}^3_{q,r}$ and $\mathcal{A}^4_{r}$.}
\label{increase_table1}
\centerline{
\begin{tabular}{|c|c|c|c|}
\hline
$(p,q,r)$ & $\mathcal{A}^1_{p,q,r}(\bell,\balpha)$ & $(p,q,r)$& $\mathcal{A}^2_{p,q,r}(\bell,\balpha)$\\ 
\hline
(0,0,0) & $b_{1,\ell_1}^{\alpha_0,\alpha_1}b_{1,\ell_2}^{2\ell_1+|\balpha^1|+1,\alpha_2} b_{1,\ell_3}^{2|\bell^2| +|\balpha^2|+2,\alpha_3} $
&(0,0,0)
&  $b_{1,\ell_1}^{\alpha_0,\alpha_1}b_{1,\ell_2}^{2\ell_1+|\balpha^1|+1,\alpha_2} b_{1,\ell_3}^{2|\bell^2| +|\balpha^2|+2,\alpha_3} $\\
\hline
(0,0,1) & $ b_{1,\ell_1}^{\alpha_0,\alpha_1}b_{1,\ell_2}^{2\ell_1+|\balpha^1|+1,\alpha_2} b_{2,\ell_3}^{2|\bell^2| +|\balpha^2|+2,\alpha_3} $   
&(0,0,1) 
& $ b_{1,\ell_1}^{\alpha_0,\alpha_1}b_{1,\ell_2}^{2\ell_1+|\balpha^1|+1,\alpha_2} b_{2,\ell_3}^{2|\bell^2| +|\balpha^2|+2,\alpha_3} $ \\
\hline
(0,1,0) & $b_{1,\ell_1}^{\alpha_0,\alpha_1}b_{2,\ell_2}^{2\ell_1+|\balpha^1|+1,\alpha_2} e_{2,\ell_3}^{2|\bell^2| +|\balpha^2|+1,\alpha_3}  $ 
& (0,1,0) & $b_{1,\ell_1}^{\alpha_0,\alpha_1}b_{2,\ell_2}^{2\ell_1+|\balpha^1|+1,\alpha_2} e_{2,\ell_3}^{2|\bell^2| +|\balpha^2|+1,\alpha_3}  $\\
\hline
(0,1,1) & $b_{1,\ell_1}^{\alpha_0,\alpha_1}b_{2,\ell_2}^{2\ell_1+|\balpha^1|+1,\alpha_2} e_{1,\ell_3}^{2|\bell^2| +|\balpha^2|+1,\alpha_3}  $ 
&(0,1,1)
& $b_{1,\ell_1}^{\alpha_0,\alpha_1}b_{2,\ell_2}^{2\ell_1+|\balpha^1|+1,\alpha_2} e_{1,\ell_3}^{2|\bell^2| +|\balpha^2|+1,\alpha_3}  $\\
\hline
(1,0,0) & $b_{2,\ell_1}^{\alpha_0,\alpha_1} e_{2,\ell_2}^{2\ell_1+|\balpha^1| ,\alpha_2} b_{1,\ell_3}^{2|\bell^2| +|\balpha^2|+2,\alpha_3} $  &(1,0,0) & $-b_{2,\ell_1}^{\alpha_1,\alpha_0} e_{2,\ell_2}^{2\ell_1+|\balpha^1| ,\alpha_2} b_{1,\ell_3}^{2|\bell^2| +|\balpha^2|+2,\alpha_3} $\\
\hline
(1,0,1) & $b_{2,\ell_1}^{\alpha_0,\alpha_1} e_{2,\ell_2}^{2\ell_1+|\balpha^1| ,\alpha_2} b_{2,\ell_3}^{2|\bell^2| +|\balpha^2|+2,\alpha_3} $  &(1,0,1) & $-b_{2,\ell_1}^{\alpha_1,\alpha_0} e_{2,\ell_2}^{2\ell_1+|\balpha^1| ,\alpha_2} b_{2,\ell_3}^{2|\bell^2| +|\balpha^2|+2,\alpha_3} $\\
\hline
(1,1,0) & $ b_{2,\ell_1}^{\alpha_0,\alpha_1} e_{1,\ell_2}^{2\ell_1+|\balpha^1| ,\alpha_2} e_{2,\ell_3}^{2|\bell^2| +|\balpha^2|+1,\alpha_3} $ &(1,1,0) & $ -b_{2,\ell_1}^{\alpha_1,\alpha_0} e_{1,\ell_2}^{2\ell_1+|\balpha^1| ,\alpha_2} e_{2,\ell_3}^{2|\bell^2| +|\balpha^2|+1,\alpha_3} $\\
\hline
(1,1,1)& $   b_{2,\ell_1}^{\alpha_0,\alpha_1} e_{1,\ell_2}^{2\ell_1+|\balpha^1| ,\alpha_2} e_{1,\ell_3}^{2|\bell^2| +|\balpha^2|+1,\alpha_3}  $ &(1,1,1) & $   -b_{2,\ell_1}^{\alpha_1,\alpha_0} e_{1,\ell_2}^{2\ell_1+|\balpha^1| ,\alpha_2} e_{1,\ell_3}^{2|\bell^2| +|\balpha^2|+1,\alpha_3}  $ \\
\hline\hline
$(q,r)$ & $\mathcal{A}^3_{q,r}(\bell,\balpha)$ & $(q,r)$& $\mathcal{A}^3_{q,r}(\bell,\balpha)$\\ 
\hline
(0,0)&$ b_{1,\ell_2}^{2\ell_1+|\balpha^1|+1,\alpha_2} b_{1,\ell_3}^{2|\bell^2|+|\balpha^2|+2,\alpha_3} $
& (1,0)
& $ -b_{2,\ell_2}^{\alpha_2,2\ell_1+|\balpha^1|+1} e_{2,\ell_3}^{2|\bell^2|+|\balpha^2|+1,\alpha_3} $\\
\hline
(0,1) & $ b_{1,\ell_2}^{2\ell_1+|\balpha^1|+1,\alpha_2} b_{2,\ell_3}^{2|\bell^2|+|\balpha^2|+2,\alpha_3} $
 & (1,1)
 & $ -b_{2,\ell_2}^{\alpha_2,2\ell_1+|\balpha^1|+1} e_{1,\ell_3}^{2|\bell^2|+|\balpha^2|+1,\alpha_3} $\\
\hline\hline
$r$ & $\mathcal{A}^4_{r}(\bell,\balpha)$ & $r$& $\mathcal{A}^4_{r}(\bell,\balpha)$\\ 
\hline
0 &
 $b_{1,\ell_3}^{2|\bell^2|+|\balpha^2|+2,\alpha_3}$ &1 & $-b_{2,\ell_3}^{\alpha_3, 2|\bell^2|+|\balpha^2|+2}$\\
\hline
\end{tabular}}
\end{table}
\end{lemma}

\begin{proof}
We take the proof of \eqref{increase3} as an example.
Other identities shall be proved in a similar way.
According to \eqref{LemEQ2(2)}, \eqref{LemEQ2(1)} and \eqref{LemEQ3(1)}, one has
{\small
\begin{equation*}
\begin{aligned}
\JJ_{\bell}^{\balpha}(\bm{\hx})&=J_{\ell_1}^{\alpha_0,\alpha_1}(\xi)
\left( \frac{1-\eta}{2}\right)^{\ell_1}  \bigg[ b_{1,\ell_2}^{2\ell_1+|\balpha^1|+1,\alpha_2} J_{\ell_2}^{2\ell_1+|\balpha^1|+1,\alpha_2+1}(\eta) \left( \frac{1-\zeta}{2}\right)^{\ell_1+\ell_2} \\
&\qquad\times
\left(
 b_{1,\ell_3}^{2|\bell^2|+|\balpha^2|+2,\alpha_3} J_{\ell_3}^{2|\bell^2| +|\balpha^2|+3,\alpha_3}(\zeta)+
b_{2,\ell_3}^{2|\bell^2|+|\balpha^2|+2,\alpha_3} J_{\ell_3-1}^{2|\bell^2| +|\balpha^2|+3,\alpha_3}(\zeta)\right)\\
&\qquad\qquad 
-b_{2,\ell_2}^{\alpha_2,2\ell_1+|\balpha^1|+1} J_{\ell_2-1}^{2\ell_1+|\balpha^1|+1,\alpha_2+1}(\eta) \left( \frac{1-\zeta}{2}\right)^{\ell_1+\ell_2-1} \\
&\qquad\times\left(
e_{1,\ell_3}^{2|\bell^2|+|\balpha^2|+1,\alpha_3} J_{\ell_3}^{2|\bell^2| +|\balpha^2|+1,\alpha_3}(\zeta)+
e_{2,\ell_3}^{2|\bell^2|+|\balpha^2|+1,\alpha_3} J_{\ell_3+1}^{2|\bell^2| +|\balpha^2|+1,\alpha_3}(\zeta)
\right) \bigg].
\end{aligned}
\end{equation*}
}
This completes the proof.
\end{proof}

\section{Recurrence relations for derivatives}\label{section_diff}

\begin{lemma}\label{lemma_diff}
For any $\balpha\in[-1,+\infty)^4$ and $\bell\in\NN_0^3$, the following recurrence relations hold:

\begin{align}
&\partial_{\hx_1}{\JJ}_{\bell}^{\balpha} (\bm{\hx})=2d_{\ell_1}^{\alpha_0,\alpha_1} {\JJ}_{\bell-(1,0,0)}^{\balpha+\dot{\bm{e}}_{0}+\dot{\bm{e}}_1} (\bm{\hx}),\label{diff1}\\[-0.3em]
&\partial_{\hx_2} {\JJ}_{\bell}^{\balpha} (\bm{\hx})= \sum\limits_{p=0}^1 \mathcal{D}^2_p(\bell,\balpha) {\JJ}_{\bell-\left(p,1-p,0\right)}^{{\balpha}+\dot{\bm{e}}_{0}+\dot{\bm{e}}_2} (\bm{\hx}),\label{diff2}\\[-0.3em]
&\left(\partial_{\hx_2}- \partial_{\hx_1}\right) {\JJ}_{\bell}^{\balpha} (\bm{\hx})= \sum\limits_{p=0}^1 \mathcal{D}^{21}_p(\bell,\balpha) {\JJ}_{\bell-\left(p,1-p,0\right)}^{{\balpha}+\dot{\bm{e}}_{1}+\dot{\bm{e}}_2} (\bm{\hx}),\label{diff12}\\[-0.3em]
&\partial_{\hx_3} {\JJ}_{\bell}^{\balpha} (\bm{\hx})=\sum\limits_{p=0}^1 \sum\limits_{q=0}^1\mathcal{D}^3_{p,q}(\bell,\balpha)
{\JJ}_{\bell-\left(p,\,q-p,\,1-q\right)}^{{\balpha}+\dot{\bm{e}}_{0}+\dot{\bm{e}}_3} (\bm{\hx}),\label{diff3}\\[-0.3em]
&\left(\partial_{\hx_1}- \partial_{\hx_3}\right){\JJ}_{\bell}^{\balpha} (\bm{\hx}) =\sum\limits_{p=0}^1 \sum\limits_{q=0}^1\mathcal{D}^{13}_{p,q}(\bell,\balpha)
{\JJ}_{\bell-\left(p,\,q-p,\,1-q\right)}^{{\balpha}+\dot{\bm{e}}_{1}+\dot{\bm{e}}_3} (\bm{\hx}),\label{diff13}\\[-0.3em]
&\left(\partial_{\hx_3}- \partial_{\hx_2}\right) {\JJ}_{\bell}^{\balpha} (\bm{\hx})=\sum\limits_{q=0}^1  \mathcal{D}_q^{32}(\bell,\balpha){\JJ}_{\bell-\left(0,q,1-q\right)}^{{\balpha}+\dot{\bm{e}}_{2}+\dot{\bm{e}}_3}(\bm{\hx}).\label{diff23}
\end{align}
With the notations
\begin{equation*}
\begin{aligned}
&\rho_{\bell}^{\balpha} := 2d_{\ell_1}^{\alpha_0,\alpha_1} e_{1,\ell_1-1}^{\alpha_1,\alpha_0+1}-\ell_1 b_{2,\ell_1}^{\alpha_0,\alpha_1},\quad
\kappa_{\bell}^{\balpha} :=  \ell_1 b_{2,\ell_1}^{\alpha_1,\alpha_0}-2d_{\ell_1}^{\alpha_0,\alpha_1} e_{1,\ell_1-1}^{\alpha_0,\alpha_1+1},\\
&\theta_{\bell}^{\balpha} := 
2d_{\ell_2}^{2\ell_1+|\balpha^1|+1,\alpha_2} e_{1,\ell_2-1}^{\alpha_2, 2\ell_1+|\balpha^1|+2} - \ell_2 b_{2,\ell_2}^{2\ell_1+|\balpha^1|+1,\alpha_2},
\end{aligned}
\end{equation*}
the corresponding  coefficients are presented as follows.
\begin{align*}
 &\mathcal{D}_0^2(\bell,\balpha)=2d_{\ell_2} ^{2\ell_1+ |\balpha^1| +1,\alpha_2} b_{1,\ell_1}^{\alpha_0,\alpha_1},
 \\[0.2em]
& \mathcal{D}_1^2(\bell,\balpha)
=\tfrac{ \left( 2d_{\ell_1}^{\alpha_0,\alpha_1} e_{1,\ell_1-1}^{\alpha_1,\alpha_0+1} -\ell_1 b_{2,\ell_1}^{\alpha_0,\alpha_1} \right) b_{1,\ell_2}^{2\ell_1+ |\balpha^1| +1,\alpha_2} + 
2d_{\ell_2}^{2\ell_1+ |\balpha^1| +1,\alpha_2}  b_{2,\ell_1}^{\alpha_0,\alpha_1} e_{2,\ell_2-1}^{2\ell_1+ |\balpha^1| +1,\alpha_2+1}}{ b_{1,\ell_2}^{2\ell_1+ |\balpha^1| ,\alpha_2+1}}.
\end{align*}

\vspace{-1em}

\begin{align*}
& \mathcal{D}_0^{21}(\bell,\balpha) = \mathcal{D}_0^{2}(\bell,\balpha),\\
& \mathcal{D}_1^{21}(\bell,\balpha) = 
\tfrac{\left(\ell_1 b_{2,\ell_1}^{\alpha_1,\alpha_0} -2d_{\ell_1}^{\alpha_0,\alpha_1} e_{1,\ell_1-1}^{\alpha_0,\alpha_1+1} \right) b_{1,\ell_2}^{2\ell_1+ |\balpha^1| +1,\alpha_2} - 
2d_{\ell_2}^{2\ell_1+ |\balpha^1| +1,\alpha_2}  b_{2,\ell_1}^{\alpha_1,\alpha_0} e_{2,\ell_2-1}^{2\ell_1+ |\balpha^1| +1,\alpha_2+1}}{ b_{1,\ell_2}^{2\ell_1+ |\balpha^1| ,\alpha_2+1}}.
\end{align*}

\vspace{-1em}

\begin{align*}
&\mathcal{D}^3_{0,0}(\bell,\balpha) = 2d_{\ell_3}^{2|\bell^2|+|\balpha^2|+2,\alpha_3} b_{1,\ell_1}^{\alpha_0,\alpha_1} b_{1,\ell_2}^{2\ell_1+|\balpha^1|+1,\alpha_2},\\
&\mathcal{D}^3_{0,1}(\bell,\balpha) = \tfrac{ b_{1,\ell_1}^{\alpha_0,\alpha_1} \theta_{\bell}^{\balpha}
b_{1,\ell_3}^{2|\bell^2| +|\balpha^2| +2,\alpha_3}
+2 d_{\ell_3}^{2|\bell^2| + |\balpha^2| +2,\alpha_3} b_{1,\ell_1}^{\alpha_0,\alpha_1} b_{2,\ell_2}^{2\ell_1+|\balpha^1| +1,\alpha_2} e_{2,\ell_3-1}^{2|\bell^2| +|\balpha^2| +2,\alpha_3+1} } {b_{1,\ell_3}^{2|\bell^2| +|\balpha^2|+2,\alpha_3+1}},\\
&\mathcal{D}^3_{1,0}(\bell,\balpha) = 2d_{\ell_3}^{2|\bell^2|+|\balpha^2|+2,\alpha_3} b_{2,\ell_1}^{\alpha_0,\alpha_1} e_{2,\ell_2}^{2\ell_1+|\balpha^1|,\alpha_2},\\
&\mathcal{D}^3_{1,1}(\bell,\balpha) = 
\tfrac{
\left(
\rho_{\bell}^{\balpha} + b_{2,\ell_1}^{\alpha_0,\alpha_1} e_{2,\ell_2-1}^{2\ell_1+|\balpha^1|+1,\alpha_2} \theta_{\bell}^{\balpha}
\right) b_{1,\ell_3}^{2|\bell^2|+|\balpha^2|+2,\alpha_3} +2b_{1,\ell_2}^{2\ell_1+|\balpha^1|,\alpha_2} 
d_{\ell_3}^{2|\bell^2|+|\balpha^2| +2,\alpha_3} 
b_{2,\ell_1}^{\alpha_0,\alpha_1} 
e_{1,\ell_2}^{2\ell_1+|\balpha^2|,\alpha_2} 
e_{2,\ell_3-1}^{2|\bell^2| +|\balpha^2|+2,\alpha_3+1}}
{b_{1,\ell_2}^{2\ell_1+|\balpha^1|,\alpha_2} b_{1,\ell_3}^{2|\bell^2|+|\balpha^2|+1,\alpha_3+1}}.
\end{align*}

\vspace{-1em}

\begin{align*}
&\mathcal{D}^{13}_{0,0}(\bell,\balpha) = 
-\mathcal{D}^{3}_{0,0}(\bell,\balpha),\\
&\mathcal{D}^{13}_{0,1}(\bell,\balpha) = 
-\mathcal{D}^{3}_{0,1}(\bell,\balpha),\\
&\mathcal{D}^{13}_{1,0}(\bell,\balpha) = 
2d_{\ell_3}^{2|\bell^2|+|\balpha^2|+2,\alpha_3} 
b_{2,\ell_1}^{\alpha_1,\alpha_0} 
e_{2,\ell_2}^{2\ell_1+|\balpha^1|,\alpha_2},\\
&\mathcal{D}^{13}_{1,1}(\bell,\balpha) = 
\tfrac{
\left(
  b_{2,\ell_1}^{\alpha_1,\alpha_0} e_{2,\ell_2-1}^{2\ell_1+|\balpha^1|+1,\alpha_2} \theta_{\bell}^{\balpha}-\kappa_{\bell}^{\balpha}
\right) b_{1,\ell_3}^{2|\bell^2|+|\balpha^2|+2,\alpha_3} +2b_{1,\ell_2}^{2\ell_1+|\balpha^1|,\alpha_2} 
d_{\ell_3}^{2|\bell^2|+|\balpha^2| +2,\alpha_3} 
b_{2,\ell_1}^{\alpha_1,\alpha_0} 
e_{1,\ell_2}^{2\ell_1+|\balpha^2|,\alpha_2} 
e_{2,\ell_3-1}^{2|\bell^2| +|\balpha^2|+2,\alpha_3+1}}
{b_{1,\ell_2}^{2\ell_1+|\balpha^1|,\alpha_2} b_{1,\ell_3}^{2|\bell^2|+|\balpha^2|+1,\alpha_3+1}}.
\end{align*}

\vspace{-1em}

\begin{align*}
&\mathcal{D}^{32}_{0}(\bell,\balpha) = 
2d_{\ell_3}^{2|\bell^2| + |\balpha^2|+2,\alpha_3} b_{1,\ell_2}^{2\ell_1+|\balpha^1|+1,\alpha_2},\\
&\mathcal{D}^{32}_{1}(\bell,\balpha)  = 
\tfrac{\left( \ell_2 b_{2,\ell_2} ^{\alpha_2, 2\ell_1+|\balpha^1|+1}-2d_{\ell_2}^{2\ell_1+|\balpha^1|+1,\alpha_2} 
e_{1,\ell_2-1}^{2\ell_1+|\balpha^1|+1,\alpha_2+1}\right) b_{1,\ell_3}^{2|\bell^2|+|\balpha^2| +2,\alpha_3} -
2d_{\ell_3}^{2|\bell^2| +|\balpha^2|+2,\alpha_3} 
b_{2,\ell_2}^{\alpha_2, 2\ell_1+|\balpha^1|+1}
e_{2,\ell_3-1}^{2|\bell^2|+|\balpha^2|+2,\alpha_3+1}}
{b_{1,\ell_3}^{2|\bell^2|+|\balpha^2|+1,\alpha_3+1}}.
\end{align*}

\end{lemma}

\begin{proof}
It follows from \eqref{duffy_3d} that
{\small
\begin{equation}
\begin{cases}
\partial_{\hx_1}  = \dfrac{8}{(1-\eta)(1-\zeta)} \partial_{\xi},\\[1em]
\partial_{\hx_2} = \dfrac{4(1+\xi)}{(1-\eta)(1-\zeta)} \partial_{\xi} + \dfrac{4}{1-\zeta} \partial_{\eta},\\[1em]
\partial_{\hx_3} = \dfrac{4(1+\xi)}{(1-\eta)(1-\zeta)} \partial_{\xi} + \dfrac{2(1+\eta)}{1-\zeta} \partial_{\eta} +2\partial_{\zeta}.
\end{cases}
\end{equation}
}
We take the proof of \eqref{diff2} as an example.
Other identities shall be proved in a similar way.
To begin with, when $\ell_1 = 0,$ one has
{\small
\begin{equation*}
\begin{aligned}
&\partial_{\hx_2}\JJ^{\alpha_0,\alpha_1,\alpha_2,\alpha_3}_{0,\ell_2,\ell_3} = 2 \partial_\eta 
J_{\ell_2} ^{ |\balpha^1| +1,\alpha_2} (\eta)
\left( \frac{1-\zeta}{2}\right)^{\ell_2-1} J_{\ell_3}^{2\ell_2+ |\balpha^2| +2,\alpha_3}(\zeta)\\
&=2d_{\ell_2}^{ |\balpha^1| +1,\alpha_2} J_{\ell_2-1} ^{ |\balpha^1| +2,\alpha_2+1} (\eta)  \left( \frac{1-\zeta}{2}\right)^{\ell_2-1} J_{\ell_3}^{2\ell_2+ |\balpha^2| +2,\alpha_3}(\zeta)=2d_{\ell_2}^{ |\balpha^1| +1,\alpha_2} \JJ_{0,\ell_2-1,\ell_3}^{\alpha_0+1,\alpha_1,\alpha_2+1,\alpha_3}.
\end{aligned}
\end{equation*}
}
When $\ell_1>1,$ a direct computation yields
{\small
\begin{equation}\label{partialx2_1}
\begin{aligned}
&\partial_{\hx_2}\JJ^{\balpha}_{\bell} = \bigg[ (1+\xi) \partial_{\xi} J_{\ell_1} ^{\alpha_0,\alpha_1} (\xi) \left( \frac{1-\eta}{2}\right)^{\ell_1-1} J_{\ell_2} ^{2\ell_1+ |\balpha^1| +1,\alpha_2} (\eta)
\\
&\quad+ 2 J_{\ell_1} ^{\alpha_0,\alpha_1} (\xi) \partial_\eta 
\big[  \left( \frac{1-\eta}{2}\right)^{\ell_1} J_{\ell_2} ^{2\ell_1+ |\balpha^1| +1,\alpha_2} (\eta)
\big]\bigg]
\times
\left( \frac{1-\zeta}{2}\right)^{|\bell^2|-1} J_{\ell_3}^{2|\bell^2|+ |\balpha^2| +2,\alpha_3}(\zeta)\\
&= \bigg[ \left(  d_{\ell_1}^{\alpha_0,\alpha_1} (1+\xi) J_{\ell_1-1}^{\alpha_0+1,\alpha_1+1}(\xi)-\ell_1 J_{\ell_1}^{\alpha_0,\alpha_1}(\xi)  \right)\left( \frac{1-\eta}{2}\right)^{\ell_1-1} J_{\ell_2} ^{2\ell_1+ |\balpha^1| +1,\alpha_2} (\eta) \\
&\quad +2d_{\ell_2}^{2\ell_1+ |\balpha^1| +1,\alpha_2} J_{\ell_1}^{\alpha_0,\alpha_1} (\xi) \left( \frac{1-\eta}{2}\right)^{\ell_1} J_{\ell_2-1} ^{2\ell_1+ |\balpha^1| +2,\alpha_2+1} (\eta)  \bigg] \left( \frac{1-\zeta}{2}\right)^{|\bell^2|-1} J_{\ell_3}^{2|\bell^2|+ |\balpha^2| +2,\alpha_3}(\zeta).
\end{aligned}
\end{equation}
}
Recalling \eqref{LemEQ2(1)} and \eqref{LemEQ3(2)}, we have
{\small
\begin{equation*}
\begin{aligned}
&\quad d_{\ell_1}^{\alpha_0,\alpha_1} (1+\xi) J_{\ell_1-1}^{\alpha_0+1,\alpha_1+1}(\xi)-\ell_1 J_{\ell_1}^{\alpha_0,\alpha_1}(\xi) \\
&\quad = 
2d_{\ell_1}^{\alpha_0,\alpha_1} 
\left( e_{1,\ell_1-1}^{\alpha_1,\alpha_0+1} J_{\ell_1-1}^{\alpha_0+1,\alpha_1}(\xi)  
-e_{2,\ell_1-1}^{\alpha_0+1,\alpha_1} J_{\ell_1}^{\alpha_0+1,\alpha_1}(\xi)\right) -\ell_1 \left( 
b_{1,\ell_1}^{\alpha_0,\alpha_1} J_{\ell_1}^{\alpha_0+1,\alpha_1}(\xi) +  b_{2,\ell_1}^{\alpha_0,\alpha_1} J_{\ell_1-1}^{\alpha_0+1,\alpha_1}(\xi)
\right)\\
&\quad =
\left( 2d_{\ell_1}^{\alpha_0,\alpha_1} e_{1,\ell_1-1}^{\alpha_1,\alpha_0+1} -\ell_1 b_{2,\ell_1}^{\alpha_0,\alpha_1}\right) J_{\ell_1-1}^{\alpha_0+1,\alpha_1}(\xi).\quad  
\left(\because -2d_{\ell_1}^{\alpha_0,\alpha_1} e_{2,\ell_1-1}^{\alpha_0+1,\alpha_1} -\ell_1 b_{1,\ell_1}^{\alpha_0,\alpha_1} = 0\right)
\end{aligned}
\end{equation*}
}
Substituting the above formula into \eqref{partialx2_1}  and using \eqref{LemEQ2(1)}, \eqref{LemEQ2(2)} and \eqref{LemEQ3(1)}, one has
{\small
\begin{equation*}
\begin{aligned}
&\partial_{\hx_2} \JJ_{\bell}^{\balpha} = 
\bigg[ \left( 2d_{\ell_1}^{\alpha_0,\alpha_1} e_{1,\ell_1-1}^{\alpha_1,\alpha_0+1} -\ell_1 b_{2,\ell_1}^{\alpha_0,\alpha_1}\right) J_{\ell_1-1}^{\alpha_0+1,\alpha_1}(\xi)
\left( \frac{1-\eta}{2}\right)^{\ell_1-1} J_{\ell_2} ^{2\ell_1+ |\balpha^1| +1,\alpha_2} (\eta) \\
&\quad +2d_{\ell_2}^{2\ell_1+ |\balpha^1| +1,\alpha_2}  
\left( b_{1,\ell_1}^{\alpha_0,\alpha_1} J_{\ell_1}^{\alpha_0+1,\alpha_1}(\xi) 
+b_{2,\ell_1}^{\alpha_0,\alpha_1} J_{\ell_1-1}^{\alpha_0+1,\alpha_1}(\xi)\right]
 \left( \frac{1-\eta}{2}\right)^{\ell_1} J_{\ell_2-1} ^{2\ell_1+ |\balpha^1| +2,\alpha_2+1} (\eta)  \bigg] \\
 &\quad\times \left( \frac{1-\zeta}{2}\right)^{|\bell^2|-1} J_{\ell_3}^{2|\bell^2|+ |\balpha^2| +2,\alpha_3}(\zeta)\\
&=2d_{\ell_2} ^{2\ell_1+ |\balpha^1| +1,\alpha_2} b_{1,\ell_1}^{\alpha_0,\alpha_1} \JJ_{\ell_1,\ell_2-1,\ell_3}^{\alpha_0+1,\alpha_1,\alpha_2+1,\alpha_3} +  J_{\ell_1-1}^{\alpha_0+1,\alpha_1}(\xi)
\left( \frac{1-\eta}{2}\right)^{\ell_1-1}
\left( \frac{1-\zeta}{2}\right)^{|\bell^2|-1} J_{\ell_3}^{2|\bell^2|+ |\balpha^2| +2,\alpha_3}(\zeta)\\
 &\quad \times \bigg[
 \left( 2d_{\ell_1}^{\alpha_0,\alpha_1} e_{1,\ell_1-1}^{\alpha_1,\alpha_0+1} -\ell_1 b_{2,\ell_1}^{\alpha_0,\alpha_1}\right)
 J_{\ell_2} ^{2\ell_1+ |\balpha^1| +1,\alpha_2} (\eta)
 +2d_{\ell_2}^{2\ell_1+ |\balpha^1| +1,\alpha_2}  b_{2,\ell_1}^{\alpha_0,\alpha_1}  \frac{1-\eta}{2} J_{\ell_2-1} ^{2\ell_1+  |\balpha^1| +2,\alpha_2+1} (\eta)
  \bigg]
\\
&=2d_{\ell_2} ^{2\ell_1+ |\balpha^1| +1,\alpha_2} b_{1,\ell_1}^{\alpha_0,\alpha_1} \JJ_{\ell_1,\ell_2-1,\ell_3}^{\alpha_0+1,\alpha_1,\alpha_2+1,\alpha_3}
+ J_{\ell_1-1}^{\alpha_0+1,\alpha_1}(\xi)
\left( \frac{1-\eta}{2}\right)^{\ell_1-1}
\left( \frac{1-\zeta}{2}\right)^{|\bell^2|-1} J_{\ell_3}^{2|\bell^2|+ |\balpha^2| +2,\alpha_3}(\zeta)\\
&\quad \times \bigg[ \left( 2d_{\ell_1}^{\alpha_0,\alpha_1} e_{1,\ell_1-1}^{\alpha_1,\alpha_0+1} -\ell_1 b_{2,\ell_1}^{\alpha_0,\alpha_1}\right) 
\left( 
b_{1,\ell_2}^{2\ell_1+ |\balpha^1| +1,\alpha_2} J_{\ell_2}^{2\ell_1+ |\balpha^1| +1,\alpha_2+1} (\eta) 
-b_{2,\ell_2}^{\alpha_2, 2\ell_1+ |\balpha^1| +1} J_{\ell_2-1}^{2\ell_1+ |\balpha^1| +1,\alpha_2+1} (\eta) 
\right)\\
&\quad+ 2d_{\ell_2}^{2\ell_1+ |\balpha^1| +1,\alpha_2}  b_{2,\ell_1}^{\alpha_0,\alpha_1} 
\left( 
e_{1,\ell_2-1}^{2\ell_1+ |\balpha^1| +1,\alpha_2+1} J_{\ell_2-1}^{2\ell_1+ |\balpha^1| +1,\alpha_2+1}(\eta)
+e_{2,\ell_2-1}^{2\ell_1+ |\balpha^1| +1,\alpha_2+1} J_{\ell_2}^{2\ell_1+ |\balpha^1| +1,\alpha_2+1}(\eta)
\right)\bigg]\\
&=2d_{\ell_2} ^{2\ell_1+ |\balpha^1| +1,\alpha_2} b_{1,\ell_1}^{\alpha_0,\alpha_1} \JJ_{\ell_1,\ell_2-1,\ell_3}^{\alpha_0+1,\alpha_1,\alpha_2+1,\alpha_3} 
+ J_{\ell_1-1}^{\alpha_0+1,\alpha_1}(\xi)
\left( \frac{1-\eta}{2}\right)^{\ell_1-1}
\left( \frac{1-\zeta}{2}\right)^{|\bell^2|-1} J_{\ell_3}^{2|\bell^2|+ |\balpha^2| +2,\alpha_3}(\zeta)\\
&\quad \times \bigg[ \left( ( 2d_{\ell_1}^{\alpha_0,\alpha_1} e_{1,\ell_1-1}^{\alpha_1,\alpha_0+1} -\ell_1 b_{2,\ell_1}^{\alpha_0,\alpha_1}) b_{1,\ell_2}^{2\ell_1+ |\balpha^1| +1,\alpha_2} + 
2d_{\ell_2}^{2\ell_1+ |\balpha^1| +1,\alpha_2}  b_{2,\ell_1}^{\alpha_0,\alpha_1} e_{2,\ell_2-1}^{2\ell_1+ |\balpha^1| +1,\alpha_2+1}  \right)
J_{\ell_2}^{2\ell_1+ |\balpha^1| +1,\alpha_2+1}(\eta)\\
&\quad+
\left( (  \ell_1 b_{2,\ell_1}^{\alpha_0,\alpha_1} - 2d_{\ell_1}^{\alpha_0,\alpha_1} e_{1,\ell_1-1}^{\alpha_1,\alpha_0+1}) 
b_{2,\ell_2}^{\alpha_2,2\ell_1+ |\balpha^1| +1}  +
2d_{\ell_2}^{2\ell_1+ |\balpha^1| +1,\alpha_2}  b_{2,\ell_1}^{\alpha_0,\alpha_1} 
e_{1,\ell_2-1}^{2\ell_1+ |\balpha^1| +1,\alpha_2+1} \right)
J_{\ell_2-1}^{2\ell_1+ |\balpha^1| +1,\alpha_2+1}(\eta)
\bigg].
\end{aligned}
\end{equation*}
Note that $b_{1,\ell_2}^{2\ell_1+ |\balpha^1| ,\alpha_2+1}\neq 0$ when $\ell_1>0.$ It is readily  checked that
\begin{equation*}
\begin{aligned}
& \quad\left( ( 2d_{\ell_1}^{\alpha_0,\alpha_1} e_{1,\ell_1-1}^{\alpha_1,\alpha_0+1} -\ell_1 b_{2,\ell_1}^{\alpha_0,\alpha_1}) b_{1,\ell_2}^{2\ell_1+ |\balpha^1| +1,\alpha_2} + 
2d_{\ell_2}^{2\ell_1+ |\balpha^1| +1,\alpha_2}  b_{2,\ell_1}^{\alpha_0,\alpha_1} e_{2,\ell_2-1}^{2\ell_1+ |\balpha^1| +1,\alpha_2+1}  \right)
J_{\ell_2}^{2\ell_1+ |\balpha^1| +1,\alpha_2+1}(\eta)\\
&\quad\,+
\left(  (  \ell_1 b_{2,\ell_1}^{\alpha_0,\alpha_1} - 2d_{\ell_1}^{\alpha_0,\alpha_1} e_{1,\ell_1-1}^{\alpha_1,\alpha_0+1}) 
b_{2,\ell_2}^{\alpha_2,2\ell_1+ |\balpha^1| +1} +
2d_{\ell_2}^{2\ell_1+ |\balpha^1| +1,\alpha_2}  b_{2,\ell_1}^{\alpha_0,\alpha_1} 
e_{1,\ell_2-1}^{2\ell_1+ |\balpha^1| +1,\alpha_2+1} \right)
J_{\ell_2-1}^{2\ell_1+ |\balpha^1| +1,\alpha_2+1}(\eta)
\\
&=\frac{( 2d_{\ell_1}^{\alpha_0,\alpha_1} e_{1,\ell_1-1}^{\alpha_1,\alpha_0+1} -\ell_1 b_{2,\ell_1}^{\alpha_0,\alpha_1}) b_{1,\ell_2}^{2\ell_1+ |\balpha^1| +1,\alpha_2} + 
2d_{\ell_2}^{2\ell_1+ |\balpha^1| +1,\alpha_2}  b_{2,\ell_1}^{\alpha_0,\alpha_1} e_{2,\ell_2-1}^{2\ell_1+ |\balpha^1| +1,\alpha_2+1}}{ b_{1,\ell_2}^{2\ell_1+ |\balpha^1| ,\alpha_2+1}}\\
&\quad\times \left(  b_{1,\ell_2}^{2\ell_1+ |\balpha^1| ,\alpha_2+1}  J_{\ell_2}^{2\ell_1+ |\balpha^1| +1,\alpha_2+1} (\eta)
+ b_{2,\ell_2}^{2\ell_1+ |\balpha^1| ,\alpha_2+1}  J_{\ell_2-1}^{2\ell_1+ |\balpha^1| +1,\alpha_2+1} (\eta)
\right)\\
&=\frac{( 2d_{\ell_1}^{\alpha_0,\alpha_1} e_{1,\ell_1-1}^{\alpha_1,\alpha_0+1} -\ell_1 b_{2,\ell_1}^{\alpha_0,\alpha_1}) b_{1,\ell_2}^{2\ell_1+ |\balpha^1| +1,\alpha_2} + 
2d_{\ell_2}^{2\ell_1+ |\balpha^1| +1,\alpha_2}  b_{2,\ell_1}^{\alpha_0,\alpha_1} e_{2,\ell_2-1}^{2\ell_1+ |\balpha^1| +1,\alpha_2+1}}{ b_{1,\ell_2}^{2\ell_1+ |\balpha^1| ,\alpha_2+1}} J_{\ell_2}^{2\ell_1+ |\balpha^1| ,\alpha_2+1} (\eta).
\end{aligned}
\end{equation*}
}
Thus, it concludes that
{\small
\begin{equation*}
\begin{aligned}
&\partial_{\hx_2} \JJ_{\bell}^{\balpha} = 2d_{\ell_2} ^{2\ell_1+ |\balpha^1| +1,\alpha_2} b_{1,\ell_1}^{\alpha_0,\alpha_1} \JJ_{\ell_1,\ell_2-1,\ell_3}^{\alpha_0+1,\alpha_1,\alpha_2+1,\alpha_3}\\
&\quad +\frac{( 2d_{\ell_1}^{\alpha_0,\alpha_1} e_{1,\ell_1-1}^{\alpha_1,\alpha_0+1} -\ell_1 b_{2,\ell_1}^{\alpha_0,\alpha_1}) b_{1,\ell_2}^{2\ell_1+ |\balpha^1| +1,\alpha_2} + 
2d_{\ell_2}^{2\ell_1+ |\balpha^1| +1,\alpha_2}  b_{2,\ell_1}^{\alpha_0,\alpha_1} e_{2,\ell_2-1}^{2\ell_1+ |\balpha^1| +1,\alpha_2+1}}{ b_{1,\ell_2}^{2\ell_1+ |\balpha^1| ,\alpha_2+1}}
\JJ_{\ell_1-1,\ell_2,\ell_3}^{\alpha_0+1,\alpha_1,\alpha_2+1,\alpha_3}.
\end{aligned}
\end{equation*}
}
This ends the proof.
\end{proof}

\section{Coefficients in the three-term recurrence relations}\label{threecoeff}
By introducing the notations,
\begin{align*}
&\tau_{1,\bell}^{\balpha} := \tfrac{c_{1,\ell_3}^{2|\bell^2|+|\balpha^2|+2,\alpha_3}}{2},\quad 
&&\tau_{2,\bell}^{\balpha} := \tfrac{c_{2,\ell_3}^{2|\bell^2|+|\balpha^2|+2,\alpha_3}}{2},\quad
&&\tau_{3,\bell}^{\balpha} := \tfrac{c_{3,\ell_3}^{2|\bell^2|+|\balpha^2|+2,\alpha_3}}{2},\\[-0.1em]
&\tau_{4,\bell}^{\balpha} := -\tfrac{a_{1,\ell_3}^{2|\bell^2|+|\balpha^2|+2,\alpha_3}}{2},\quad
&&\tau_{5,\bell}^{\balpha} := \tfrac{(1-a_{2,\ell_3}^{2|\bell^2|+|\balpha^2|+2,\alpha_3})}{2},\quad
&&\tau_{6,\bell}^{\balpha} := -\tfrac{a_{3,\ell_3}^{2|\bell^2|+|\balpha^2|+2,\alpha_3}}{2},\\[-0.1em]
&\tau_{7,\bell}^{\balpha} := \tfrac{g_{1,\ell_3}^{2|\bell^2|+|\balpha^2|,\alpha_3}}{2},\quad
&& \tau_{8,\bell}^{\balpha} := \tfrac{g_{2,\ell_3}^{2|\bell^2|+|\balpha^2|,\alpha_3}}{2},\quad
&& \tau_{9,\bell}^{\balpha} := \tfrac{g_{3,\ell_3}^{2|\bell^2|+|\balpha^2|,\alpha_3}}{2},
\end{align*}
we list the coefficient $\mathscr{C}_{p,q,r}(\bell,\balpha)$, $\mathscr{C}_{q,r}(\bell,\balpha)$ and $\mathscr{C}_{r}(\bell,\balpha)$ in Theorem \ref{3Recurrence}  as follows:

\begin{table}[H]
\caption{{The values of $\mathscr{C}_{p,q,r}(\bell,\balpha)$, $\mathscr{C}_{q,r}(\bell,\balpha)$ and $\mathscr{C}_{r}(\bell,\balpha)$.}}
\label{threeterm_table}
\centerline{
\begin{tabular}{|c|c|c|c|}
\hline 
{{$(p,q,r)$}} & { $\mathscr{C}_{p,q,r}(\bell,\balpha)$}& {{$(p,q,r)$}}&   {$\mathscr{C}_{p,q,r} (\bell,\balpha)$}\\ 
\hline
(-1,-1,-1)
&$a_{1,\ell_1}^{\alpha_0,\alpha_1} c_{1,\ell_2}^{2\ell_1+  |\balpha^1|+1,\alpha_2} \tau_{1,\bell}^{\balpha}$ 
&(1,-1,-1)
& $a_{3,\ell_1}^{\alpha_0,\alpha_1} g_{1,\ell_2}^{2\ell_1+  |\balpha^1|-1,\alpha_2} \tau_{1,\bell}^{\balpha}$\\
 \hline
(-1,-1,0)
& $a_{1,\ell_1}^{\alpha_0,\alpha_1} c_{1,\ell_2}^{2\ell_1+  |\balpha^1|+1,\alpha_2} \tau_{2,\bell}^{\balpha}$  
& (1,-1,0)
& $a_{3,\ell_1}^{\alpha_0,\alpha_1} g_{1,\ell_2}^{2\ell_1+  |\balpha^1|-1,\alpha_2} \tau_{2,\bell}^{\balpha}$\\
\hline
(-1,-1,1)
 & $a_{1,\ell_1}^{\alpha_0,\alpha_1} c_{1,\ell_2}^{2\ell_1+  |\balpha^1|+1,\alpha_2} \tau_{3,\bell}^{\balpha}$
 & (1,-1,1)
 & $a_{3,\ell_1}^{\alpha_0,\alpha_1} g_{1,\ell_2}^{2\ell_1+  |\balpha^1|-1,\alpha_2} \tau_{3,\bell}^{\balpha}$\\
\hline
(-1,0,-1)
&  $a_{1,\ell_1}^{\alpha_0,\alpha_1} c_{2,\ell_2}^{2\ell_1+  |\balpha^1|+1,\alpha_2}\tau_{4,\bell}^{\balpha} /2$
& (1,0,-1)
& $ a_{3,\ell_1}^{\alpha_0,\alpha_1} g_{2,\ell_2}^{2\ell_1+  |\balpha^1|-1,\alpha_2} \tau_{4,\bell}^{\balpha}/2$\\
  \hline
 (-1,0,0)
 &  $a_{1,\ell_1}^{\alpha_0,\alpha_1} c_{2,\ell_2}^{2\ell_1+  |\balpha^1|+1,\alpha_2}\tau_{5,\bell}^{\balpha} /2$
&(1,0,0)
& $ a_{3,\ell_1}^{\alpha_0,\alpha_1} g_{2,\ell_2}^{2\ell_1+  |\balpha^1|-1,\alpha_2} \tau_{5,\bell}^{\balpha}/2$\\
\hline
(-1,0,1)
&  $a_{1,\ell_1}^{\alpha_0,\alpha_1} c_{2,\ell_2}^{2\ell_1+  |\balpha^1|+1,\alpha_2} \tau_{6,\bell}^{\balpha}/2$
 & (1,0,1)
 & $ a_{3,\ell_1}^{\alpha_0,\alpha_1} g_{2,\ell_2}^{2\ell_1+  |\balpha^1|-1,\alpha_2} \tau_{6,\bell}^{\balpha}/2$\\
\hline  
(-1,1,-1)
& $a_{1,\ell_1}^{\alpha_0,\alpha_1} c_{3,\ell_2}^{2\ell_1+  |\balpha^1|+1,\alpha_2} \tau_{7,\bell}^{\balpha}$ 
 &(1,1,-1)
 & $a_{3,\ell_1}^{\alpha_0,\alpha_1} g_{3,\ell_2}^{2\ell_1+  |\balpha^1|-1,\alpha_2} \tau_{7,\bell}^{\balpha}$\\
\hline
(-1,1,0)
& $a_{1,\ell_1}^{\alpha_0,\alpha_1} c_{3,\ell_2}^{2\ell_1+  |\balpha^1|+1,\alpha_2} \tau_{8,\bell}^{\balpha}$ 
 & (1,1,0)
 & $a_{3,\ell_1}^{\alpha_0,\alpha_1} g_{3,\ell_2}^{2\ell_1+  |\balpha^1|-1,\alpha_2} \tau_{8,\bell}^{\balpha}$\\
\hline
(-1,1,1)
& $a_{1,\ell_1}^{\alpha_0,\alpha_1} c_{3,\ell_2}^{2\ell_1+  |\balpha^1|+1,\alpha_2} \tau_{9,\bell}^{\balpha}$  
&(1,1,1)
& $a_{3,\ell_1}^{\alpha_0,\alpha_1} g_{3,\ell_2}^{2\ell_1+  |\balpha^1|-1,\alpha_2} \tau_{9,\bell}^{\balpha}$\\
\hline
\hline 
{{$(p,q,r)$}} & $\mathscr{C}_{p,q,r}(\bell,\balpha)$ & {{$(q,r)$}}& $\mathscr{C}_{q,r} (\bell,\balpha)$\\ 
\hline
(0,-1,-1)
& $-(1+a_{2,\ell_1}^{\alpha_0,\alpha_1}) a_{1,\ell_2}^{2\ell_1+|\balpha^1|+1,\alpha_2} \tau_{1,\bell}^{\balpha} /2$
 &(-1,-1)
 & $a_{1,\ell_2}^{2\ell_1+|\balpha^1|+1,\alpha_2} \tau_{1,\bell}^{\balpha}$\\
\hline
(0,-1,0)
& $-(1+a_{2,\ell_1}^{\alpha_0,\alpha_1}) a_{1,\ell_2}^{2\ell_1+|\balpha^1|+1,\alpha_2} \tau_{2,\bell}^{\balpha}/2$
 &(-1,0)
 & $a_{1,\ell_2}^{2\ell_1+|\balpha^1|+1,\alpha_2} \tau_{2,\bell}^{\balpha}$\\
 \hline
 (0,-1,1)
 & $-(1+a_{2,\ell_1}^{\alpha_0,\alpha_1}) a_{1,\ell_2}^{2\ell_1+|\balpha^1|+1,\alpha_2} \tau_{3,\bell}^{\balpha}/2$
 &(-1,1)
 & $a_{1,\ell_2}^{2\ell_1+|\balpha^1|+1,\alpha_2} \tau_{3,\bell}^{\balpha}$ \\
 \hline
 (0,0,-1)
 &  $(1+a_{2,\ell_1}^{\alpha_0,\alpha_1}) (1-a_{2,\ell_2}^{2\ell_1+|\balpha^1|+1,\alpha_2}) \tau_{4,\bell}^{\balpha}/4$
  &(0,-1)
  & $(1+a_{2,\ell_2}^{2\ell_1+|\balpha^1|+1,\alpha_2}) \tau_{4,\bell}^{\balpha}/2$\\
 \hline
(0,0,0)
&$(1+a_{2,\ell_1}^{\alpha_0,\alpha_1}) (1-a_{2,\ell_2}^{2\ell_1+|\balpha^1|+1,\alpha_2}) \tau_{5,\bell}^{\balpha}/4$
 &(0,0)
 & $(1+a_{2,\ell_2}^{2\ell_1+|\balpha^1|+1,\alpha_2}) \tau_{5,\bell}^{\balpha}/2$\\
 \hline
 (0,0,1)
 & $(1+a_{2,\ell_1}^{\alpha_0,\alpha_1}) (1-a_{2,\ell_2}^{2\ell_1+|\balpha^1|+1,\alpha_2}) \tau_{6,\bell}^{\balpha}/4$ 
  &(0,1)
  & $-(1+a_{2,\ell_2}^{2\ell_1+|\balpha^1|+1,\alpha_2})\tau_{6,\bell}^{\balpha}/2$\\
 \hline
 (0,1,-1)
 & $-(1+a_{2,\ell_1}^{\alpha_0,\alpha_1}) a_{3,\ell_2}^{2\ell_1+|\balpha^1|+1,\alpha_2} \tau_{7,\bell}^{\balpha}/2$
  &(1,-1)
  & $a_{3,\ell_2}^{2\ell_1+|\balpha^1|+1,\alpha_2} \tau_{7,\bell}^{\balpha}$\\
 \hline
(0,1,0)
 & $-(1+a_{2,\ell_1}^{\alpha_0,\alpha_1}) a_{3,\ell_2}^{2\ell_1+|\balpha^1|+1,\alpha_2} \tau_{8,\bell}^{\balpha}/2$
  & (1,0)
  &$a_{3,\ell_2}^{2\ell_1+|\balpha^1|+1,\alpha_2} \tau_{8,\bell}^{\balpha}$ \\
 \hline
 (0,1,1)
 & $-(1+a_{2,\ell_1}^{\alpha_0,\alpha_1}) a_{3,\ell_2}^{2\ell_1+|\balpha^1|+1,\alpha_2} \tau_{9,\bell}^{\balpha}/2$
  &(1,1)
  & $a_{3,\ell_2}^{2\ell_1+|\balpha^1|+1,\alpha_2} \tau_{9,\bell}^{\balpha}$\\
 \hline 
\end{tabular}}
\end{table}
\vspace{-2em}
\begin{equation*}
\left[ \mathscr{C}_{-1}(\bell,\balpha), \mathscr{C}_{0}(\bell,\balpha), \mathscr{C}_{1}(\bell,\balpha) \right] = \bigg[
\frac{a_{1,\ell_3}^{2|\bell^2|+|\balpha^2|+2,\alpha_3}}{2}, 
\frac{1+a_{2,\ell_3}^{2|\bell^2|+|\balpha^2|+2,\alpha_3}}{2},
\frac{a_{3,\ell_3}^{2|\bell^2|+|\balpha^2|+2,\alpha_3}}{2}
\bigg].
\end{equation*}

\begin{proof} We shall take the proof of \eqref{recurrence_x2} as an example to explain the derivations of these coefficients.
From \eqref{duffy_3d}, one obtains 
$$\hx_2 = \frac{1+\eta}{2} \frac{1-\zeta}{2}.$$
It then follows from \eqref{three_1d}, \eqref{LemEQ2(3)} and \eqref{LemEQ3(3)} that
{\small
\begin{equation*}
\begin{aligned}
&\hx_2 \JJ_{\bell}^{\balpha} = J_{\ell_1}^{\alpha_0,\alpha_1}(\xi) \left(\frac{1-\eta}{2} \right)^{\ell_1} \frac{1+\eta}{2} J_{\ell_2}^{2\ell_1+  |\balpha^1|+1,\alpha_2} (\eta)\left( \frac{1-\zeta}{2}\right)^{|\bell^2|+1}J_{\ell_3}^{2 |\bell^2| +|\balpha^2|+2,\alpha_3}(\zeta)\\
&\quad =J_{\ell_1}^{\alpha_0,\alpha_1}(\xi) \left(\frac{1-\eta}{2} \right)^{\ell_1}
\bigg[
\frac{a_{1,\ell_2}^{2\ell_1+|\balpha^1|+1,\alpha_2}}{2} J_{\ell_2+1}^{2\ell_1+|\balpha^1|+1,\alpha_2} (\eta) \left( \frac{1-\zeta}{2}\right)^{|\bell^2|+1}\\
&\qquad\quad\times \bigg(
c_{1,\ell_3}^{2|\bell^2|+|\balpha^2|+2,\alpha_3} J_{\ell_3}^{2|\bell^2|+|\balpha^2|+4,\alpha_3}(\zeta)+
c_{2,\ell_3}^{2|\bell^2|+|\balpha^2|+2,\alpha_3} J_{\ell_3-1}^{2|\bell^2|+|\balpha^2|+4,\alpha_3}(\zeta)\\
&\qquad\qquad 
+
c_{3,\ell_3}^{2|\bell^2|+|\balpha^2|+2,\alpha_3} J_{\ell_3-2}^{2|\bell^2|+|\balpha^2|+4,\alpha_3}(\zeta)
\bigg)\\
&\qquad\qquad\qquad+
\frac{1+a_{2,\ell_2}^{2\ell_1+|\balpha^1|+1,\alpha_2}}{2} J_{\ell_2}^{2\ell_1+|\balpha^1|+1,\alpha_2} (\eta)\left( \frac{1-\zeta}{2}\right)^{|\bell^2|} \\
&\qquad\quad\times \bigg(
-\frac{a_{1,\ell_3}^{2|\bell^2|+|\balpha^2| +2,\alpha_3}}{2} J_{\ell_3+1}^{2|\bell^2|+|\balpha^2|+2,\alpha_3}(\zeta)+
\frac{1-a_{2,\ell_3}^{2|\bell^2|+|\balpha^2| +2,\alpha_3}}{2} J_{\ell_3}^{2|\bell^2|+|\balpha^2|+2,\alpha_3}(\zeta)\\
&\qquad\qquad 
-
\frac{a_{3,\ell_3}^{2|\bell^2|+|\balpha^2| +2,\alpha_3}}{2} J_{\ell_3-1}^{2|\bell^2|+|\balpha^2|+2,\alpha_3}(\zeta)
\bigg)\\
&\qquad\qquad\qquad+
\frac{a_{3,\ell_2}^{2\ell_1+|\balpha^1|+1,\alpha_2}}{2} J_{\ell_2-1}^{2\ell_1+|\balpha^1|+1,\alpha_2} (\eta)\left( \frac{1-\zeta}{2}\right)^{|\bell^2|-1} \\
&\qquad\quad\times \left(
g_{1,\ell_3}^{2|\bell^2|+|\balpha^2|,\alpha_3} J_{\ell_3+2}^{2|\bell^2|+|\balpha^2|,\alpha_3}(\zeta)+
g_{2,\ell_3}^{2|\bell^2|+|\balpha^2|,\alpha_3} J_{\ell_3+1}^{2|\bell^2|+|\balpha^2|,\alpha_3}(\zeta)+
g_{3,\ell_3}^{2|\bell^2|+|\balpha^2|,\alpha_3} J_{\ell_3}^{2|\bell^2|+|\balpha^2|,\alpha_3}(\zeta) 
\right)
\bigg].
\end{aligned}
\end{equation*}
}
The proof is completed.
\end{proof}

\section{Exact eigenvalues of homogeneous Dirichlet Laplacian on $\TT_F$}\label{exact_TF}
We first claim that the generalized sine functions are eigenfunctions of the Dirichlet Laplacian on $\TT_F$. 
Actually, motivated by the study of \cite{LiXu2008},
we introduce homogeneous coordinates $\bss\in \RR_H^4$ with
\begin{equation}
 \RR_H^4:= \left\{\bss=(s_0,s_1,s_2,s_3)\in \RR^4: |\bss| = 0 \right\},\quad | \bss | = \sum\limits_{j=0}^3 s_j.
\end{equation}
For convenience, we adopt the convention of using bold letters, such as $\bss$ and $\bk$, to denote points represented in homogeneous coordinates.
The transformation between $\bs{x}\in \RR^3$ and $\bss\in \RR^4_H$ is then defined by \cite[(3.1)]{LiXu2008},
\begin{equation}\label{trans_homo}
\begin{cases}
x_1 =  s_2+s_3,\\
x_2 = s_3+s_1,\\
x_3 = s_1 + s_2,
\end{cases}
\end{equation}
and $s_0=-s_1-s_2-s_3.$

We further define the function on $\Omega_H=\left\{ \bss \in \RR^4_H: -1\le s_i-s_j\le 1, 0\le i,j\le 3\right\}$ that
\begin{equation}
\begin{aligned}
&\phi_{\bk}(\bss):=e^{\frac{\pi \ii}{2} \bk\cdot \bss},\quad \bk \in \Lambda_0,\\ 
&\Lambda_0:=\left\{ \bk\in \RR^4_H \cap \ZZ^4:
k_0\equiv k_1\equiv k_2 \equiv k_3 \,\,({\rm{mod}}\, 4),
k_0<k_1<k_2<k_3 \right\}.
\end{aligned}
\end{equation}
Here $\ii$ is the imaginary number satisfying $\ii^2=-1.$
Let $\mathcal{G}$ be the permutation group of four elements. For $\bk\in \RR^4_{H}$ and $\sigma\in \mathcal{G},$ 
the permutation of the elements in $\bk$ by $\sigma$  is denoted by $\bk \sigma.$ The generalized sine functions are then defined as \cite[Definition 4.2]{LiXu2008}
\begin{equation}\label{g_sine}
{\rm TS}_{\bk}(\bss):= \frac{1}{24} \sum\limits_{\sigma\in \mathcal{G}} (-1)^{|\sigma|} \phi_{\bk \sigma}(\bss),\quad \bk\in \Lambda_0,
\end{equation}
where $|\sigma|$ represents the number of inversions in $\sigma.$
Thus, we arrive at the following lemma.

\begin{lemma}\label{fundamental_eig}
The generalized sine functions ${\rm TS}_{\mathbf{k}}(\mathbf{s}), \bk\in\Lambda_0$ are the eigenfunctions of the Laplacian on $\TT_F$ subject to the homogeneous Dirichlet boundary condition: 
\begin{equation}
\begin{cases}
-\Delta {\rm TS}_{\bk}(\bss) = \mu_{\bk} {\rm TS}_{\bk}(\bss),\quad &{\rm{in}}\,\TT_F,\\
{\rm TS}_{\bk}(\bss)=0,\quad &{\rm{on}}\,\partial\TT_F,
\end{cases}
\end{equation}
where 
\begin{equation}
\mu_{\bk}=\frac{\pi^2|\bk|^2}{4},\quad |\bk|^2=\sum\limits_{j=0}^3 k_j^2.
\end{equation}
\end{lemma}
\begin{proof}
Due to the symmetry of ${\rm TS}_{\bk}(\bss)$, it vanishes on $\partial\TT_F.$
From the transformation \eqref{trans_homo}, we have
\begin{equation*}
\begin{aligned}
&\partial_{s_1}-\partial_{s_0} = \partial_{x_2}+\partial_{x_3},\quad
\partial_{s_2}-\partial_{s_0} = \partial_{x_3}+\partial_{x_1},\quad
\partial_{s_3}-\partial_{s_0} = \partial_{x_1}+\partial_{x_2},\\
&\partial_{s_1}-\partial_{s_2} = \partial_{x_2}-\partial_{x_1},\quad
\partial_{s_2}-\partial_{s_3} = \partial_{x_3}-\partial_{x_2},\quad
\partial_{s_3}-\partial_{s_1} = \partial_{x_1}-\partial_{x_3}.
\end{aligned}
\end{equation*}
One easily obtains an equivalent expression of the Laplacian operator in homogenous coordinates that
\begin{equation}\label{laplace_homo}
\Delta = \frac{1}{4} \sum\limits_{1\le i< m\le 3} \left( 
\left(\partial_{x_i}+\partial_{x_m}\right)^2 +
\left(\partial_{x_i}-\partial_{x_m}\right)^2
 \right)=
 \frac{1}{4} \sum\limits_{0\le j< n\le 3}  \left( \partial_{s_j} - \partial_{s_n}\right)^2.
\end{equation}
Applying \eqref{laplace_homo} on $\phi_{\bk}$ yields
\begin{equation*}
\begin{aligned}
-\Delta \phi_{\bk}(\bss) &= -\frac{1}{4} \sum\limits_{0\le j< n\le 3}  \left( \partial_{s_j} - \partial_{s_n}\right)^2 \phi_{\bk}(\bss)= \frac{\pi^2}{16} \sum\limits_{0\le j< n\le 3}  \left( k_j - k_n\right)^2 \phi_{\bk}(\bss)\\
&= \frac{\pi^2}{32} \sum\limits_{0\le j, n\le 3 \atop j\neq n}  \left( k_j - k_n\right)^2 \phi_{\bk}(\bss)\\
&= \frac{\pi^2}{32} \left( 4\sum\limits_{j=0}^3 k_j^2 + 4\sum\limits_{n=0}^3 k_n^2 -2 \left( \sum\limits_{j=0}^3 k_j \right) \left( \sum\limits_{n=0}^3 k_n \right)\right)  \phi_{\bk}(\bss)\\
&= \frac{\pi^2}{4} \sum\limits_{j=0}^3 k_j^2  \phi_{\bk}(\bss)= \frac{\pi^2}{4} 
|\bk|^2 \phi_{\bk}(\bss).
\end{aligned}
\end{equation*}
Therefore, by the definition of generalized sine functions \eqref{g_sine}, it holds that
\begin{equation*}
\begin{aligned}
-\Delta {\rm TS}_{\bk}(\bss) &=  \frac{1}{24} \sum\limits_{\sigma\in \mathcal{G}} (-1)^{|\sigma|+1}  \Delta \phi_{\bk \sigma}(\bss)= \frac{\pi^2}{4} \frac{1}{24}  \sum\limits_{\sigma\in \mathcal{G}} (-1)^{|\sigma|}  |\bk\sigma|^2 \phi_{\bk \sigma}(\bss)\\
&=\frac{\pi^2  |\bk|^2  }{4} \frac{1}{24}  \sum\limits_{\sigma\in \mathcal{G}} (-1)^{|\sigma|} \phi_{\bk \sigma}(\bss)= \frac{\pi^2  |\bk|^2  }{4}  {\rm TS}_{\bk}(\bss).
\end{aligned}
\end{equation*}
This completes the proof.
\end{proof}

\end{appendix}

\section*{Acknowledgements}
The research of the second author is supported in part by the National Natural Science Foundation of China grants  NSFC 11871455 and NSFC 11971016. The research of the third author is supported in part by the National Natural Science Foundation of China  grants NSFC 11871092 and NSAF U1930402.

\bibliographystyle{plain}%

\bibliography{reference_koorn}

\end{document}